\documentclass[11pt, a4paper, reqno]{amsart}

\usepackage{microtype}

\usepackage[utf8]{inputenc}

\usepackage{fancyhdr}

\fancyheadoffset[L,R]{0pt}

\usepackage{amsmath}
\numberwithin{equation}{section}
\usepackage[letterpaper,hmargin=1.0in,vmargin=1.0in]{geometry}
\usepackage{amsmath,amsthm,amsfonts,amssymb,mathrsfs,stmaryrd,bigints}
\usepackage{dsfont}
\usepackage{mathtools}
\usepackage[usenames]{color}
\usepackage{MnSymbol}
\usepackage[colorlinks=true,linkcolor=blue]{hyperref}

\usepackage{tikz}
\usetikzlibrary{decorations.pathreplacing,calligraphy}
\usepackage{comment}
\usepackage{float}
\usepackage{epstopdf}
\usepackage{amssymb}
\usepackage{setspace}
\usepackage{enumerate}
\usepackage{bigstrut}
\usepackage{multirow}
\usepackage{mathtools,xparse}
\usepackage{mathrsfs}
\usepackage{hyperref}

\usepackage [autostyle, english = american]{csquotes}
\newtheorem*{lemma*}{Lemma}
\newtheorem{theorem}{Theorem}[section]
\newtheorem{corollary}[theorem]{Corollary}
\newtheorem{lemma}[theorem]{Lemma}
\newtheorem{proposition}[theorem]{Proposition}
\theoremstyle{definition}

\newtheorem{remark}[theorem]{Remark}

\allowdisplaybreaks
\usepackage{chngcntr}

\definecolor{darkmagenta}{RGB}{160, 10, 160}

\newcommand{\E}{\mathbb{E}}

\newcommand{\ent}{h}
\newcommand{\Z}{\mathbb{Z}}

	\renewcommand{\P}{\mathbb{P}}

\DeclareMathOperator{\diam}{diam}

\renewcommand{\setminus}{\backslash}

\newtheorem{maintheorem}{Theorem}[section]

\def\E{\mathbb E}

\def\Ld L
\def\L{\mathbb L}

\def\P{\mathbb P}

\def\Z{\mathbb Z}

\def\been{\begin{enumerate}}
\def\bee{\begin{example}}
\def\beit{\begin{itemize}}
\def\bel{\begin{lemma}}
\def\bepr{\begin{proposition}}
\def\bep{\begin{proof}}
\def\bet{\begin{theorem}}
\def\bec{\begin{corollary}}
\def\enc{\end{corollary}}

\def\de\de 
\def\diam{\ms{diam}}

\def\enen{\end{enumerate}}
\def\ene{\end{example}}
\def\enit{\end{itemize}}
\def\enl{\end{lemma}}
\def\enpr{\end{proposition}}
\def\enp{\end{proof}}
\def\ent{\end{theorem}}

\def\ms{\mathsf}

\def\de{\delta}

\def\to \infty{\uparrow\infty}

\def\x{x}

\def\y{y}
\def\bef{\begin{figure}[!h]}
\def\enf{\end{figure}}

\newtheorem{theoremA}{Theorem}

\usepackage{kotex}

\begin{document}

\title[Disordered Wetting for the 2D SOS Model]
{Wetting Transition of the Two-Dimensional Solid-on-Solid Model with Quenched Disorder}
\date{\today}

\author{Seokun Choi}
\begin{abstract}
We study the wetting transition for a two-dimensional Solid-On-Solid (SOS) interface above a hard wall in the presence of quenched disorder.
The interface is represented by a nonnegative integer-valued height function $\phi$, with Hamiltonian
\begin{align}
    \mathcal H^\omega(\phi)
    =
    \beta \sum_{x\sim y} |\phi(x)-\phi(y)|
    -
    \sum_x
    \bigl(h+\alpha\omega_x-\lambda(\alpha)\bigr)
    \mathbf 1_{\{\phi(x)=0\}},
\end{align}
where $(\omega_x)_{x\in\mathbb Z^2}$ is an i.i.d.\ centered field, $\alpha\geq0$ is the disorder strength, and $\lambda(\alpha)=\log\mathbb E[e^{\alpha\omega_0}]$.
The normalization is chosen so that the annealed model coincides with the corresponding homogeneous wetting model.
As the wall attraction $h$ increases, the interface undergoes a wetting transition from a delocalized phase, in which contacts with the wall have vanishing density, to a localized phase with a positive density of contacts.
The critical point $h_c$ is the threshold value of $h$ separating these two regimes.

For sufficiently large $\beta$, the homogeneous wetting point and the sharp near-critical behavior of the homogeneous free energy were identified in \cite{lacoin2018wetting}.
We prove that, for every fixed $\alpha\geq0$, the quenched critical point coincides with the homogeneous wetting point:
\begin{align}
    h_c(\beta,\alpha)
    =
    h_{w}(\beta)
    =
    -\log\bigl(1-e^{-4\beta}\bigr).
\end{align}
Moreover, writing $u=h-h_{w}(\beta)$, we show that the quenched excess free energy satisfies
\begin{align}
    \overline{F}_\beta(\alpha,u)
    =
    \overline{F}_{\mathrm{homo}}(\beta,u)
    +
    o(u^3),
    \qquad
    u\rightarrow 0+,
\end{align}
where $\overline{F}_{\mathrm{homo}}(\beta,u)$ denotes the leading homogeneous wetting asymptotic and satisfies $\overline{F}_{\mathrm{homo}}(\beta,u)\asymp u^3$.

Thus, at the level of the critical point and the leading near-critical free-energy asymptotics, the quenched model has the same behavior as the homogeneous SOS wetting model.
This contrasts with the disordered SOS pinning model, in which quenched disorder leaves the critical point unchanged but modifies the leading critical behavior of the free energy.
\end{abstract}

\address[Seokun Choi]{Department of Mathematical Sciences, KAIST, South Korea}
\email{seokunchoi@kaist.ac.kr}
\subjclass[2010]{}
\keywords{}
\maketitle
\setcounter{secnumdepth}{2}
\setcounter{tocdepth}{2}
\tableofcontents
\section{Introduction}\label{chap:1}

\subsection{Motivation for the SOS model}\label{subsec:1.1}

The Solid-On-Solid (SOS) model is a classical lattice model for random surfaces, going back to the work of Temperley \cite{temperley1952statistical}.
At low temperature, it is commonly viewed as an effective description of an interface separating two coexisting phases, with the low-temperature three-dimensional Ising model under mixed boundary conditions providing a prototypical example; see, for instance, \cite{caputo2016scaling,lacoin2018wetting}.
In the SOS approximation, the interface is represented as the graph of an integer-valued height function, thereby suppressing overhangs and microscopic bulk fluctuations.
This simplification makes the model substantially more tractable while retaining a rich collection of interface phenomena, including rigidity, roughening, entropic repulsion, wetting, and layering.

Of particular relevance to the present work is the behavior of the SOS interface in the presence of a hard wall.
Even in the low-temperature regime, where the unconstrained interface is rigid, the constraint that the surface remain above the wall induces an entropic repulsion effect: the interface moves away from the wall in order to preserve its freedom to fluctuate.
This phenomenon was studied rigorously in the classical work \cite{bricmont1986random} and has since been analyzed in considerable detail; see, e.g., \cite{caputo2016scaling,velenik2006localization}.
In the two-dimensional SOS model, this competition between energetic rigidity and entropic effects leads to a particularly rich geometric structure, which can be studied through level lines and contour representations.

A further competition arises when the wall exerts an attractive interaction on the interface.
The attraction favors contacts with the wall, while entropic repulsion favors configurations that remain away from it.
Their competition gives rise to a wetting transition between a localized regime, in which a positive density of contacts with the wall is present, and a delocalized regime, in which the interface is repelled from the wall.
The existence of such a transition for the SOS model goes back to Chalker \cite{chalker1982pinning}, and its low-temperature structure has subsequently been studied in increasing detail \cite{alexander2011layering,lacoin2018wetting}.
This makes the SOS model a natural setting in which to investigate how microscopic modifications of the wall interaction affect the macroscopic wetting behavior.

\subsection{Disorder and wetting phenomena}\label{subsec:1.2}

Disorder enters statistical-mechanical modeling in a natural way, since interacting constituents need not be identical and substrates or surrounding media are rarely perfectly homogeneous.
Homogeneous models are obtained by neglecting such heterogeneities, a simplification that is often both natural and mathematically useful, especially at the level of idealized models \cite{giacomin2011disorder}.
This raises the basic question of whether the phenomena identified in the homogeneous system remain stable when quenched randomness is introduced.
In particular, even a weak local random perturbation may have a non-negligible effect near a phase transition, where the system is sensitive to fluctuations over large scales.

For localization and wetting models, one is therefore led to compare homogeneous and disordered systems at two distinct levels: the location of the transition and the behavior close to criticality.
These are conceptually different questions.
This distinction is closely related to the notions of disorder relevance and irrelevance originating from the work of Harris \cite{harris1974effect}; see also \cite{giacomin2011disorder}.
Roughly speaking, disorder is called irrelevant at the level of critical behavior when the leading critical behavior of the disordered system agrees with that of the corresponding pure system, whereas relevant disorder modifies this behavior.
In the present model, we show that the homogeneous and disordered wetting free energies have the same leading near-critical order, namely of order $u^3$.
For sufficiently large $\beta$ and every fixed $\alpha\ge0$, we prove that the quenched and homogeneous excess free energies differ by $o(u^3)$ as $u=h-h_w(\beta)\rightarrow 0+$, while both are of order $u^3$.
Thus the leading variational asymptotic, and not merely its order, is preserved.
In this paper, disorder irrelevance refers to this agreement of the leading free-energy asymptotics, without asserting that finer thermodynamic or geometric properties are unchanged.
Thus, in this sense, the quenched disorder is irrelevant at the level of the leading critical behavior.

Random-interface models provide examples in which disorder has different effects.
For the lattice Gaussian free field with a disordered pinning potential in dimension $d\geq3$, the quenched critical point coincides with the annealed one, while the critical behavior of the free energy is modified by the disorder \cite{giacomin2017pinning}.
The analysis combines fractional-moment estimates and a tilt of the disorder for the upper bound with finite-volume and coarse-graining arguments, together with replica coupling and Gaussian interpolation, for the sharp lower bound.
In a related wetting problem for the lattice free field in dimension three or larger, quenched disorder has also been studied in the presence of a wall, and the resulting behavior differs from that of the pinning problem: the transition remains of infinite order, as in the homogeneous model \cite{giacomin2018disorder}.
These examples illustrate that the effect of disorder depends on both the underlying interface and the way in which it interacts with the substrate.

Similar questions arise naturally for discrete SOS interfaces.
In the two-dimensional SOS model, the interface is represented by an integer-valued height function on $\mathbb Z^2$, with neighboring height differences penalized energetically.
At low temperature the interface is rigid, but the presence of a hard wall and an attractive interaction with the substrate creates a nontrivial competition between entropic repulsion and localization at the wall.
This makes the SOS model a particularly natural setting in which to study how a localization transition is affected by spatial disorder.
Both the homogeneous wetting problem and the disordered pinning problem for the SOS model have been studied sharply in the low-temperature regime.
These results provide the natural reference points for the present work and will be recalled more precisely in Subsection~\ref{subsec:1.3}.
Our question is whether the near-critical behavior of homogeneous SOS wetting persists when the interaction with the wall is made spatially disordered.

\subsection{Previous results and our contribution}\label{subsec:1.3}

The homogeneous wetting problem for the two-dimensional SOS model was studied sharply by Lacoin in \cite{lacoin2018wetting}.
In the low-temperature regime, the homogeneous wetting point, which we denote by $h_w(\beta)$, was identified explicitly as
\begin{align}
    h_w(\beta)
    =
     \log\left(
        \frac{e^{4\beta}}{e^{4\beta}-1}\right).
\end{align}
Moreover, the free energy near the wetting point was described sharply through a variational problem over the integer localization height of the interface; see \eqref{eq:homogeneous-wetting-variational}.
In particular, if $u>0$ denotes the distance to the wetting point, then the excess free energy is of order $u^3$ as $u\rightarrow 0+$.
The main result of the present work shows that this leading near-critical behavior is preserved under quenched disorder.
In this sense, the disorder is irrelevant at the level of the leading critical behavior.

Such a conclusion is not automatic in disordered interface models.
A closely related example is the disordered SOS pinning model studied in \cite{lacoin2021solid}, in which the interface is allowed to cross the defect plane and receives a site-dependent random reward upon visiting level zero.
With a centered normalization of the disorder, the quenched critical point coincides with the homogeneous one.
Nevertheless, the critical behavior of the free energy is changed by the disorder: while the homogeneous pinning free energy grows linearly above criticality, the disordered free energy is of quadratic order and is governed by a different variational asymptotic.
Thus, equality of the quenched and homogeneous critical points does not by itself imply that disorder is irrelevant for the critical behavior.

For the disordered wetting model considered here, we show that this leading-order picture remains unchanged.
For sufficiently large $\beta$ and every fixed disorder strength $\alpha\geq0$, we prove that the quenched critical point, denoted by $h_c(\beta,\alpha)$, coincides with the homogeneous wetting point:
\begin{align}
    h_c(\beta,\alpha)=h_w(\beta).
\end{align}
Moreover, the quenched excess free energy has the same leading near-critical asymptotics as in the homogeneous wetting model, up to an error of smaller order than $u^3$.
Thus, unlike in the disordered SOS pinning model discussed above, quenched disorder changes neither the location of the transition nor its leading free-energy behavior.
The precise definitions and the full statement of the result are given in Subsections \ref{subsec:2.2}, \ref{subsec:2.3}.

Our proof is partly inspired by the coarse-graining method developed for disordered SOS pinning in \cite{lacoin2021solid}, but the wetting interaction introduces an additional difficulty.
In the pinning model, the interaction is sitewise and therefore factorizes over coarse-grained regions once the SOS field has been conditionally decomposed.
In the wetting representation used here, by contrast, the interaction depends on the connected components of the contact set.
A single connected component may intersect several coarse-grained regions, and the interaction therefore does not inherit the same factorization.
A main technical ingredient of the present work is to overcome this obstruction while preserving the sharp local contributions needed for the near-critical asymptotics.

\subsection{Organization of the paper}
\label{subsec:1.4}

Chapter~\ref{chap:2} introduces the model and the thermodynamic quantities used throughout the paper.
We first recall the two-dimensional SOS model and the basic low-temperature properties that will be needed later.
We then define the disordered wetting model, its quenched free energy, and the corresponding critical point in Subsections~\ref{subsec:2.1}--\ref{subsec:2.2}.
The main result is stated in Subsection~\ref{subsec:2.3}: at sufficiently low temperature, the quenched critical point coincides with the homogeneous wetting point and the leading near-critical free-energy asymptotics agrees with that of the homogeneous model.
Subsection~\ref{subsec:2.4} gives an overview of the proof and explains the main difficulties caused by the connected-component dependence of the wetting interaction, while Subsection~\ref{subsec:2.5} discusses several questions left open by the present result.

Chapter~\ref{chap:3} collects the probabilistic and geometric tools needed for the lower-bound argument.
In Subsection~\ref{subsec:3.1}, we remove the hard-wall constraint by rewriting the wetting partition function in terms of an unconstrained SOS field and a wall penalty depending on the connected components of its non-positive level set.
After a change of boundary condition and reflection, this reduces the problem to the study of rare high-level excursions.
Subsection~\ref{subsec:3.2} recalls the contour representation of the low-temperature SOS model and introduces contour-restricted measures, which are the natural conditional laws arising from the coarse-graining procedure.
Finally, Subsection~\ref{subsec:3.3} collects the one-point, two-point, and multipoint peak estimates required later, together with the sharp asymptotics for isolated one-site and nearest-neighbor two-site high-level components.

Chapter~\ref{chap:4} performs the coarse-grained reduction of the quenched lower bound.
We begin in Subsection~\ref{subsec:4.1} by passing to a finite-volume high-level representation and identifying the localization window for the relevant boundary height $n$, characterized by $J^n\asymp u$.
In Subsection~\ref{subsec:4.2}, we restrict to configurations without long contours, introduce a mesoscopic grid, and condition on the external contours crossing this grid; this yields a decomposition into regions over which the underlying SOS field is conditionally independent.
Subsection~\ref{subsec:4.3} records the geometry of these regions and isolates deterministic cell cores, showing that cell-boundary strips, contour regions, and cross-region edges occupy only a lower-order portion of the system.
In Subsection~\ref{subsec:4.4}, we use this geometry to decompose the wall interaction: exact singleton and nearest-neighbor two-site components in the cell cores are retained sharply, whereas all remaining non-singleton contributions are controlled through rough occupied-edge penalties and a cross-region error.
Combining this pointwise decomposition with the conditional factorization in Subsection~\ref{subsec:4.5} reduces the global lower bound to three terms: sharp cell partition functions, rough contour-region contributions, and a cross-edge correction.

Chapter~\ref{chap:5} estimates the three terms produced by the reduction in Chapter~\ref{chap:4} and completes the proof of the sharp lower bound.
Subsection~\ref{subsec:5.1} first compares the conditioned SOS law deep inside a cell with the infinite-volume SOS measure, thereby recovering the sharp one-site and two-site peak probabilities with the coefficients appearing in the homogeneous variational formula.
These pointwise estimates are converted in Subsection~\ref{subsec:5.2} into estimates for the numbers of exact singleton and pair components, while the remaining occupied edges and the contact cutoff are shown to be negligible.
Subsection~\ref{subsec:5.3} then derives the sharp deterministic contribution of a bulk cell and proves, through a second-moment argument, that the centered disorder changes its quenched logarithm only at a lower order.
In Subsection~\ref{subsec:5.4}, the contour-region and cross-edge terms are controlled by rough contour-restricted estimates and are shown to be $o(N^2u^3)$.
Finally, Subsection~\ref{subsec:5.5} combines these bounds with the coarse-grained decomposition, lets the volume tend to infinity, and optimizes over the localization height to obtain the matching quenched lower bound and hence complete the proof of the main theorem.

Appendix~\ref{app:coarse-grained-geometry} contains the geometric estimates for the coarse-grained regions used in Subsection~\ref{subsec:4.3}, including the bounds on the deterministic cores, contour regions, and cross-edges.
Appendix~\ref{app:local-comparison} proves the local comparison used in Subsection~\ref{subsec:5.1}, showing that sufficiently deep inside a coarse-grained cell the conditioned SOS law differs from the infinite-volume law only by an exponentially small error.

\section{Model and preliminary properties}
\label{chap:2}

\subsection{The SOS model}\label{subsec:2.1}

In this subsection, we introduce the basic notation for the two-dimensional SOS model that will be used throughout the paper.
We also recall two standard properties of the low-temperature SOS measure that will be repeatedly used below: the existence of the thermodynamic free energy and the exponential convergence of finite-volume local statistics to their infinite-volume counterparts.

Throughout the paper, we work on the lattice $\mathbb Z^2$ equipped with the nearest-neighbor graph structure.
For integers $a\le b$, we write
\begin{align*}
    \llbracket a,b\rrbracket
    &:=
    [a,b]\cap\mathbb Z,
    \qquad
    \Lambda_N
    :=
    \llbracket 1,N\rrbracket^2.
\end{align*}
For a finite set $\Lambda\subset\mathbb Z^2$, its external boundary is defined by
\begin{align*}
    \partial\Lambda
    &:=
    \bigl\{
        x\in\mathbb Z^2\setminus\Lambda:
        \exists\,y\in\Lambda,\ x\sim y
    \bigr\}.
\end{align*}
We write
\begin{align*}
    \Omega_\Lambda
    &:=
    \mathbb Z^\Lambda,
    \qquad
    \Omega_\Lambda^+
    :=
    \mathbb Z_+^\Lambda,
    \qquad
    \mathbb Z_+
    :=
    \mathbb Z\cap[0,\infty).
\end{align*}
Given a boundary height $n\in\mathbb Z$ and a configuration $\phi\in\Omega_\Lambda$, the SOS Hamiltonian with boundary condition $n$ is
\begin{align}
    \mathcal H_\Lambda^n(\phi)
    &:=
    \frac12
    \sum_{\substack{x,y\in\Lambda\\x\sim y}}
    |\phi(x)-\phi(y)|
    +
    \sum_{\substack{x\in\Lambda,\ y\in\partial\Lambda\\x\sim y}}
    |\phi(x)-n|.
    \label{eq:SOS-Hamiltonian}
\end{align}
The factor $1/2$ is included so that each unordered nearest-neighbor edge inside $\Lambda$ is counted once.
For $\beta>0$, the SOS measure with boundary condition $n$ is
\begin{align}
    \P_{\Lambda,\beta}^n(\phi)
    &:=
    \frac{1}{\mathcal Z_{\Lambda,\beta}}
    \exp\bigl\{-\beta\mathcal H_\Lambda^n(\phi)\bigr\},
    \qquad
    \phi\in\Omega_\Lambda,
    \label{eq:SOS-measure}
\end{align}
where
\begin{align}
    \mathcal Z_{\Lambda,\beta}
    &:=
    \sum_{\phi\in\Omega_\Lambda}
    \exp\bigl\{-\beta\mathcal H_\Lambda^n(\phi)\bigr\}.
    \label{eq:SOS-partition}
\end{align}
By vertical translation invariance, $\mathcal Z_{\Lambda,\beta}$ does not depend on $n$.
When $n=0$, we omit the superscript and simply write
\begin{align*}
    \mathcal H_\Lambda,
    \qquad
    \P_{\Lambda,\beta},
    \qquad
    \E_{\Lambda,\beta}.
\end{align*}
When $\Lambda=\Lambda_N$, we write
\begin{align*}
    \mathcal Z_{N,\beta}
    &:=
    \mathcal Z_{\Lambda_N,\beta},
\end{align*}
and adopt the analogous convention for the corresponding probability measure and expectation.

The free energy of the unconstrained SOS model will serve as the reference thermodynamic quantity when we introduce the wetting free energy below.
We therefore first record the existence of its infinite-volume limit.
If $\Lambda^{(1)}$ and $\Lambda^{(2)}$ are disjoint finite subsets of $\mathbb Z^2$, then, for every $\phi\in\Omega_{\Lambda^{(1)}\cup\Lambda^{(2)}}$,
\begin{align}
    \mathcal H_{\Lambda^{(1)}\cup\Lambda^{(2)}}(\phi)
    &\le
    \mathcal H_{\Lambda^{(1)}}
    \bigl(\phi\bigr)
    +
    \mathcal H_{\Lambda^{(2)}}
    \bigl(\phi\bigr).
    \label{eq:SOS-Hamiltonian-superadditivity}
\end{align}
Consequently,
\begin{align}
    \mathcal Z_{\Lambda^{(1)}\cup\Lambda^{(2)},\beta}
    &\ge
    \mathcal Z_{\Lambda^{(1)},\beta}
    \mathcal Z_{\Lambda^{(2)},\beta}.
    \label{eq:SOS-partition-superadditivity}
\end{align}
By the standard superadditivity argument; see, for instance, \cite[Section 3.2]{friedli1945statistical}, the limit
\begin{align}
    F(\beta)
    &:=
    \lim_{N\rightarrow\infty}
    \frac{1}{N^2}
    \log\mathcal Z_{N,\beta}
    \label{eq:SOS-free-energy}
\end{align}
exists.
We refer to $F(\beta)$ as the free energy of the unconstrained SOS model.

The following exponential convergence result allows us to replace finite-volume expectations of local observables by their infinite-volume counterparts, with an error that decays exponentially in the distance to the boundary.
A function
\begin{align*}
    g:\mathbb Z^{\mathbb Z^2}\rightarrow\mathbb R
\end{align*}
is called local if there exists a finite set $A\subset\mathbb Z^2$ such that $g(\phi)$ depends only on $(\phi(x))_{x\in A}$.
We denote the minimal such set by $\operatorname{Supp}(g)$.
For two nonempty finite subsets $A,B\subset\mathbb Z^2$, we define
\begin{align}
    \operatorname{dist}(A,B)
    &:=
    \min_{\substack{x\in A\\y\in B}}
    |x-y|_1,
    \label{eq:set-distance}
\end{align}
where $|\cdot|_1$ denotes the $\ell^1$-distance on $\mathbb Z^2$.

\begin{theoremA}[From \cite{brandenberger1982decay}]
\label{thm:exponential-convergence-SOS}
There exist $\beta_0>0$ and $c>0$ such that, for every $\beta>\beta_0$, there exists an infinite-volume Gibbs measure $\P_\beta$ satisfying the following property.
For every local function
\begin{align*}
    g:\mathbb Z^{\mathbb Z^2}\rightarrow[0,1]
\end{align*}
with $\operatorname{Supp}(g)=A$, and every finite $\Lambda\subset\mathbb Z^2$ containing $A$, we have
\begin{align}
    \left|
        \E_{\Lambda,\beta}[g(\phi)]
        -
        \E_\beta[g(\phi)]
    \right|
    &\le
    |A|
    \exp\bigl\{-c\beta \operatorname{dist}(A,\partial\Lambda)\bigr\}.
    \label{eq:SOS-exponential-convergence}
\end{align}
\end{theoremA}

\subsection{The disordered wetting model and quenched free energy}\label{subsec:2.2}

In this subsection, we introduce the SOS wetting model with a disordered interaction at the wall.
Our aim is to understand how the presence of disorder affects the localization transition compared with the homogeneous wetting model.
We first define the corresponding Gibbs measure and then introduce the thermodynamic quantities that will be used to study this question.

Let $(\omega_x)_{x\in\mathbb Z^2}$ be an i.i.d.\ field, and let $\P$ and $\E$ denote its law and expectation, respectively.
We assume that
\begin{align}
    \E[\omega_0]
    &=
    0,
    \qquad
    \lambda(\alpha)
    :=
    \log \E\bigl[e^{\alpha\omega_0}\bigr]
    <
    \infty,
    \qquad
    \alpha\in\mathbb R.
    \label{eq:disorder-assumption}
\end{align}
In the following discussion of disorder-induced rounding, we assume that the disorder distribution is non-degenerate, for instance that $\operatorname{Var}(\omega_0)>0$.
For $\alpha\ge0$, we set
\begin{align}
    \eta_x^{(\alpha)}
    &:=
    \alpha\omega_x-\lambda(\alpha),
    \label{eq:normalized-disorder}
\end{align}
so that
\begin{align}
    \E\left[e^{\eta_x^{(\alpha)}}\right]
    &=
    1.
    \label{eq:normalized-disorder-mean}
\end{align}

Given $\beta>0$, $\alpha\ge0$, $h\in\mathbb R$, and a boundary height $n\in\mathbb Z_+$, we define the disordered wetting partition function by
\begin{align}
    \mathcal Z_{\Lambda,\beta}^{+,n,h,\alpha,\omega}
    &:=
    \sum_{\phi\in\Omega_\Lambda^+}
    \exp\left\{
        -\beta\mathcal H_\Lambda^n(\phi)
        +
        \sum_{x\in\Lambda}
        \bigl(h+\eta_x^{(\alpha)}\bigr) \mathbf 1_{\{\phi(x)=0\}}
    \right\}.
    \label{eq:disordered-wetting-partition}
\end{align}
The corresponding Gibbs measure is
\begin{align}
    \P_{\Lambda,\beta}^{+,n,h,\alpha,\omega}(\phi)
    &:=
    \frac{1}{\mathcal Z_{\Lambda,\beta}^{+,n,h,\alpha,\omega}}
    \exp\left\{
        -\beta\mathcal H_\Lambda^n(\phi)
        +
        \sum_{x\in\Lambda}
        \bigl(h+\eta_x^{(\alpha)}\bigr) \mathbf 1_{\{\phi(x)=0\}}
    \right\},
    \qquad
    \phi\in\Omega_\Lambda^+,
    \label{eq:disordered-wetting-measure}
\end{align}
and we write $\E_{\Lambda,\beta}^{+,n,h,\alpha,\omega}$ for the corresponding expectation.
The term $-\lambda(\alpha)$ is a normalization which amounts only to a shift of the parameter $h$, but it is convenient when considering the annealed model, as in \cite{lacoin2021solid}.
Indeed, by independence of the disorder and \eqref{eq:normalized-disorder-mean}, for every finite $\Lambda$,
\begin{align}
    \E\left[
        \mathcal Z_{\Lambda,\beta}^{+,n,h,\alpha,\omega}
    \right]
    &=
    \mathcal Z_{\Lambda,\beta}^{+,n,h},
    \label{eq:annealed-is-homogeneous}
\end{align}
where the right-hand side denotes the homogeneous wetting partition function obtained by setting $\alpha=0$.
Thus, with this normalization, the annealed model coincides with the homogeneous wetting model with the same parameter $h$.
We simplify the notation as follows.
When $\alpha=0$, we write
\begin{align*}
    \mathcal Z_{\Lambda,\beta}^{+,n,h},
    \qquad
    \P_{\Lambda,\beta}^{+,n,h},
    \qquad
    \E_{\Lambda,\beta}^{+,n,h}.
\end{align*}
When $n=0$, we omit $n$ from the notation.
When $\Lambda=\Lambda_N$, we replace the subscript $\Lambda$ by $N$.

Our aim is to understand whether the introduction of disorder changes the wetting transition and the near-critical behavior of the free energy.
We first introduce the finite-volume free-energy density associated with the disordered wetting model.
For $n\in\mathbb Z_+$, set
\begin{align}
    F_N^{n,\omega}(\beta,\alpha,h)
    &:=
    \frac{1}{N^2}
    \log
    \mathcal Z_{N,\beta}^{+,n,h,\alpha,\omega}.
    \label{eq:finite-volume-wetting-free-energy}
\end{align}
Since the unconstrained SOS model already contributes the bulk free energy $F(\beta)$, the effect of the wall interaction and the disorder is described by the excess free energy.
We therefore define
\begin{align}
    \overline{F}_N^{n,\omega}(\beta,\alpha,h)
    &:=
    F_N^{n,\omega}(\beta,\alpha,h)
    -
    \frac{1}{N^2}
    \log\mathcal Z_{N,\beta}
    \notag\\
    &=
    \frac{1}{N^2}
    \log
    \frac{
        \mathcal Z_{N,\beta}^{+,n,h,\alpha,\omega}
    }{
        \mathcal Z_{N,\beta}
    }.
    \label{eq:finite-volume-quenched-excess-free-energy}
\end{align}

The basic thermodynamic properties of the homogeneous wetting free energy, including the existence of the thermodynamic limit, monotonicity and convexity in the interaction parameter, and the identification of its derivative with the asymptotic contact fraction, are established in \cite[Section~2.2]{lacoin2018wetting}.
We also use the fact that changing a fixed boundary height affects the partition function only through a boundary-order term, as in \cite[Lemma~5.2]{lacoin2018wetting}.
To pass to the quenched setting, we combine these wetting arguments with the concentration and convexity arguments used for disordered SOS pinning in \cite[Proposition~2.2 and Appendix~A.1]{lacoin2021solid}.
In particular, the almost-sure and $L^1$ convergence of the random free energy, its deterministic limit, and the quenched--annealed comparison below follow by the same arguments.

\begin{proposition}[Basic properties of the quenched wetting free energy]
\label{prop:free-energy-basic-properties}
For every $\beta>0$, $\alpha\ge0$, $h\in\mathbb R$, and $n\in\mathbb Z_+$, the limit
\begin{align}
    F(\beta,\alpha,h)
    &:=
    \lim_{N\rightarrow\infty}
    F_N^{n,\omega}(\beta,\alpha,h)
    \label{eq:quenched-wetting-free-energy}
\end{align}
exists $\P$-almost surely and in $L^1(\P)$.
The limit is deterministic and is independent of the fixed boundary height $n$.

Consequently, the quenched excess free energy
\begin{align}
    \overline{F}(\beta,\alpha,h)
    &:=
    F(\beta,\alpha,h)-F(\beta)
    \notag\\
    &=
    \lim_{N\rightarrow\infty}
    \overline{F}_N^{n,\omega}(\beta,\alpha,h)
    \label{eq:quenched-excess-free-energy}
\end{align}
is well defined.

Moreover, the map
\begin{align*}
    h \longmapsto F(\beta,\alpha,h)
\end{align*}
is convex and non-decreasing.
At every point $h$ at which it is differentiable,
\begin{align}
    \partial_h F(\beta,\alpha,h)
    &=
    \lim_{N\rightarrow\infty}
    \frac{1}{N^2}
    \E\left[
        \E_{N,\beta}^{+,n,h,\alpha,\omega}
        \left[
            \sum_{x\in\Lambda_N}
            \mathbf 1_{\{\phi(x)=0\}}
        \right]
    \right].
    \label{eq:contact-fraction-derivative}
\end{align}

In the homogeneous case, we write
\begin{align}
    F(\beta,h)
    &:=
    F(\beta,0,h),
    \qquad
    \overline{F}(\beta,h)
    :=
    \overline{F}(\beta,0,h).
    \label{eq:homogeneous-free-energy}
\end{align}
Then, for every $\alpha\ge0$,
\begin{align}
    F\bigl(\beta,h-\lambda(\alpha)\bigr)
    &\le
    F(\beta,\alpha,h)
    \le
    F(\beta,h),
    \label{eq:quenched-annealed-comparison}
\end{align}
and
\begin{align}
    F(\beta,\alpha,h)
    &\ge
    F(\beta).
    \label{eq:quenched-free-energy-lower-bound}
\end{align}
Equivalently,
\begin{align}
    \overline{F}\bigl(\beta,h-\lambda(\alpha)\bigr)
    &\le
    \overline{F}(\beta,\alpha,h)
    \le
    \overline{F}(\beta,h),
    \label{eq:quenched-excess-comparison}\\
    \overline{F}(\beta,\alpha,h)
    &\ge
    0.
    \label{eq:quenched-excess-free-energy-nonnegative}
\end{align}
\end{proposition}

\begin{proof}[Proof of Proposition ~\ref{prop:free-energy-basic-properties}]
The thermodynamic properties concerning monotonicity, convexity, and the contact-fraction identity follow from the homogeneous wetting argument in \cite[Section~2.2]{lacoin2018wetting}, while the independence of the limit from the fixed boundary height follows as in \cite[Lemma~5.2]{lacoin2018wetting}.

The almost-sure and $L^1(\P)$ convergence in \eqref{eq:quenched-wetting-free-energy}, the determinism of the limit, and the comparison bounds \eqref{eq:quenched-annealed-comparison}--\eqref{eq:quenched-free-energy-lower-bound} follow from the arguments in \cite[Proposition~2.2 and Appendix~A.1]{lacoin2021solid}.
The only modification is the restriction to $\Omega_\Lambda^+$, which is preserved under the spatial decompositions used there and does not affect the boundary comparison, concentration, convexity, or Jensen arguments.
For \eqref{eq:quenched-free-energy-lower-bound}, the same no-contact restriction applies, now within the hard-wall constraint.

Finally, \eqref{eq:quenched-excess-free-energy}, \eqref{eq:quenched-excess-comparison}, and \eqref{eq:quenched-excess-free-energy-nonnegative} follow directly from the definition of the excess free energy.
\end{proof}

The identity \eqref{eq:contact-fraction-derivative} identifies the derivative of the free energy with the asymptotic density of contacts with the wall.
The localization transition is therefore characterized by the positivity of the excess free energy.
We define the quenched critical point by
\begin{align}
    h_c(\beta,\alpha)
    &:=
    \inf
    \bigl\{
        h\in\mathbb R:
        \overline{F}(\beta,\alpha,h)>0
    \bigr\}.
    \label{eq:quenched-critical-point}
\end{align}
By monotonicity and non-negativity of $\overline{F}(\beta,\alpha,\cdot)$, equivalently,
\begin{align}
    h_c(\beta,\alpha)
    &=
    \sup
    \bigl\{
        h\in\mathbb R:
        \overline{F}(\beta,\alpha,h)=0
    \bigr\}.
    \label{eq:quenched-critical-point-equivalent}
\end{align}
For the homogeneous model, we define the wetting point by
\begin{align}
    h_w(\beta)
    &:=
    \inf
    \bigl\{
        h\in\mathbb R:
        \overline{F}(\beta,h)>0
    \bigr\}.
    \label{eq:homogeneous-wetting-point}
\end{align}
The comparison \eqref{eq:quenched-excess-comparison} immediately yields
\begin{align}
    h_w(\beta)
    &\le
    h_c(\beta,\alpha)
    \le
    h_w(\beta)+\lambda(\alpha).
    \label{eq:basic-critical-bounds}
\end{align}
Our goal is to sharpen this comparison in the low-temperature regime and determine the near-critical behavior of the quenched excess free energy $\overline{F}(\beta,\alpha,h)$.

\subsection{Main result}\label{subsec:2.3}

We now state the main result of the paper.
In the low-temperature regime, the homogeneous wetting point and the sharp near-critical behavior of the homogeneous wetting free energy were identified in \cite[Theorem~2.1]{lacoin2018wetting}.
In particular, for $\beta$ sufficiently large,
\begin{align}
    h_w(\beta)
    &=
    \log\left(
        \frac{e^{4\beta}}{e^{4\beta}-1}
    \right).
    \label{eq:homogeneous-wetting-point-explicit}
\end{align}
It will be convenient to write
\begin{align}
    J(\beta)
    &:=
    e^{-2\beta},
    \label{eq:J-beta}
\end{align}
and, when no confusion can arise, simply $J$ for $J(\beta)$.
Then
\begin{align}
    h_w(\beta)
    &=
    -\log\bigl(1-J^2\bigr).
    \label{eq:wetting-point-J}
\end{align}

To study the free energy near the wetting point, we introduce the shifted parameter
\begin{align}
    u
    &:=
    h-h_w(\beta).
    \label{eq:u-shift}
\end{align}
Accordingly, we write
\begin{align}
    \overline F_\beta(\alpha,u)
    &:=
    \overline F\bigl(\beta,\alpha,h_w(\beta)+u\bigr).
    \label{eq:shifted-quenched-free-energy}
\end{align}
In the homogeneous case, we simply write
\begin{align}
    \overline F_\beta(u)
    &:=
    \overline F_\beta(0,u)
    =
    \overline F\bigl(\beta,h_w(\beta)+u\bigr).
    \label{eq:shifted-homogeneous-free-energy}
\end{align}

To describe the sharp asymptotic, let
\begin{align}
    a_1(\beta)
    &:=
    \lim_{n\rightarrow\infty}
    J(\beta)^{-2n}
    \P_\beta\bigl(\phi(0)\ge n\bigr),
    \label{eq:a1-beta}\\
    a_2(\beta)
    &:=
    \lim_{n\rightarrow\infty}
    J(\beta)^{-3n}
    \P_\beta
    \left(
        \min\{\phi(0),\phi((1,0))\}\ge n
    \right).
    \label{eq:a2-beta}
\end{align}
The existence of the one-point limit $a_1(\beta)$ was established in \cite[Lemma~2.4]{caputo2016scaling}.
The existence of the two-point limit, together with the positivity and finiteness of these constants in the low-temperature regime, follows from \cite[Proposition~4.6]{lacoin2018wetting}.

Recall from \cite[Theorem~2.1]{lacoin2018wetting} the variational function governing the near-critical asymptotics of the homogeneous wetting free energy.
In our notation, For $u>0$, it is given by
\begin{align}
    \overline F_{\mathrm{homo}}(\beta,u)
    &:=
    \max_{n\in\mathbb Z_+}
    \left\{
        a_1(\beta)J^{2n}u
        -
        \frac{
            2a_2(\beta)
            \bigl(J^3-J^4\bigr)
        }{
            1-J^4
        }
        J^{3n}
    \right\}.
    \label{eq:homogeneous-wetting-variational}
\end{align}
By \cite[Theorem~2.1]{lacoin2018wetting}, the homogeneous excess free energy satisfies
\begin{align}
    \overline F_\beta(u)
    &=
    \overline F_{\mathrm{homo}}(\beta,u)
    +
    o(u^3),
    \qquad
    u\rightarrow0+.
    \label{eq:homogeneous-sharp-asymptotic}
\end{align}
In particular,
\begin{align}
    \overline F_{\mathrm{homo}}(\beta,u)
    &\asymp
    u^3,
    \qquad
    u\rightarrow0+.
    \label{eq:Fwet-order}
\end{align}
Here and below, for two positive functions $f$ and $g$, we write
\begin{align*}
    f(u)
    &\asymp
    g(u),
    \qquad
    u\rightarrow0+,
\end{align*}
if there exist constants $0<c<C<\infty$ and $u_0>0$ such that
\begin{align*}
    c\,g(u)
    &\le
    f(u)
    \le
    C\,g(u),
    \qquad
    u\in(0,u_0].
\end{align*}
Unless stated otherwise, constants may depend on the fixed parameters of the model and on the law of the disorder.
A subscript indicates the parameters on which a constant is allowed to depend; in particular, $C_\beta$ is independent of $\alpha$, while $C_{\alpha,\beta}$ may depend on the fixed disorder strength $\alpha$.
None of these constants depends on the volume parameter $N$ unless explicitly stated otherwise.

\begin{maintheorem}
\label{thm:main-result}
There exists $\beta_0<\infty$ such that, for every $\beta>\beta_0$ and every fixed $\alpha\ge0$, the quenched wetting critical point coincides with the homogeneous wetting critical point:
\begin{align}
    h_c(\beta,\alpha)
    &=
    h_w(\beta)
    =
    -\log\bigl(1-J(\beta)^2\bigr).
    \label{eq:main-critical-point}
\end{align}
Moreover,
\begin{align}
    \overline F_\beta(\alpha,u)
    &=
    \overline F_{\mathrm{homo}}(\beta,u)
    +
    o(u^3),
    \qquad
    u\rightarrow0+.
    \label{eq:main-free-energy-asymptotic}
\end{align}
Consequently,
\begin{align}
    \overline F_\beta(\alpha,u)
    &=
    \overline F_\beta(u)
    +
    o(u^3),
    \qquad
    u\rightarrow0+.
    \label{eq:quenched-homogeneous-asymptotic}
\end{align}
\end{maintheorem}

The upper bound required for \eqref{eq:main-free-energy-asymptotic} follows from the annealed comparison and the homogeneous asymptotic \eqref{eq:homogeneous-sharp-asymptotic}.
Indeed,
\begin{align}
    \overline F_\beta(\alpha,u)
    &\le
    \overline F_\beta(u)
    =
    \overline F_{\mathrm{homo}}(\beta,u)
    +
    o(u^3).
\end{align}
Thus the main task is to establish the matching quenched lower bound in Proposition~\ref{prop:sharp-lower-bound}; see \eqref{eq:sharp-quenched-lower-bound}.
Similarly, the annealed comparison gives
\begin{align}
    h_c(\beta,\alpha)
    &\ge
    h_w(\beta),
\end{align}
whereas \eqref{eq:sharp-quenched-lower-bound}, together with \eqref{eq:Fwet-order}, implies
\begin{align}
    \overline F_\beta(\alpha,u)
    &>
    0
\end{align}
for all sufficiently small $u>0$.
Hence
\begin{align}
    h_c(\beta,\alpha)
    &\le
    h_w(\beta),
\end{align}
which yields \eqref{eq:main-critical-point}.

\begin{remark}
The independence of $h_c(\beta,\alpha)$ from $\alpha$ is a consequence of the centered normalization of the disorder.
If the interaction is instead written in the uncentered form
\begin{align*}
    h+\alpha\omega_x,
\end{align*}
then the corresponding critical point is
\begin{align}
    h_w(\beta)-\lambda(\alpha),
\end{align}
and therefore depends on $\alpha$.
\end{remark}

\subsection{Proof strategy and technical difficulties}
\label{subsec:2.4}

The upper bound in Theorem~\ref{thm:main-result} follows directly from the annealed comparison and the sharp homogeneous asymptotic \eqref{eq:homogeneous-sharp-asymptotic}.
Thus, the main task is to establish the matching quenched lower bound, stated in Proposition~\ref{prop:sharp-lower-bound}.

We first remove the hard-wall constraint using the alternative representation developed in Subsection~\ref{subsec:3.1}; see Lemma~\ref{lemma:alternative-representation-disordered}.
The resulting interaction is expressed in terms of the non-positive level set $\phi^{-1}(\mathbb Z_-)$ and the shifted wall penalty $\overline H$, whose basic properties are summarized in Lemma~\ref{lemma:Hbar-basic-properties}.
A change of boundary condition followed by reflection then turns the problem into one concerning high-level excursions of the zero-boundary SOS field; see Lemma~\ref{lemma:change-boundary-disordered}.

The homogeneous variational formula identifies the relevant localization heights.
More precisely, Lemma~\ref{lem:localization-window} in Subsection~\ref{subsec:4.1} shows that the maximizing heights satisfy
\begin{align}
    J^n \asymp u,
\end{align}
or equivalently
\begin{align}
    n
    =
    \frac{1}{2\beta}\log\frac1u
    +
    O_\beta(1);
\end{align}
see also \eqref{eq:localization-height-scale}.

For such a height $n$, we restrict the SOS field to configurations without long contours.
By Lemma~\ref{lem:good-event} and \eqref{eq:restriction-to-good-event}, the cost of this restriction is $o(u^3)$ at the level of the free energy.
We then introduce the mesoscopic scale
\begin{align}
    M=\lfloor u^{-\chi}\rfloor,
    \qquad
    0<\chi<\frac1{50},
\end{align}
as in \eqref{eq:coarse-graining-scale}, and condition on the external contours crossing the coarse-graining grid; see Subsection~\ref{subsec:4.2}.
This decomposes the box into regions over which the underlying SOS field is conditionally independent, as expressed in \eqref{eq:conditional-factorization}.

At this point the main technical difficulty appears.
In the disordered pinning model of \cite{lacoin2021solid}, the interaction is sitewise.
Hence, if the coarse-graining produces a partition
\begin{align}
    \Lambda_N
    =
    R_1\dot\cup\cdots\dot\cup R_m,
\end{align}
then the interaction weight factorizes pointwise:
\begin{align}
    \exp\left\{
        \sum_{x\in\Lambda_N}
        \bigl(u+\eta_x^{(\alpha)}\bigr)
        \mathbf 1_{\{\phi(x)=0\}}
    \right\}
    =
    \prod_{i=1}^m
    \exp\left\{
        \sum_{x\in R_i}
        \bigl(u+\eta_x^{(\alpha)}\bigr)
        \mathbf 1_{\{\phi(x)=0\}}
    \right\}.
    \label{eq:proof-strategy-pinning-factorization}
\end{align}
Together with the conditional independence of the SOS field, this is the factorization mechanism used in \cite[Equations~(4.6)--(4.7)]{lacoin2021solid}.

For the wetting model, however, the shifted wall penalty depends on the connected components of
\begin{align}
    A(\phi):=\phi^{-1}(\mathbb Z_-).
\end{align}
Although $\overline H$ is additive over distinct connected components, it is not in general additive under the spatial partition produced by the coarse-graining:
\begin{align}
    \overline H_\beta(A(\phi))
    \neq
    \sum_{i=1}^m
    \overline H_\beta\bigl(A(\phi)\cap R_i\bigr)
\end{align}
in general, since a single connected component of $A(\phi)$ may intersect several regions.
Thus the SOS measure factorizes, but the wetting interaction does not factorize pointwise.

This obstruction is resolved in Subsection~\ref{subsec:4.4}.
Lemma~\ref{lem:sharp-local-coarse-graining} replaces the non-factorized wall interaction by a pointwise lower bound consisting of local cell contributions together with lower-order cross-region terms.
Combining this decomposition with the conditional factorization yields Proposition~\ref{prop:coarse-grained-lower-bound-decomposition}.

The remaining estimates are carried out in Section~5.
The local comparison in Lemma~\ref{lem:local-comparison-conditioned-cell} and the peak estimates of Proposition~\ref{prop:conditioned-cell-peak-estimates} recover the sharp coefficients $a_1(\beta)$ and $a_2(\beta)$ after conditioning.
These are converted into cluster-count estimates in Lemma~\ref{lem:sharp-cluster-counts-cell}.
The resulting sharp cell contribution is obtained in Lemma~\ref{lem:sharp-deterministic-cell-expansion}, while Lemma~\ref{lem:disorder-correction-sharp-cell} shows that the centered disorder contributes only a lower-order correction.
Finally, Lemma~\ref{lem:contour-region-cross-edge-errors} controls the remaining contour-region and cross-region terms.

Combining these estimates gives, uniformly over the localization window,
\begin{align}
    \overline F_\beta(\alpha,u)
    \ge
    a_1(\beta)uJ^{2n}
    -
    2a_2(\beta)\frac{J^3-J^4}{1-J^4}J^{3n}
    -
    o(u^3).
\end{align}
Optimizing over $n$ as in Lemma~\ref{lem:localization-window} yields
\begin{align}
    \overline F_\beta(\alpha,u)
    \ge
    \overline F_{\mathrm{homo}}(\beta,u)
    -
    o(u^3),
\end{align}
which proves Proposition~\ref{prop:sharp-lower-bound}.

\subsection{Further questions}\label{subsec:2.5}

Our main result shows that, at the level of the leading near-critical asymptotics, the quenched disorder does not modify the behavior of the homogeneous wetting model.
More precisely, by \eqref{eq:main-free-energy-asymptotic},
\begin{align}
    \overline F_\beta(\alpha,u)
    =
    \overline F_{\mathrm{homo}}(\beta,u)
    +
    o(u^3),
    \qquad
    u\rightarrow 0+,
\end{align}
where the leading term is given by the optimization over the integer localization height in \eqref{eq:homogeneous-wetting-variational}.
This raises the question of whether the resemblance with the homogeneous model persists at the finer scale at which individual layering transitions can be resolved.

For the homogeneous SOS wetting model, this question is by now well understood.
In \cite{lacoin2020wetting}, it was proved that, for $\beta$ sufficiently large, there exists a decreasing sequence $(h_n^*(\beta))_{n\geq 1}$ converging to $h_w(\beta)$ at which the homogeneous free energy is not differentiable.
Between two consecutive transition points, there is a unique translation-invariant Gibbs state localized around a fixed integer height, whereas at $h_n^*(\beta)$ at least two ergodic Gibbs states coexist, with respective typical heights $n-1$ and $n$.
Thus the changes of the preferred localization height suggested by the variational picture of \cite{lacoin2018wetting} correspond, in the homogeneous model, to genuine first-order layering transitions.

It is not clear that the same picture should persist in the presence of quenched disorder.
In our model the random variable $\eta_x^{(\alpha)}$ is coupled directly to the local contact observable, and, by \eqref{eq:contact-fraction-derivative}, the derivative of the quenched free energy with respect to $h$ is precisely the asymptotic contact density.
The rounding mechanism of Aizenman and Wehr \cite{aizenman1990rounding} therefore suggests a different scenario.
Their general result shows that, in dimension two, quenched randomness can eliminate first-order discontinuities in thermodynamic observables conjugate to the randomized parameter.
This suggests that, for fixed $\alpha>0$, the jump in the contact fraction associated with the homogeneous layering transitions may be rounded, so that the sharp first-order transitions are replaced by smoother crossovers.

We emphasize, however, that the result of \cite{aizenman1990rounding} does not apply directly to the present SOS model.
In particular, their general framework is formulated for bounded single-site state spaces, whereas the SOS height variable is unbounded.
Establishing an appropriate rounding result in the present setting would therefore require additional control of the infinite-volume Gibbs states.
It would be particularly interesting to determine whether, for every fixed $\alpha>0$, the quenched free energy is differentiable at the points corresponding to the homogeneous layering transitions.

Such a rounding scenario is not in contradiction with \eqref{eq:main-free-energy-asymptotic}.
Indeed, the latter determines the free energy only up to an $o(u^3)$ error and hence does not resolve the much finer structure in a small neighborhood of an individual crossing between two competing localization heights.
Disorder may therefore be invisible at the leading $u^3$ scale while still playing a decisive role at the finer scale at which a single layering transition is observed.

A second, closely related question concerns the actual height of the interface under the quenched Gibbs measure.
Let $n_u$ be a maximizer of the homogeneous variational problem.
By Lemma~\ref{lem:localization-window} and \eqref{eq:localization-height-scale},
\begin{align}
    J^{n_u}
    \asymp
    u,
    \qquad
    n_u
    =
    \frac{1}{2\beta}\log\frac1u
    +
    O_\beta(1).
\end{align}
Since the quenched free energy has the same leading asymptotic as the homogeneous one, it is natural to expect that, away from the narrow regions where two neighboring variational branches compete, the quenched interface is also localized around the corresponding height $n_u$.
Proving such a statement would give a probabilistic interpretation of the variational optimizer and would complement the free-energy result of Theorem~\ref{thm:main-result}.

If the layering transitions are indeed rounded by disorder, a further question is to describe the transition region in which the preferred height changes from $n$ to $n-1$.
In particular, one would like to understand the width of this crossover, the distribution of the interface height inside it, and the possible structure of the corresponding quenched Gibbs states.
This would clarify how the sharp layering picture of the homogeneous model is modified by spatial disorder, despite the fact that the leading free-energy asymptotics remains unchanged.

It would also be interesting to determine at which order beyond $u^3$ the dependence on the disorder first appears, and whether the same picture persists for other planar lattices or in higher dimensions.
\section{Preliminaries}\label{chap:3}

\subsection{Alternative representation and wall penalty}
\label{subsec:3.1}

We begin with the reformulation of the homogeneous wetting problem developed in Subsection~3.1 of \cite{lacoin2018wetting}.
The same representation remains valid in the present disordered setting after replacing the homogeneous contact reward by a site-dependent random reward.
Its purpose is to remove the positivity constraint and rewrite the wetting problem in terms of visits of an unconstrained SOS field to the non-positive half-plane.
A subsequent change of boundary condition will then turn these visits into high-level excursions of the unconstrained SOS field.

Set
\begin{align}
    \mathbb Z_-
    &:=
    \mathbb Z\cap(-\infty,0].
    \label{eq:Z-minus}
\end{align}
For a finite set $\Gamma\subset\mathbb Z^2$, define
\begin{align}
    \mathcal Z_{\Gamma,\beta}^{+}
    &:=
    \sum_{\psi\in\Omega_\Gamma^+}
    \exp\bigl\{-\beta\mathcal H_\Gamma(\psi)\bigr\},
    \qquad
    H_{\beta}(\Gamma)
    :=
    \log\mathcal Z_{\Gamma,\beta}^{+}.
    \label{eq:wall-partition-H}
\end{align}
We use the convention $H_{\beta}(\varnothing)=0$.
Since our analysis is carried out near the homogeneous wetting point, we introduce the shifted wall penalty
\begin{align}
    \overline H_{\beta}(\Gamma)
    &:=
    H_{\beta}(\Gamma)
    -
    h_w(\beta)|\Gamma|.
    \label{eq:Hbar-def}
\end{align}

The first lemma removes the hard-wall constraint and expresses the disordered wetting partition function as an expectation under the unconstrained SOS measure.
It is the disordered wetting counterpart of Lemma~3.1 and the shifted representation~(3.11) in \cite{lacoin2018wetting}.
The only modification is that the homogeneous reward attached to $\phi^{-1}(\mathbb Z_-)$ is replaced by the site-dependent reward $u+\eta_x^{(\alpha)}$.

\begin{lemma}[Alternative representation for disordered wetting; Adapted from \ \cite{lacoin2018wetting}, Lemma~3.1 and (3.11)]
\label{lemma:alternative-representation-disordered}
For every finite $\Lambda\subset\mathbb Z^2$, every $u\in\mathbb R$, and every realization of the disorder,
\begin{align}
    \frac{
        \mathcal Z_{\Lambda,\beta}^{+,h_w(\beta)+u,\alpha,\omega}
    }{
        \mathcal Z_{\Lambda,\beta}
    }
    &=
    \E_{\Lambda,\beta}
    \left[
        \exp\left\{
            \sum_{x\in\phi^{-1}(\mathbb Z_-)}
            \bigl(u+\eta_x^{(\alpha)}\bigr)
            -
            \overline H_{\beta}\bigl(\phi^{-1}(\mathbb Z_-)\bigr)
        \right\}
    \right].
    \label{eq:alternative-representation-disordered}
\end{align}
\end{lemma}

\begin{proof}[Proof of Lemma ~\ref{lemma:alternative-representation-disordered}]
The proof follows the positive/negative-part decomposition in \cite[Lemma~3.1]{lacoin2018wetting}, together with the shift leading to \cite[(3.11)]{lacoin2018wetting}.
The only change is the replacement
\begin{align*}
    u\bigl|\phi^{-1}(\mathbb Z_-)\bigr|
    \qquad\text{by}\qquad
    \sum_{x\in\phi^{-1}(\mathbb Z_-)}
    \bigl(u+\eta_x^{(\alpha)}\bigr).
\end{align*}
Since the additional disorder term depends only on the set $\phi^{-1}(\mathbb Z_-)$, it does not affect the positive/negative-part decomposition.
\end{proof}

To use Lemma~\ref{lemma:alternative-representation-disordered} sharply, we need to distinguish isolated visits from connected clusters containing several sites.
Isolated visits produce the leading positive contribution to the free energy, while nearest-neighbor pairs produce the first non-trivial deterministic wall penalty.
The following result is the shifted formulation of the wall-penalty estimates in \cite[Lemma~3.2]{lacoin2018wetting}, together with the additivity property stated immediately before that lemma.
It will allow us later to retain the singleton and pair contributions sharply and control all larger clusters by a rough linear bound.

We call two finite sets $\Gamma_1,\Gamma_2\subset\mathbb Z^2$ separated if they are disjoint and there is no nearest-neighbor edge joining them.

\begin{lemma}[Wall penalty in the wetting representation; Adapted from \cite{lacoin2018wetting}, Lemma~3.2]
\label{lemma:Hbar-basic-properties}
Assume that $\beta$ is sufficiently large.
If $\Gamma_1$ and $\Gamma_2$ are separated, then
\begin{align}
    \overline H_{\beta}(\Gamma_1\cup\Gamma_2)
    &=
    \overline H_{\beta}(\Gamma_1)
    +
    \overline H_{\beta}(\Gamma_2).
    \label{eq:Hbar-additivity}
\end{align}
Consequently, $\overline H_{\beta}(\Gamma)$ is additive over the connected components of $\Gamma$.

Moreover, for every $x\in\mathbb Z^2$,
\begin{align}
    \overline H_{\beta}(\{x\})
    &=
    0.
    \label{eq:Hbar-singleton}
\end{align}
If $x\sim y$, then
\begin{align}
    \overline H_{\beta}(\{x,y\})
    &=
    \log\left(
        \frac{1-J(\beta)^4}{1-J(\beta)^3}
    \right).
    \label{eq:Hbar-doubleton}
\end{align}
We set
\begin{align}
    \ell_2(\beta)
    &:=
    \log\left(
        \frac{1-J(\beta)^4}{1-J(\beta)^3}
    \right).
    \label{eq:ell2-def}
\end{align}
Finally, there exist constants $c_1(\beta),c_2(\beta)>0$ such that, for every finite connected set $\Gamma\subset\mathbb Z^2$ with $|\Gamma|\ge2$,
\begin{align}
    c_1(\beta)|\Gamma|
    &\le
    \overline H_{\beta}(\Gamma)
    \le
    c_2(\beta)|\Gamma|.
    \label{eq:Hbar-linear-bound}
\end{align}
\end{lemma}

\begin{proof}[Proof of Lemma ~\ref{lemma:Hbar-basic-properties}]
The statement follows directly from \cite[Lemma~3.2]{lacoin2018wetting} and the additivity property stated immediately before that lemma, in the low-temperature regime in which \eqref{eq:wetting-point-J} holds after introducing
\begin{align*}
    \overline H_{\beta}(\Gamma)
    =
    H_{\beta}(\Gamma)-h_w(\beta)|\Gamma|.
\end{align*}
\end{proof}

Lemma~\ref{lemma:Hbar-basic-properties} makes the contribution of the different connected components of $\phi^{-1}(\mathbb Z_-)$ explicit.
If $\Gamma$ is such a connected component, then its contribution to the exponent in \eqref{eq:alternative-representation-disordered} is
\begin{align}
    \sum_{x\in\Gamma}
    \bigl(u+\eta_x^{(\alpha)}\bigr)
    -
    \overline H_{\beta}(\Gamma).
    \label{eq:component-contribution}
\end{align}
In particular, an isolated component $\{x\}$ contributes
\begin{align*}
    u+\eta_x^{(\alpha)},
\end{align*}
whereas a nearest-neighbor component $\{x,y\}$ contributes
\begin{align*}
    2u
    +
    \eta_x^{(\alpha)}
    +
    \eta_y^{(\alpha)}
    -
    \ell_2(\beta).
\end{align*}
Thus the disorder modifies only the site-dependent reward, while the deterministic penalty associated with non-singleton components is the same as in the homogeneous wetting model.

The representation above is still formulated in terms of visits to the non-positive half-plane.
For the lower-bound argument, we instead want to use the known estimates for rare high-level excursions of the unconstrained SOS field.
The next lemma converts $\phi^{-1}(\mathbb Z_-)$ into the high-level set $\phi^{-1}[n,\infty)$.
It is the disordered wetting counterpart of \cite[Lemma~5.2]{lacoin2018wetting}; the boundary-condition comparison and reflection argument are unchanged because the disorder remains attached to the same spatial sites.

\begin{lemma}[High-level representation for disordered wetting; Adapted from\ \cite{lacoin2018wetting}, Lemma~5.2]
\label{lemma:change-boundary-disordered}
Fix $n\in\mathbb Z_+$.
For every $u\in\mathbb R$,
\begin{align}
    \overline F_\beta(\alpha,u)
    &=
    \lim_{N\rightarrow\infty}
    \frac{1}{N^2}
    \log
    \E_{N,\beta}
    \left[
        \exp\left\{
            \sum_{x\in\phi^{-1}[n,\infty)}
            \bigl(u+\eta_x^{(\alpha)}\bigr)
            -
            \overline H_{\beta}\bigl(\phi^{-1}[n,\infty)\bigr)
        \right\}
    \right].
    \label{eq:change-boundary-disordered}
\end{align}
The limit in \eqref{eq:change-boundary-disordered} holds $\P$-almost surely and in $L^1(\P)$.
\end{lemma}

\begin{proof}
The proof follows \cite[Lemma~5.2]{lacoin2018wetting}.
The change from zero boundary condition to boundary height $n$ has only a boundary-order cost, since
\begin{align*}
\left|
H_{\Lambda_N}^{n}(\phi)
-
H_{\Lambda_N}(\phi)
\right|
&\le
4nN.
\end{align*}
Under the reflection $\psi=n-\phi$, the set $\phi^{-1}(\mathbb Z_-)$ is mapped to $\psi^{-1}[n,\infty)$.
Since this reflection changes only the height variable and leaves the spatial sites unchanged, the disorder reward $\eta_x^{(\alpha)}$ is unaffected.
Combining this identity with Lemma~\ref{lemma:alternative-representation-disordered}, Proposition~\ref{prop:free-energy-basic-properties}, and the boundary-order estimate above yields \eqref{eq:change-boundary-disordered}.
The same comparison preserves both the $\P$-almost-sure and the $L^1(\P)$ convergence.
\end{proof}

Consequently, the lower-bound problem is reduced to studying rare connected components of the high-level set
\begin{align*}
    \phi^{-1}[n,\infty)
\end{align*}
under the unconstrained SOS measure.
The required high-level estimates are collected in the following subsections.

\subsection{Contour representation and restricted SOS measures}
\label{subsec:3.2}

We recall the contour representation of the low-temperature SOS measure and introduce the contour-restricted measures that will be used throughout the coarse-graining argument.
The notation and the basic facts below follow \cite[Section~4.2]{lacoin2018wetting} and \cite[Section~3.1]{lacoin2021solid}.
The main point for our purposes is that the contour representation remains available after restricting the collection of admissible contours.
This will allow us later to identify the conditional laws arising from the coarse-graining procedure as contour-restricted SOS measures.

We first recall the geometric contour notation.
Let $(\mathbb Z^2)^*$ denote the dual lattice of $\mathbb Z^2$.
Two adjacent dual edges meeting at a dual vertex $x^*$ are said to be \emph{linked} if they lie on the same side of the line passing through $x^*$ and making an angle $\pi/4$ with the horizontal axis.
A contour sequence is a finite sequence of distinct dual edges $(e_1,\ldots,e_m)$ such that consecutive edges share an endpoint, with the cyclic convention $e_{m+1}=e_1$, and such that whenever four edges of the sequence meet at the same dual vertex, the two consecutive pairs are linked according to the above convention.

A \emph{geometric contour} is a set of dual edges $\widetilde\gamma$ which can be ordered as a contour sequence.
We write
\begin{align}
    |\widetilde\gamma|
    &:=
    \text{the number of dual edges in }\widetilde\gamma
    \label{eq:geometric-contour-length}
\end{align}
for its length.
The set of lattice vertices enclosed by $\widetilde\gamma$ is denoted by
\begin{align*}
    \overline\gamma
    \subset
    \mathbb Z^2,
\end{align*}
and will be referred to as the \emph{interior} of the contour.

A \emph{signed contour} is a pair consisting of a geometric contour and a sign:
\begin{align*}
    \gamma
    &=
    \bigl(
        \widetilde\gamma,
        \varepsilon(\gamma)
    \bigr),
    \qquad
    \varepsilon(\gamma)\in\{-1,+1\}.
\end{align*}
We denote by $\mathsf{Cont}$ the collection of all signed contours in $\mathbb Z^2$, and, for a finite set $\Lambda\subset\mathbb Z^2$, set
\begin{align}
    \mathsf{Cont}_\Lambda
    &:=
    \left\{
        \gamma\in\mathsf{Cont}:
        \overline\gamma\subset\Lambda
    \right\}.
    \label{eq:contours-in-domain}
\end{align}

For a signed contour $\gamma$, let $\Delta_\gamma$ denote its lattice neighborhood.
More precisely, $\Delta_\gamma$ consists of the lattice sites at Euclidean distance $1/2$ from $\widetilde\gamma$, together with the sites at distance $1/\sqrt 2$ from a dual vertex at which two non-linked edges of $\widetilde\gamma$ meet.
We decompose this neighborhood into its internal and external parts:
\begin{align}
    \Delta_\gamma^-
    &:=
    \Delta_\gamma\cap\overline\gamma,
    \qquad
    \Delta_\gamma^+
    :=
    \Delta_\gamma\cap\overline\gamma^{\,c}.
    \label{eq:internal-external-contour-neighborhoods}
\end{align}

Let $\Lambda\subset\mathbb Z^2$ be finite and let $n\in\mathbb Z$.
Given $\phi\in\Omega_\Lambda$, we use the convention $\phi(x)=n$ for $x\notin\Lambda$.
A signed contour $\gamma\in\mathsf{Cont}_\Lambda$ is said to be a \emph{contour of $\phi$ with boundary height $n$} if
\begin{align}
    k(\gamma,\phi)
    &:=
    \min_{
        \substack{
            x\in\Delta_\gamma^-\\
            y\in\Delta_\gamma^+
        }
    }
    \varepsilon(\gamma)
    \bigl(
        \phi(x)-\phi(y)
    \bigr)
    >
    0.
    \label{eq:contour-intensity-definition}
\end{align}
Since $\phi$ is integer-valued, $k(\gamma,\phi)\in\mathbb N$.
We call this quantity the \emph{intensity} of the contour.

A \emph{cylinder} is a contour together with a positive integer intensity.
We write
\begin{align*}
    \widehat\gamma
    &=
    (\gamma,k),
    \qquad
    k=k(\widehat\gamma)\in\mathbb N,
\end{align*}
and associate with it the cylinder function
\begin{align}
    \varphi_{\widehat\gamma}(x)
    &:=
    \varepsilon(\gamma)
    k(\widehat\gamma)
    \mathbf 1_{\overline\gamma}(x).
    \label{eq:cylinder-function}
\end{align}
For a configuration $\phi$ with boundary height $n$, we denote by
\begin{align*}
    \widehat\Upsilon_n(\phi)
\end{align*}
the collection of all its cylinders, and by
\begin{align*}
    \Upsilon_n(\phi)
\end{align*}
the corresponding collection of signed contours.
When the boundary height is zero, we simply write $\widehat\Upsilon(\phi)$ and $\Upsilon(\phi)$.

We next recall the compatibility relation between contours.
Two cylinders $\widehat\gamma$ and $\widehat\gamma'$ are said to be \emph{compatible} if both are cylinders of the configuration
\begin{align*}
    \varphi_{\widehat\gamma}
    +
    \varphi_{\widehat\gamma'}.
\end{align*}
This condition depends only on the underlying signed contours and not on their intensities.
We therefore also say that two signed contours $\gamma$ and $\gamma'$ are compatible whenever the cylinders $(\gamma,1)$ and $(\gamma',1)$ are compatible.
A collection of signed contours or cylinders is called compatible if its elements are pairwise compatible.

Equivalently, compatible contours do not cross one another, their interiors are either disjoint or nested, and their signs satisfy the usual local compatibility conditions along their internal and external neighborhoods; see \cite[Section~4.2]{lacoin2018wetting}.

We say that a finite set $\Lambda\subset\mathbb Z^2$ is simply connected if it is the interior of a geometric contour.
On a simply connected domain, the cylinder representation determines the field uniquely.
The following standard facts collect the parts of the contour representation that will be used later.

\begin{lemma}[Contour and cylinder representation]
\label{lem:contour-cylinder-representation}
Let $\Lambda\subset\mathbb Z^2$ be finite and simply connected, and let $n\in\mathbb Z$.

For every $\phi\in\Omega_\Lambda$, the cylinder collection $\widehat\Upsilon_n(\phi)$ is compatible.
Conversely, every compatible collection of cylinders in $\Lambda$ determines a unique SOS configuration with boundary height $n$.

More precisely, for every $x\in\Lambda$,
\begin{align}
    \phi(x)
    &=
    n
    +
    \sum_{
        \widehat\gamma
        \in
        \widehat\Upsilon_n(\phi)
    }
    \varepsilon(\gamma)
    k(\widehat\gamma)
    \mathbf 1_{\overline\gamma}(x).
    \label{eq:cylinder-reconstruction}
\end{align}
Moreover,
\begin{align}
    \mathcal H_\Lambda^n(\phi)
    &=
    \sum_{
        \widehat\gamma
        \in
        \widehat\Upsilon_n(\phi)
    }
    k(\widehat\gamma)
    |\widetilde\gamma|.
    \label{eq:contour-energy-representation}
\end{align}
\end{lemma}

The identities \eqref{eq:cylinder-reconstruction} and \eqref{eq:contour-energy-representation} are the reason why the SOS measure admits a particularly simple description in terms of its contours.
We recall this description next.

Let $\mathbb Q_{\Lambda,\beta}$ be the product probability measure on collections
\begin{align*}
    \Xi
    \subset
    \mathsf{Cont}_\Lambda
\end{align*}
under which the indicators $\mathbf 1_{\{\gamma\in\Xi\}}$ are independent and
\begin{align}
    \mathbb Q_{\Lambda,\beta}
    \bigl(
        \gamma\in\Xi
    \bigr)
    &=
    e^{-\beta|\widetilde\gamma|},
    \qquad
    \gamma\in\mathsf{Cont}_\Lambda.
    \label{eq:independent-contour-measure}
\end{align}
The law of the contour ensemble under the SOS measure is obtained by conditioning this product measure on compatibility.

\begin{lemma}[Distribution of contours and intensities]
\label{lem:contour-distribution-domination}
Let $\Lambda$ be finite and simply connected and let $n\in\mathbb Z$.
Under $\P_{\Lambda,\beta}^n$,
\begin{align}
    \operatorname{Law}
    \bigl(
        \Upsilon_n(\phi)
    \bigr)
    &=
    \mathbb Q_{\Lambda,\beta}
    \left(
        \Xi\in\cdot
        \,\middle|\,
        \Xi
        \text{ is compatible}
    \right).
    \label{eq:contour-law-product-conditioned}
\end{align}
In particular, the contour collection $\Upsilon_n(\phi)$ is stochastically dominated, with respect to inclusion, by the independent contour process $\mathbb Q_{\Lambda,\beta}$.

Conditionally on $\Upsilon_n(\phi)$, the intensities of the corresponding cylinders are independent geometric random variables.
More precisely, if $\gamma\in\Upsilon_n(\phi)$, then
\begin{align}
    \P_{\Lambda,\beta}^n
    \left(
        k(\widehat\gamma)=m
        \,\middle|\,
        \Upsilon_n(\phi)
    \right)
    &=
    \left(
        1-e^{-\beta|\widetilde\gamma|}
    \right)
    e^{-\beta|\widetilde\gamma|(m-1)},
    \qquad
    m\ge1.
    \label{eq:conditional-cylinder-intensity}
\end{align}
\end{lemma}

The compatibility conditioning in \eqref{eq:contour-law-product-conditioned} is decreasing with respect to inclusion.
Hence the stochastic domination in Lemma~\ref{lem:contour-distribution-domination} follows from the FKG--Harris inequality for the product contour measure.
This domination will be used repeatedly to control the occurrence of long contours.

For the coarse-graining argument we also need versions of the SOS measure in which only a prescribed family of contours is allowed.
Let
\begin{align*}
    \mathcal L
    \subset
    \mathsf{Cont}_\Lambda
\end{align*}
be a fixed family of signed contours.
We define the \emph{contour-restricted SOS measure} by
\begin{align}
    \P_{\mathcal L,\Lambda,\beta}^n
    (\cdot)
    &:=
    \P_{\Lambda,\beta}^n
    \left(
        \,\cdot\,
        \middle|
        \Upsilon_n(\phi)
        \subset
        \mathcal L
    \right).
    \label{eq:contour-restricted-sos-measure}
\end{align}
When $n=0$, we omit the superscript and write
\begin{align*}
    \P_{\mathcal L,\Lambda,\beta}.
\end{align*}
We denote the corresponding expectation by $\E_{\mathcal L,\Lambda,\beta}^n$, or simply by $\E_{\mathcal L,\Lambda,\beta}$ when $n=0$.

The contour representation above remains valid under this restriction.
More precisely, under $\P_{\mathcal L,\Lambda,\beta}^n$, the contour collection has the law of the independent contour process restricted to $\mathcal L$ and conditioned on compatibility:
\begin{align}
    \operatorname{Law}_{
        \P_{\mathcal L,\Lambda,\beta}^n
    }
    \bigl(
        \Upsilon_n(\phi)
    \bigr)
    &=
    \mathbb Q_{\Lambda,\beta}
    \left(
        \Xi\in\cdot
        \,\middle|\,
        \Xi\subset\mathcal L,
        \,
        \Xi
        \text{ is compatible}
    \right).
    \label{eq:restricted-contour-law}
\end{align}
Consequently, for every increasing function $F$ of the contour collection,
\begin{align}
    \E_{\mathcal L,\Lambda,\beta}^n
    \left[
        F\bigl(\Upsilon_n(\phi)\bigr)
    \right]
    &\le
    \mathbb E_{\mathbb Q_{\Lambda,\beta}}
    \left[
        F\bigl(\Xi\cap\mathcal L\bigr)
    \right]
    \nonumber\\
    &\le
    \mathbb E_{\mathbb Q_{\Lambda,\beta}}
    \left[
        F(\Xi)
    \right].
    \label{eq:restricted-contour-domination}
\end{align}
Thus all Peierls-type upper bounds obtained from the independent contour process remain valid uniformly under contour restrictions.

We shall need to quantify how far a site is from a contour that is forbidden by the restriction.
Following \cite[Section~3.2]{lacoin2021solid}, for $x\in\Lambda$ we define
\begin{align}
    d
    \left(
        x,
        \mathcal L^c
    \right)
    &:=
    \inf_{
        \gamma
        \in
        \mathsf{Cont}
        \setminus
        \mathcal L
    }
    \;
    \max_{
        y\in\overline\gamma
    }
    |x-y|_1,
    \label{eq:distance-to-forbidden-contours}
\end{align}
with the convention that the infimum over the empty set is $+\infty$.
Here the complement is taken in the collection of all signed contours in $\mathbb Z^2$.
In particular, if every contour contained in $\Lambda$ is allowed, then
\begin{align}
    d
    \left(
        x,
        \mathcal L^c
    \right)
    &\ge
    d
    \left(
        x,
        \Lambda^c
    \right).
    \label{eq:forbidden-contour-distance-domain}
\end{align}
The quantity $d(x,\mathcal L^c)$ will measure the boundary and contour-restriction error in the high-level peak estimates of Subsection~\ref{subsec:3.3}.

Finally, we introduce the notion of external contours, which will be used in the coarse-graining decomposition.
Let $\Gamma$ be a compatible collection of signed contours.
A contour $\gamma\in\Gamma$ is called \emph{external} if its interior is maximal under inclusion, that is, if for every $\gamma'\in\Gamma$,
\begin{align}
    \overline{\gamma'}
    &\subset
    \overline\gamma
    \qquad\text{or}\qquad
    \overline{\gamma'}
    \cap
    \overline\gamma
    =
    \varnothing.
    \label{eq:external-contour-definition}
\end{align}
We denote by
\begin{align*}
    \Upsilon_n^{\mathrm{ext}}(\phi)
\end{align*}
the collection of external contours of $\Upsilon_n(\phi)$.
Their interiors are pairwise disjoint.

The contour representation also yields the spatial factorization that will be used later.
Conditionally on the realization of the external contours, the SOS fields inside their interiors are independent.
The same property remains valid under a contour restriction $\P_{\mathcal L,\Lambda,\beta}^n$, with the corresponding induced boundary heights and allowed-contour families.
More generally, after fixing a collection of outer cylinders, the remaining cylinders are partitioned according to the connected regions in which they are contained, and the corresponding SOS fields factorize over these regions.

This conditional-independence property is the basis of the coarse-graining construction in Subsection~\ref{subsec:4.2}.
In particular, the conditional cell laws appearing there, and the reflected laws used in Appendix~\ref{app:local-comparison}, can be identified with contour-restricted SOS measures of the form \eqref{eq:contour-restricted-sos-measure}.

The uniform one-point, two-point, and multipoint estimates for these restricted measures are stated in the next subsection.

\subsection{High-level peak estimates}
\label{subsec:3.3}

We collect here the high-level estimates for the low-temperature SOS measure that will be used in the coarse-graining argument.
There are two types of estimates that we need.
First, we require bounds which remain uniform under restrictions on the admissible contours, since the conditional cell measures arising later are naturally of this form.
Second, in the bulk of a cell we need the sharp infinite-volume asymptotics of isolated one-site peaks and nearest-neighbor two-site peaks in order to recover the constants $a_1(\beta)$ and $a_2(\beta)$ appearing in \eqref{eq:a1-beta}--\eqref{eq:a2-beta}.

We begin with the estimates which are uniform under contour restrictions.
Recall the contour-restricted measure $\P_{\mathcal L,\Lambda,\beta}$ and the distance $d(x,\mathcal L^c)$ introduced in Subsection~\ref{subsec:3.2}.
The following result is adapted from \cite[Proposition~3.4]{lacoin2021solid}.
The preceding contour representation was recalled explicitly for simply connected domains.
For the peak estimates below, however, we use the more general contour-restricted formulation of \cite[Proposition~3.4]{lacoin2021solid}, which is uniform over arbitrary finite domains $\Lambda$ and admissible contour families $\mathcal L$.

\begin{proposition}[Contour-restricted peak estimates]
\label{prop:restricted-peak-estimates}
Assume that $\beta$ is sufficiently large.
There exists $C_\beta<\infty$ such that, uniformly over finite simply connected domains $\Lambda\subset\mathbb Z^2$, families of admissible contours $\mathcal L\subset\mathsf{Cont}_\Lambda$, and integers $n\ge1$, the following estimates hold.

For every $x\in\Lambda$,
\begin{align}
    \P_{\mathcal L,\Lambda,\beta}
    \bigl(
        \phi(x)\ge n
    \bigr)
    &\le
    C_\beta J^{2n}.
    \label{eq:restricted-one-point}
\end{align}
For every pair of distinct sites $x,y\in\Lambda$,
\begin{align}
    \P_{\mathcal L,\Lambda,\beta}
    \left(
        \min\{\phi(x),\phi(y)\}\ge n
    \right)
    &\le
    C_\beta J^{3n}.
    \label{eq:restricted-two-point}
\end{align}

Moreover, the one-point probability admits the sharp approximation
\begin{align}
    \left|
        \P_{\mathcal L,\Lambda,\beta}
        \bigl(
            \phi(x)\ge n
        \bigr)
        -
        a_1(\beta)J^{2n}
    \right|
    &\le
    C_\beta
    \left(
        J^{3n}
        +
        e^{-d(x,\mathcal L^c)}
    \right).
    \label{eq:restricted-sharp-one-point}
\end{align}

Finally, for every fixed $k\ge1$, there exists $C_{\beta,k}<\infty$ such that, for any collection of distinct sites $x_1,\ldots,x_k\in\Lambda$,
\begin{align}
    \P_{\mathcal L,\Lambda,\beta}
    \left(
        \phi(x_i)\ge n
        \text{ for all }1\le i\le k
    \right)
    &\le
    C_{\beta,k}
    e^{-3\beta n\sqrt{k}}.
    \label{eq:restricted-multipoint}
\end{align}
\end{proposition}

\begin{remark}
In the displayed statement of \cite[Proposition~3.4, (3.18)]{lacoin2021solid}, the factor $n$ is missing from the exponent in the multipoint estimate.
The proof in \cite[Appendix~B.4, in particular (B.19)]{lacoin2021solid} gives the $n$-dependent estimate \eqref{eq:restricted-multipoint}, which is the form used here.
\end{remark}

The rough bounds \eqref{eq:restricted-one-point}, \eqref{eq:restricted-two-point}, and \eqref{eq:restricted-multipoint} follow from Peierls-type estimates and are stable under the contour restriction because of the stochastic domination \eqref{eq:restricted-contour-domination}.
The sharp estimate \eqref{eq:restricted-sharp-one-point} additionally uses the low-temperature cluster expansion.
Notice that the error produced by the contour restriction becomes exponentially small when $x$ is far from the forbidden contours.

We next recall the corresponding sharp estimates for the unrestricted SOS measure.
By \cite[Propositions~4.5 and~4.6]{lacoin2018wetting}, for nearest-neighbor sites $x\sim y$,
\begin{align}
    \left|
        \P_\beta
        \bigl(
            \phi(x)\ge n
        \bigr)
        -
        a_1(\beta)J^{2n}
    \right|
    &\le
    C_\beta J^{3n},
    \label{eq:infinite-volume-sharp-one-point}
    \\
    \left|
        \P_\beta
        \left(
            \min\{\phi(x),\phi(y)\}\ge n
        \right)
        -
        a_2(\beta)J^{3n}
    \right|
    &\le
    C_\beta nJ^{4n}.
    \label{eq:infinite-volume-sharp-two-point}
\end{align}
In particular,
\begin{align}
    \P_\beta
    \bigl(
        \phi(x)\ge n
    \bigr)
    &\asymp
    J^{2n},
    \nonumber\\
    \P_\beta
    \left(
        \min\{\phi(x),\phi(y)\}\ge n
    \right)
    &\asymp
    J^{3n},
    \label{eq:rough-infinite-volume-one-two-point}
\end{align}
uniformly in $n\ge1$.

For the lower-bound argument, we need slightly more precise information than \eqref{eq:infinite-volume-sharp-one-point}--\eqref{eq:infinite-volume-sharp-two-point}.
Namely, we must distinguish exact singleton and exact two-site components of the high-level set from larger clusters.

For $x\in\mathbb Z^2$ and $n\ge1$, define the exact singleton event
\begin{align}
    S_{x,n}^{\mathrm{hi}}
    &:=
    \left\{
        \{x\}
        \text{ is a connected component of }
        \phi^{-1}[n,\infty)
    \right\}.
    \label{eq:high-level-singleton-event}
\end{align}
For a nearest-neighbor edge $e=\{x,y\}$, define the exact pair event
\begin{align}
    P_{e,n}^{\mathrm{hi}}
    &:=
    \left\{
        \{x,y\}
        \text{ is a connected component of }
        \phi^{-1}[n,\infty)
    \right\}.
    \label{eq:high-level-pair-event}
\end{align}

The following consequence of the estimates in \cite[Section~4.3]{lacoin2018wetting} is the form that will be used later.

\begin{corollary}[Exact high-level clusters]
\label{cor:exact-peak-clusters}
Assume that $\beta$ is sufficiently large.
There exists $C_\beta<\infty$ such that, for every $n\ge1$, $x\in\mathbb Z^2$, and every nearest-neighbor edge $e=\{x,y\}$,
\begin{align}
    \left|
        \P_\beta
        \bigl(
            S_{x,n}^{\mathrm{hi}}
        \bigr)
        -
        a_1(\beta)J^{2n}
    \right|
    &\le
    C_\beta J^{3n},
    \label{eq:infinite-volume-exact-singleton}
    \\
    \left|
        \P_\beta
        \bigl(
            P_{e,n}^{\mathrm{hi}}
        \bigr)
        -
        a_2(\beta)J^{3n}
    \right|
    &\le
    C_\beta nJ^{4n},
    \label{eq:infinite-volume-exact-pair}
    \\
    \P_\beta
    \left(
    \{
        \phi(x)\ge n,\,
        \phi(y)\ge n
    \}
    \setminus
    P_{e,n}^{\mathrm{hi}}
    \right)
    &\le
    C_\beta nJ^{4n}.
    \label{eq:infinite-volume-pair-in-larger-cluster}
\end{align}
\end{corollary}

\begin{proof}[Proof of Corollary ~\ref{cor:exact-peak-clusters}]
Let $q(\phi,x,n)$ denote the cardinality of the connected component of $\phi^{-1}[n,\infty)$ containing $x$, with the convention $q(\phi,x,n)=0$ when $\phi(x)<n$.
By \cite[Proposition~4.6, (4.12)]{lacoin2018wetting},
\begin{align}
    \left|
        \P_\beta
        \bigl(
            q(\phi,x,n)=1
        \bigr)
        -
        a_1(\beta)J^{2n}
    \right|
    &\le
    C_\beta J^{3n},
    \label{eq:q-one-infinite-volume}
    \\
    \left|
        \P_\beta
        \bigl(
            q(\phi,x,n)=2
        \bigr)
        -
        4a_2(\beta)J^{3n}
    \right|
    &\le
    C_\beta nJ^{4n}.
    \label{eq:q-two-infinite-volume}
\end{align}
Since
\begin{align*}
    S_{x,n}^{\mathrm{hi}}
    &=
    \{
        q(\phi,x,n)=1
    \},
\end{align*}
the first estimate follows immediately from \eqref{eq:q-one-infinite-volume}.

For the second estimate, observe that the event $\{q(\phi,x,n)=2\}$ is the disjoint union, over the four nearest neighbors $y$ of $x$, of the events $P_{\{x,y\},n}^{\mathrm{hi}}$.
By translation and lattice symmetries under $\P_\beta$, these four events have the same probability.
Hence
\begin{align*}
    \P_\beta
    \bigl(
        P_{e,n}^{\mathrm{hi}}
    \bigr)
    &=
    \frac14
    \P_\beta
    \bigl(
        q(\phi,x,n)=2
    \bigr),
\end{align*}
and \eqref{eq:infinite-volume-exact-pair} follows from \eqref{eq:q-two-infinite-volume}.

Finally, since
\begin{align*}
    P_{e,n}^{\mathrm{hi}}
\subset
\{\phi(x)\ge n,\phi(y)\ge n\},
\end{align*}
we have
\begin{align*}
&
\P_\beta
\left(
    \{\phi(x)\ge n,\phi(y)\ge n\}
    \setminus
    P_{e,n}^{\mathrm{hi}}
\right)
\\
&=
\P_\beta
\bigl(
    \phi(x)\ge n,\phi(y)\ge n
\bigr)
-
\P_\beta
\bigl(
    P_{e,n}^{\mathrm{hi}}
\bigr).
\end{align*}
Therefore, by \eqref{eq:infinite-volume-sharp-two-point} and \eqref{eq:infinite-volume-exact-pair},
\begin{align*}
&
\P_\beta
\left(
    \{\phi(x)\ge n,\phi(y)\ge n\}
    \setminus
    P_{e,n}^{\mathrm{hi}}
\right)
\\
&\le
\left|
\P_\beta
\bigl(
    \phi(x)\ge n,\phi(y)\ge n
\bigr)
-
a_2(\beta)J^{3n}
\right|+
\left|
\P_\beta
\bigl(
    P_{e,n}^{\mathrm{hi}}
\bigr)
-
a_2(\beta)J^{3n}
\right|
\\
&\le
C_\beta nJ^{4n}.
\end{align*}
This proves \eqref{eq:infinite-volume-pair-in-larger-cluster}.
\end{proof}

The distinction between \eqref{eq:restricted-one-point}--\eqref{eq:restricted-multipoint} and Corollary~\ref{cor:exact-peak-clusters} will be important later.
The contour-restricted estimates are robust enough to be applied throughout the random cell regions produced by the coarse-graining, whereas the exact constants $a_1(\beta)$ and $a_2(\beta)$ are recovered only in the deterministic cores by comparison with the infinite-volume measure.
This is carried out in Subsection~\ref{subsec:5.1}.

\section{Coarse-grained reduction of the lower bound}
\label{chap:4}

The annealed comparison already yields the upper bound in Theorem~\ref{thm:main-result}.
Thus, the main remaining step is to establish the following matching quenched lower bound.

\begin{proposition}[Sharp quenched lower bound]
\label{prop:sharp-lower-bound}
Assume that $\beta$ is sufficiently large and fix $\alpha\ge0$.
Then, as $u\rightarrow 0+$,
\begin{align}
\overline{F}_\beta(\alpha,u)
\ge
\overline{F}_{\mathrm{homo}}(\beta,u)-o(u^3).
\label{eq:sharp-quenched-lower-bound}
\end{align}
\end{proposition}

Section~\ref{chap:4} reduces Proposition~\ref{prop:sharp-lower-bound} to estimates on coarse-grained cells and exceptional regions.
These estimates are established in Section~\ref{chap:5}.

By Lemma~\ref{lemma:change-boundary-disordered}, the problem can be reformulated in terms of the rare high-level set of the unconstrained SOS field, with the interaction given by the random contact reward and the shifted wall penalty introduced in Subsection~\ref{subsec:3.1}.

We first identify, from the variational formula \eqref{eq:homogeneous-wetting-variational}, the range of heights relevant at order $u^3$.
We then restrict to configurations with no long contours and introduce a coarse-grained decomposition under which the SOS field becomes conditionally independent across the resulting regions.
Finally, we isolate the sharp one-site and nearest-neighbor contributions from the lower-order geometric and cross-region terms, leading to the coarse-grained lower bound used in Section~\ref{chap:5}.

\subsection{Finite-volume reduction and localization height}
\label{subsec:4.1}

We first introduce the finite-volume quantity to which the coarse-graining argument will be applied and identify the range of boundary heights relevant at order $u^3$.
Fix an integer boundary height $n\ge1$.

Recall the notation introduced in Subsection~\ref{subsec:2.2} and \eqref{eq:Hbar-def}.
We define
\begin{align}
S_N^\omega(\phi)
&:=
\sum_{x\in\phi^{-1}(\mathbb Z_-)}
\left(
u+\eta_x^{(\alpha)}
\right)
-
\overline H_\beta
\bigl(
\phi^{-1}(\mathbb Z_-)
\bigr).
\label{eq:finite-volume-wall-functional}
\end{align}
The dependence of $S_N^\omega$ on $u$ and $\alpha$ is suppressed from the notation.
The corresponding reduced finite-volume partition function is
\begin{align}
Z_{N,n}^\omega(u)
&:=
\E_{N,\beta}^n
\left[
\exp\left\{
S_N^\omega(\phi)
\right\}
\right].
\label{eq:finite-volume-reduced-partition}
\end{align}

By reflection, this is precisely the finite-volume high-level functional appearing in Lemma~\ref{lemma:change-boundary-disordered}.
Indeed, if $\psi:=n-\phi$, then under $\P_{N,\beta}^n$ the field $\psi$ has law $\P_{N,\beta}$ and $\phi^{-1}(\mathbb Z_-)=\psi^{-1}[n,\infty)$.
Consequently,
\begin{align}
Z_{N,n}^\omega(u)
=
\E_{N,\beta}
\left[
\exp\left\{
\sum_{x\in\psi^{-1}[n,\infty)}
\left(
u+\eta_x^{(\alpha)}
\right)
-
\overline H_\beta
\bigl(
\psi^{-1}[n,\infty)
\bigr)
\right\}
\right].
\label{eq:reflection-finite-volume-partition}
\end{align}
It follows from Lemma~\ref{lemma:change-boundary-disordered} and Proposition~\ref{prop:free-energy-basic-properties} that, for every fixed $n\ge1$,
\begin{align}
\overline{F}_\beta(\alpha,u)
&=
\lim_{N \rightarrow\infty}
\frac{1}{N^2}
\log Z_{N,n}^\omega(u)
\nonumber\\
&=
\lim_{N\rightarrow\infty}
\frac{1}{N^2}
\E\log Z_{N,n}^\omega(u).
\label{eq:free-energy-finite-volume-representation}
\end{align}
The first convergence holds $\P$-almost surely and in $L^1(\P)$.
Although the boundary height will later depend on $u$, it is fixed when the limit $N\rightarrow\infty$ is taken.
We therefore seek a lower bound on $\E\log Z_{N,n}^\omega(u)$ that is uniform over the relevant choices of $n$.

Recall from \eqref{eq:ell2-def} that
\begin{align*}
\ell_2(\beta)
&=
\log\left(
\frac{1-J^4}{1-J^3}
\right)
\end{align*}
is the shifted wall penalty associated with a nearest-neighbor pair.
Set
\begin{align}
\theta_2(\beta)
&:=
1-e^{-\ell_2(\beta)}
=
\frac{J^3-J^4}{1-J^4}.
\label{eq:theta2-def}
\end{align}
Then the homogeneous variational function \eqref{eq:homogeneous-wetting-variational} can be written as
\begin{align}
\overline{F}_{\mathrm{homo}}(\beta,u)
&=
\max_{m\in\mathbb Z_+}
\left\{
a_1(\beta)J^{2m}u
-
2a_2(\beta)\theta_2(\beta)J^{3m}
\right\}.
\label{eq:homogeneous-variational-theta2}
\end{align}
The following lemma identifies the scale of the maximizing height.

\begin{lemma}[Localization window]
\label{lem:localization-window}
Let $n_u$ be a maximizer in \eqref{eq:homogeneous-variational-theta2}.
There exist constants $0<c_\beta<C_\beta<\infty$ such that, for all sufficiently small $u>0$,
\begin{align}
c_\text{loc}(\beta) u
\le
J^{n_u}
\le
C_\text{loc}(\beta) u.
\label{eq:localization-window}
\end{align}
\end{lemma}

\begin{proof}[Proof of Lemma ~\ref{lem:localization-window}]
Choose $\rho_\beta>0$ sufficiently small that $2a_2(\beta)\theta_2(\beta)\rho_\beta\le a_1(\beta)/2$, and, for each sufficiently small $u$, choose $m=m(u)\in\mathbb Z_+$ such that $J\rho_\beta u<J^m\le\rho_\beta u$.
By \eqref{eq:homogeneous-variational-theta2},
\begin{align}
\overline{F}_{\mathrm{homo}}(\beta,u)
&\ge
a_1(\beta)uJ^{2m}
-
2a_2(\beta)\theta_2(\beta)J^{3m}
\notag\\
&\ge
\frac{a_1(\beta)}{2}uJ^{2m}
\notag\\
&\ge
\frac{a_1(\beta)}{2}J^2\rho_\beta^2u^3.
\label{eq:positive-variational-lower-bound}
\end{align}

If $J^n> a_1(\beta)u/ \bigl(2a_2(\beta)\theta_2(\beta)\bigr)$, then the corresponding variational term is negative.
Since the maximum is positive by \eqref{eq:positive-variational-lower-bound}, every maximizer satisfies
\begin{align*}
J^{n_u}
\le
\frac{a_1(\beta)}
{2a_2(\beta)\theta_2(\beta)}u.
\end{align*}

For the lower bound, choose $c_\beta>0$ sufficiently small that $c_\beta<J\rho_\beta/\sqrt{2}$.
If $J^n<c_\beta u$, then
\begin{align*}
a_1(\beta)uJ^{2n}
-
2a_2(\beta)\theta_2(\beta)J^{3n}
&\le
a_1(\beta)c_\beta^2u^3
<
\overline{F}_{\mathrm{homo}}(\beta,u),
\end{align*}
where the last inequality follows from \eqref{eq:positive-variational-lower-bound}.
Hence such an $n$ cannot be a maximizer.
\end{proof}

Fix constants $c_\beta$ and $C_\beta$ as in Lemma~\ref{lem:localization-window}.
In view of the variational formula, it is now sufficient to prove, uniformly over integers $n$ satisfying
\begin{align}
c_\text{loc}(\beta) u
\le
J^n
\le
C_\text{loc}(\beta) u,
\label{eq:admissible-localization-window}
\end{align}
that
\begin{align}
\overline{F}_\beta(\alpha,u)
&\ge
a_1(\beta)uJ^{2n}
-
2a_2(\beta)\theta_2(\beta)J^{3n}
-
o(u^3).
\label{eq:fixed-height-sharp-lower-bound}
\end{align}
Indeed, choosing $n=n_u$ and using \eqref{eq:homogeneous-variational-theta2} then gives Proposition~\ref{prop:sharp-lower-bound}.

From now on, we fix an integer $n=n(u)$ satisfying \eqref{eq:admissible-localization-window}.
In particular,
\begin{align}
J^n
&\asymp u,
\qquad
n
=
\frac{1}{2\beta}
\log\frac{1}{u}
+
O_\beta(1).
\label{eq:localization-height-scale}
\end{align}

\subsection{Good-contour restriction and coarse-graining}
\label{subsec:4.2}

We now introduce the coarse-grained decomposition used in the lower-bound argument, following the construction of \cite[Section~4.2]{lacoin2021solid}.
Recall from \eqref{eq:localization-height-scale} that $J^n\asymp u$.
We first restrict the SOS field to configurations containing no long contours; the corresponding free-energy cost is negligible at order $u^3$.
We then condition on the external contours intersecting a grid of mesh size $M$.
Once these contours are fixed, the SOS field decomposes into conditionally independent fields on the resulting regions.

This is the same conditional-independence mechanism as in the disordered pinning setting of \cite{lacoin2021solid}.
In the present wetting model, however, this does not by itself yield a factorization of the interaction: the shifted wall penalty depends on connected components of the contact set, which may intersect more than one coarse-grained region.
This additional difficulty will be addressed in Subsection~\ref{subsec:4.4}.

Recall that $\Upsilon_n(\phi)$ denotes the collection of signed contours of $\phi$ with boundary condition $n$.
For a contour $\gamma$, we write $\widetilde\gamma$ for its geometric contour and $\overline\gamma$ for its interior.
Define
\begin{align}
A_{N,n}
&:=
\left\{
|\widetilde\gamma|\le n^4
\text{ for every }
\gamma\in\Upsilon_n(\phi)
\right\}.
\label{eq:good-event}
\end{align}
We denote the SOS measure conditioned on $A_{N,n}$ by
\begin{align*}
    \widetilde{\P}_{N,\beta}^n(\,\cdot\,)
    &:=
    \P_{N,\beta}^n
    \bigl(
        \,\cdot\mid A_{N,n}
    \bigr),
\end{align*}
and write $\widetilde{\E}_{N,\beta}^n$ for the corresponding expectation.

\begin{lemma} \cite[Lemma~4.2]{lacoin2021solid}
\label{lem:good-event}
Assume that $\beta>2\log 3$.
There exist constants $c=c(\beta)>0$ and $C=C(\beta)<\infty$ such that, for every $N,n\ge1$,
\begin{align}
\P_{N,\beta}^n(A_{N,n})
&\ge
\exp\left\{
-CN^2e^{-cn^4}
\right\}.
\label{eq:good-event-probability}
\end{align}
\end{lemma}

\begin{proof}[Proof of Lemma ~\ref{lem:good-event}]
This is \cite[Lemma~4.2]{lacoin2021solid}.
\end{proof}

Consequently, for every nonnegative random weight $G^\omega(\phi)$,
\begin{align}
\E\log
\E_{N,\beta}^n
\left[
G^\omega(\phi)
\right]
&\ge
-CN^2e^{-cn^4}
+
\E\log
\widetilde{\E}_{N,\beta}^n
\left[
G^\omega(\phi)
\right].
\label{eq:restriction-to-good-event}
\end{align}
Since $n\asymp|\log u|$ by \eqref{eq:localization-height-scale}, we have $e^{-cn^4}=o(u^p)$ for every $p>0$.
Hence the first term on the right-hand side of \eqref{eq:restriction-to-good-event} is $o(N^2u^3)$, uniformly over the localization window.

We next introduce the coarse-graining grid.
Fix
\begin{align}
M
&:=
\lfloor u^{-\chi}\rfloor,
\qquad
0<\chi<\frac{1}{50}.
\label{eq:coarse-graining-scale}
\end{align}
Since the free-energy limit exists, it is enough to let $N\rightarrow\infty$ along multiples of $M$.
We therefore write $N=kM$ for some integer $k\ge1$.

Let
\begin{align}
H_M
&:=
\left\{
x^*=(x_1^*,x_2^*)\in(\Z^2)^*:
x_1^*-\frac12\in M\Z
\text{ or }
x_2^*-\frac12\in M\Z
\right\}
\label{eq:dual-grid}
\end{align}
be the dual grid of mesh size $M$.
It divides $\Lambda_N$ into the deterministic square cells
\begin{align}
C_z
&:=
Mz+\Lambda_M,
\qquad
z\in\llbracket0,k-1\rrbracket^2.
\label{eq:deterministic-cells}
\end{align}
Let $\Upsilon_n^{\mathrm{ext}}(\phi)$ denote the collection of external contours, namely those contours whose interiors are maximal under inclusion.
We define the coarse-graining datum by
\begin{align}
U(\phi)
&:=
\left\{
\gamma\in\Upsilon_n^{\mathrm{ext}}(\phi):
\widetilde\gamma\cap H_M\ne\varnothing
\right\}.
\label{eq:grid-crossing-contours}
\end{align}
Let
\begin{align}
    \mathcal U_{N,n,M}
    &:=
    \left\{
        \Gamma:
        \widetilde{\P}_{N,\beta}^n
        \bigl(
            U=\Gamma
        \bigr)>0
    \right\}.
    \label{eq:admissible-coarse-grainings}
\end{align}
be the collection of admissible realizations of the coarse-graining datum.
Fix $\Gamma\in\mathcal U_{N,n,M}$.

Since the contours in $\Gamma$ are external, their interiors are pairwise disjoint.
We remove these interiors from the deterministic cells and define
\begin{align}
B_z
&:=
C_z
\setminus
\bigcup_{\gamma\in\Gamma}
\overline\gamma.
\label{eq:random-cell-region}
\end{align}
The sets $B_z$ need not be connected.
Together with the contour interiors, however, they form the disjoint decomposition
\begin{align}
\Lambda_N
&=
\left(
\mathop{\dot\bigcup}_{
z\in\llbracket0,k-1\rrbracket^2
}
B_z
\right)
\mathbin{\dot\cup}
\left(
\mathop{\dot\bigcup}_{\gamma\in\Gamma}
\overline\gamma
\right).
\label{eq:region-decomposition}
\end{align}

We denote the corresponding collection of regions by
\begin{align}
\mathcal R_{\Gamma}
&:=
\left\{
B_z:
z\in\llbracket0,k-1\rrbracket^2
\right\}
\cup
\left\{
\overline\gamma:
\gamma\in\Gamma
\right\}.
\label{eq:region-family}
\end{align}

For $R\in\mathcal R_{\Gamma}$, let
\begin{align}
\nu_R(\,\cdot\,)
&:=
\widetilde{\P}_{N,\beta}^n
\left(
\phi|_R\in\,\cdot\,
\,\middle|\,
U=\Gamma
\right).
\label{eq:conditional-region-law}
\end{align}
For cell regions we write
\begin{align}
    \nu_z:=\nu_{B_z},    
\end{align}\label{eq:defnu_z}
while for contour regions we write $\nu_\gamma:=\nu_{\overline\gamma}$.

Once the grid-crossing external contours are fixed, every remaining contour is contained in exactly one region of $\mathcal R_{\Gamma}$.
Moreover, conditioning on the good event $A_{N,n}$ preserves the product structure across these regions.
Consequently, by the contour representation and the additivity of the contour energy, for every collection of bounded measurable functions $(F_R)_{R\in\mathcal R_{\Gamma}}$,
\begin{align}
&
\widetilde{\E}_{N,\beta}^n
\left[
\prod_{R\in\mathcal R_{\Gamma}}
F_R(\phi|_R)
\,\middle|\,
U=\Gamma
\right]
\nonumber\\
&
=
\prod_{R\in\mathcal R_{\Gamma}}
\nu_R[F_R].
\label{eq:conditional-factorization}
\end{align}
This is the factorization underlying \cite[(4.6)--(4.8)]{lacoin2021solid}.

The conditional factorization above concerns the SOS measure.
The wetting interaction itself remains nonlocal across the regions because of the shifted wall penalty.
We next record the geometry of the coarse-grained regions before resolving this interaction in Subsection~\ref{subsec:4.4}.

\subsection{Geometry of the coarse-grained regions}
\label{subsec:4.3}

We now record the geometric properties of the coarse-grained regions introduced in Subsection~\ref{subsec:4.2}.
Recall that $M=\lfloor u^{-\chi}\rfloor$, with $0<\chi<1/50$, and that we consider $N=kM$.
We use throughout the notation introduced there: $H_M$ denotes the dual grid of mesh size $M$, $C_z$ the deterministic square cells, and $\Gamma\in\mathcal U_{N,n,M}$ an admissible realization of the coarse-graining datum.
The corresponding regions are
\begin{align*}
    \mathcal R_{\Gamma}
    &=
    \{B_z\}_z
    \cup
    \left\{
        \overline\gamma:
        \gamma\in\Gamma
    \right\}.
\end{align*}

For a finite set $R\subset\Lambda_N$, let
\begin{align*}
    E(R)
    &:=
    \left\{
        \{x,y\}\subset R:
        x\sim y
    \right\}.
\end{align*}
We define the set of edges whose endpoints belong to distinct regions of $\mathcal R_{\Gamma}$ by
\begin{align}
    E_{\mathrm{cr}}(\Gamma)
    &:=
    E(\Lambda_N)\;
    \setminus
    \bigcup_{R\in\mathcal R_{\Gamma}}
    E(R).
    \label{eq:cross-edges}
\end{align}
We define the number of occupied cross-edges by
\begin{align}
    X_{\mathrm{cr}}(\phi)
    &:=
    \sum_{\{x,y\}\in E_{\mathrm{cr}}(\Gamma)}
    \mathbf 1_{\{\phi(x)\le0,\ \phi(y)\le0\}}.
    \label{eq:occupied-cross-edges}
\end{align}

The sharp one-site and two-site contributions will be extracted only from a deterministic core of each cell.
We define
\begin{align}
    C_z^\circ
    &:=
    \left\{
        x\in C_z:
        \operatorname{dist}(x,\partial C_z)\ge4n^4
    \right\}.
    \label{eq:deterministic-core}
\end{align}
The width $4n^4$ is chosen so that the core, together with its nearest-neighbor enlargement, remains separated from every contour intersecting the coarse-graining grid.
This will ensure that the local events used later to identify exact singleton and two-site components are measurable with respect to the corresponding cell region $B_z$.

By \eqref{eq:localization-height-scale} and \eqref{eq:coarse-graining-scale},
\begin{align*}
    n
    &=
    \frac{1}{2\beta}
    \log\frac{1}{u}
    +
    O_\beta(1),
    \qquad
    \frac{M}{n^4}
    \longrightarrow\infty
    \quad\text{as }u\rightarrow 0+.
\end{align*}
We henceforth assume that $u$ is sufficiently small that
\begin{align*}
    M
    &\ge16n^4.
\end{align*}

The grid-crossing contours are short and intersect $H_M$, and are therefore confined to an $O(n^4)$-neighborhood of the coarse-graining grid.
The following lemma records the geometric estimates needed in the sequel.
Its proof is deferred to Appendix~\ref{app:coarse-grained-geometry}.

\begin{lemma}[Geometry of the coarse-grained regions]
\label{lem:coarse-grained-geometry}
There exists a constant $C<\infty$ such that, uniformly over $\Gamma\in\mathcal U_{N,n,M}$, the following estimates hold.

For every cell $C_z$,
\begin{align}
\left\{
y\in C_z:
\operatorname{dist}(y,C_z^\circ)\le1
\right\}
&\subset
B_z,
\label{eq:core-nearest-neighbor-enlargement}
\\
|C_z^\circ|
&=
M^2+O(Mn^4),
\\
|E(C_z^\circ)|
&=
2M^2+O(Mn^4),
\label{eq:core-cardinality-estimates}
\\
\left|
E(B_z)\setminus E(C_z^\circ)
\right|
&\le
CMn^4.
\label{eq:cell-boundary-edge-estimate}
\end{align}

Moreover, for every $\gamma\in\Gamma$,
\begin{align}
|\overline\gamma|
&\le
Cn^8,
\label{eq:single-contour-region-size}
\end{align}
and
\begin{align}
\sum_{\gamma\in\Gamma}
|\overline\gamma|
&\le
C\frac{N^2n^4}{M},
\label{eq:total-contour-volume}
\\
|E_{\mathrm{cr}}(\Gamma)|
&\le
C\frac{N^2n^4}{M}.
\label{eq:cross-edge-count}
\end{align}
\end{lemma}

\begin{proof}[Proof of Lemma~\ref{lem:coarse-grained-geometry}]
The proof is given in Appendix~\ref{app:coarse-grained-geometry}.
\end{proof}

We conclude by recording the scale consequences of the preceding lemma.
By \eqref{eq:localization-height-scale} and \eqref{eq:coarse-graining-scale},
\begin{align}
    \frac{n^4}{M}
    &=
    O\left(
        |\log u|^4u^\chi
    \right)
    =
    o(1).
    \label{eq:coarse-graining-boundary-ratio}
\end{align}
Furthermore, since $J^n\asymp u$,
\begin{align}
    &\left(
        \sum_{\gamma\in\Gamma}
        |\overline\gamma|
        +
        |E_{\mathrm{cr}}(\Gamma)|
        +
        \sum_z
        \left|
            E(B_z)\setminus E(C_z^\circ)
        \right|
    \right)
    J^{3n}
    \nonumber\\
    &\qquad\le
    C\frac{N^2n^4}{M}J^{3n}
    \nonumber\\
    &\qquad=
    O\left(
        N^2|\log u|^4u^{3+\chi}
    \right)
    =
    o(N^2u^3).
    \label{eq:exceptional-regions-negligible}
\end{align}
Thus the leading-order contribution can be extracted from the deterministic cores $C_z^\circ$, while cell-boundary edges, contour regions, and cross-edges may be treated by rough estimates without affecting the sharp $N^2u^3$ scale.
In the next subsection, we use this geometric decomposition to rewrite the shifted wall penalty in terms of local singleton, pair, and occupied-edge observables.

\subsection{Decomposition of the wall interaction}
\label{subsec:4.4}

We now rewrite the wall interaction in a form compatible with the coarse-grained decomposition.
The leading contribution inside each cell will be extracted from exact singleton and exact two-site components contained in the deterministic core $C_z^\circ$.
All other non-singleton components will be controlled through the number of occupied edges.

Recall the contact set $\phi^{-1}(\mathbb Z_-)$.
For a cell $C_z$, define the number of exact singleton components contained in its deterministic core by
\begin{align}
    I_{1,z}(\phi)
    &:=
    \sum_{x\in C_z^\circ}
    \mathbf 1_{
        \left\{
            \{x\}
            \text{ is a connected component of }\phi^{-1}(\mathbb Z_-)
        \right\}
    }.
    \label{eq:cell-singleton-count}
\end{align}
Similarly, define the number of exact two-site components contained in the core by
\begin{align}
    I_{2,z}(\phi)
    &:=
    \sum_{\{x,y\}\in E(C_z^\circ)}
    \mathbf 1_{
        \left\{
            \{x,y\}
            \text{ is a connected component of }\phi^{-1}(\mathbb Z_-)
        \right\}
    }.
    \label{eq:cell-pair-count}
\end{align}
Thus, $I_{2,z}$ counts two-site components rather than sites.

We also introduce the total number of occupied edges in the cell region $B_z$:
\begin{align}
    Q_z(\phi)
    &:=
    \sum_{\{x,y\}\in E(B_z)}
    \mathbf 1_{\{\phi(x)\le0, \phi(y)\le0\}}.
    \label{eq:cell-occupied-edge-count}
\end{align}
The occupied edges belonging to exact pair components counted by $I_{2,z}$ are separated from the remaining occupied edges by setting
\begin{align}
    R_z(\phi)
    &:=
    Q_z(\phi)-I_{2,z}(\phi).
    \label{eq:cell-edge-remainder}
\end{align}
Every exact two-site component counted by $I_{2,z}$ contributes exactly one occupied edge to $Q_z$.
Consequently,
\begin{align*}
    R_z(\phi)
    &\ge0.
\end{align*}

For the cutoff used later in the cell expansion, define
\begin{align}
    D_z^\kappa
    &:=
    \left\{
        \sum_{x\in B_z}
    \mathbf 1_{\{\phi(x)\le0\}}\le\kappa
    \right\}.
    \label{eq:cell-contact-count-and-cutoff}
\end{align}
The cutoff is an auxiliary restriction whose probability cost will be shown to be negligible.
Besides simplifying the deterministic cell expansion, it bounds the contact overlap of two replicas by $\kappa$.
This allows the exponential overlap term in the disorder second-moment estimate to be controlled by the expected overlap.
Here $\kappa$ is a fixed sufficiently large integer; eventually we take $\kappa=100$.

By \eqref{eq:core-nearest-neighbor-enlargement}, all sites needed to determine whether a vertex in $C_z^\circ$ is an exact singleton, or whether an edge in $E(C_z^\circ)$ is an exact two-site component, belong to $B_z$.
Consequently,
\begin{align*}
    I_{1,z},\quad
    I_{2,z},\quad
    Q_z,\quad
    R_z,\quad
    D_z^\kappa
\end{align*}
are measurable with respect to the restriction $\phi|_{B_z}$.

Recall that $\ell_2(\beta)$ denotes the shifted wall penalty of a nearest-neighbor pair; see \eqref{eq:Hbar-doubleton}.
We set
\begin{align}
    C_\star(\beta)
    &:=
    \max\left\{
        2c_2(\beta),
        \ell_2(\beta)
    \right\},
    \label{eq:coarse-graining-penalty-constant}
\end{align}
where $c_2(\beta)$ is the constant appearing in \eqref{eq:Hbar-linear-bound}.
The choice of $C_\star(\beta)$ ensures that a rough penalty per occupied edge dominates the wall penalty of all non-singleton components that will not be treated sharply.
For a cell $B_z$, define the sharp local potential
\begin{align}
V_z(\phi)
&:=
uI_{1,z}(\phi)
-
\ell_2(\beta)I_{2,z}(\phi)
-
C_\star(\beta)R_z(\phi).
\label{eq:sharp-cell-potential}
\end{align}

\begin{lemma}[Coarse-grained decomposition of the wall interaction]
\label{lem:sharp-local-coarse-graining}
Fix $u\ge0$ and $\Gamma\in\mathcal U_{N,n,M}$.
On the event $A_{N,n}\cap\{U=\Gamma\}$,
\begin{align}
S_N^\omega(\phi)
\ge{}&
\sum_z
\left[
\sum_{x\in B_z}
\eta_x^{(\alpha)}
\mathbf 1_{{\phi(x)\le0}}
+
V_z(\phi)
\right]
\nonumber\\
&+
\sum_{\gamma\in\Gamma}
\left[
\sum_{x\in\overline\gamma}
\bigl(
u+\eta_x^{(\alpha)}
\bigr)
\mathbf 1_{{\phi(x)\le0}}
-
C_\star(\beta)
\sum_{\{x,y\}\in E(\overline\gamma)}
\mathbf 1_{{\phi(x)\le0,\phi(y)\le0}}
\right]
\nonumber\\
&-
C_\star(\beta)X_{\mathrm{cr}}(\phi).
\label{eq:sharp-local-coarse-graining}
\end{align}
\end{lemma}

\begin{proof}[Proof of Lemma ~\ref{lem:sharp-local-coarse-graining}]
By the additivity of $\overline H_\beta$ over the connected components of $\phi^{-1}(\mathbb Z_-)$, it is enough to compare the two sides of \eqref{eq:sharp-local-coarse-graining} separately for each connected component $\mathcal C$ of $\phi^{-1}(\mathbb Z_-)$.

If $\mathcal C=\{x\}$, then $\overline H_\beta(\mathcal C)=0$ by \eqref{eq:Hbar-singleton}.
If $x\in C_z^\circ$, the full contribution $u+\eta_x^{(\alpha)}$ is retained through $I_{1,z}$; if $x\in\overline\gamma$ for some $\gamma\in\Gamma$, it is retained by the contour-region term.
If $x\in B_z\setminus C_z^\circ$, only $\eta_x^{(\alpha)}$ is retained.
In every case the discarded part is nonnegative since $u\ge0$.

Suppose next that $\mathcal C=\{x,y\}$ is an exact two-site component contained in some $C_z^\circ$.
Then it contributes one unit to $I_{2,z}$ and one occupied edge to $Q_z$, so its contribution to $R_z=Q_z-I_{2,z}$ vanishes.
The right-hand side retains
\begin{align*}
\eta_x^{(\alpha)}
+
\eta_y^{(\alpha)}
-
\ell_2(\beta),
\end{align*}
whereas, by \eqref{eq:Hbar-doubleton}, its exact contribution to $S_N^\omega(\phi)$ is larger by $2u$.

It remains to consider a connected component $\mathcal C$ with $|\mathcal C|\ge2$ which is not an exact two-site component contained in a deterministic core.
Since $\mathcal C$ is connected,
\begin{align*}
|E(\mathcal C)|
&\ge
|\mathcal C|-1
\ge
\frac{|\mathcal C|}{2}.
\end{align*}
Hence, by \eqref{eq:Hbar-linear-bound} and \eqref{eq:coarse-graining-penalty-constant},
\begin{align}
\overline H_\beta(\mathcal C)
&\le
c_2(\beta)|\mathcal C|
\le
2c_2(\beta)|E(\mathcal C)|
\le
C_\star(\beta)|E(\mathcal C)|.
\label{eq:component-wall-penalty-by-edges}
\end{align}

No occupied edge of such a component is subtracted through $I_{2,z}$.
Each occupied edge of $\mathcal C$ is therefore accounted for exactly once: by $R_z$ if both endpoints belong to the same cell region, by the contour-region edge penalty if both endpoints belong to the same contour region, and by $X_{\mathrm{cr}}$ if its endpoints belong to distinct regions.
Likewise, since $\mathcal R_\Gamma$ is a partition of $\Lambda_N$, each disorder reward $\eta_x^{(\alpha)}$ is retained exactly once.

Finally, the total $u$-reward retained on the right-hand side is at most $u|\mathcal C|$.
Therefore the contribution retained for $\mathcal C$ is at most
\begin{align*}
\sum_{x\in\mathcal C}
\bigl(
u+\eta_x^{(\alpha)}
\bigr)
-
C_\star(\beta)|E(\mathcal C)|
&\le
\sum_{x\in\mathcal C}
\bigl(
u+\eta_x^{(\alpha)}
\bigr)
-
\overline H_\beta(\mathcal C),
\end{align*}
where the last inequality follows from \eqref{eq:component-wall-penalty-by-edges}.
Summing over the connected components completes the proof.
\end{proof}

The preceding pointwise estimate separates the sharp contribution coming from singleton and exact two-site components in the deterministic cell cores from the rough contributions associated with the remaining occupied edges and the contour regions.
In the next subsection, we combine this decomposition with the conditional factorization \eqref{eq:conditional-factorization} to obtain a corresponding lower bound for the finite-volume quenched partition function.

\subsection{Coarse-grained lower-bound decomposition}
\label{subsec:4.5}

We now combine the pointwise decomposition of Lemma~\ref{lem:sharp-local-coarse-graining} with the conditional factorization \eqref{eq:conditional-factorization}.
This yields a finite-volume lower bound in terms of sharp cell contributions, rough contour-region contributions, and the remaining cross-edge interaction.

Fix $\Gamma\in\mathcal U_{N,n,M}$ and recall the conditional region laws $\nu_z$ and $\nu_\gamma$ from Subsection~\ref{subsec:4.2}, as well as the sharp local potential $V_z$ from \eqref{eq:sharp-cell-potential}.

The local terms in Lemma~\ref{lem:sharp-local-coarse-graining} are treated differently in cell regions and contour regions.
In the cell regions, we retain the singleton and exact-pair contributions sharply, together with the cutoff $D_z^\kappa$.
We therefore define
\begin{align}
Z_{B_z}^{\omega,\sharp}(\Gamma)
&:=
\nu_z
\left[
\mathbf 1_{D_z^\kappa}(\phi)
\exp\left\{
V_z(\phi)
+
\sum_{x\in B_z}
\eta_x^{(\alpha)}
\mathbf 1_{{\phi(x)\le0}}
\right\}
\right].
\label{eq:sharp-cell-partition-function}
\end{align}

In the contour regions, no sharp identification of singleton and pair components is needed.
We keep only the contact reward and a rough penalty for occupied edges.
Thus, for $\gamma\in\Gamma$, set
\begin{align}
Z_{\overline\gamma}^{\omega,\mathrm{rough}}(\Gamma)
&:=
\nu_\gamma
\left[
\exp\left\{
\sum_{x\in\overline\gamma}
\bigl(
u+\eta_x^{(\alpha)}
\bigr)
\mathbf 1_{{\phi(x)\le0}}
-
C_\star(\beta)
\sum_{\{x,y\}\in E(\overline\gamma)}
\mathbf 1_{{\phi(x)\le0,\phi(y)\le0}}
\right\}
\right].
\label{eq:rough-contour-region-partition-function}
\end{align}

The corresponding normalized measures describe the local SOS laws tilted by these interaction weights.
For a cell $B_z$, let
\begin{align}
\mu_z^{\omega,\sharp}(d\phi_{B_z})
&:=
\frac{
\mathbf 1_{D_z^\kappa}(\phi)
\exp\left\{
V_z(\phi)
+
\sum_{x\in B_z}
\eta_x^{(\alpha)}
\mathbf 1_{\{\phi(x)\le0\}}
\right\}
}{
Z_{B_z}^{\omega,\sharp}(\Gamma)
}
\nu_z(d\phi_{B_z}).
\label{eq:sharp-cell-random-tilted-measure}
\end{align}
For a contour region $\overline\gamma$, define
\begin{align}
\mu_{\overline\gamma}^{\omega,\mathrm{rough}}
(d\phi_{\overline\gamma})
&:=
\frac{
\exp\left\{
\sum_{x\in\overline\gamma}
\bigl(
u+\eta_x^{(\alpha)}
\bigr)
\mathbf 1_{{\phi(x)\le0}}
-
C_\star(\beta)
\sum_{\{x,y\}\in E(\overline\gamma)}
\mathbf 1_{{\phi(x)\le0,\phi(y)\le0}}
\right\}
}{
Z_{\overline\gamma}^{\omega,\mathrm{rough}}(\Gamma)
}
\nu_\gamma(d\phi_{\overline\gamma}).
\label{eq:rough-contour-random-tilted-measure}
\end{align}

Since the conditional SOS law factorizes over the coarse-grained regions, these tilted local measures can be combined into the product measure
\begin{align}
\mu_{\Gamma}^{\omega,\mathrm{mix}}
&:=
\left(
\bigotimes_z
\mu_z^{\omega,\sharp}
\right)
\otimes
\left(
\bigotimes_{\gamma\in\Gamma}
\mu_{\overline\gamma}^{\omega,\mathrm{rough}}
\right).
\label{eq:mixed-product-tilted-measure}
\end{align}
The only interaction not absorbed into these local weights is the cross-edge term $X_{\mathrm{cr}}$, which will be treated separately under $\mu_{\Gamma}^{\omega,\mathrm{mix}}$.

The following proposition is the main output of the coarse-graining reduction.

\begin{proposition}[Coarse-grained lower-bound decomposition]
\label{prop:coarse-grained-lower-bound-decomposition}
For every sufficiently small $u>0$, every localization height $n$ satisfying \eqref{eq:admissible-localization-window}, and every $N$ divisible by $M$, we have
\begin{align}
    \E\log Z_{N,n}^\omega(u)
    \ge{}&
    -
    CN^2e^{-cn^4}+\\
    &\min_{
        \Gamma
        \in
        \mathcal U_{N,n,M}
    }
    \Bigg\{
        \sum_z
        \E\log
        Z_{B_z}^{\omega,\sharp}
        \bigl(
            \Gamma
        \bigr)
        +
        \sum_{\gamma\in\Gamma}
        \E\log
        Z_{\overline\gamma}^{\omega,\mathrm{rough}}
        \bigl(
            \Gamma
        \bigr)
        -
        C_\star(\beta)
        \E
        \left[
            \mu_{\Gamma}^{\omega,\mathrm{mix}}
            \left[
                X_{\mathrm{cr}}
            \right]
        \right]
    \Bigg\}.
    \label{eq:sharp-coarse-grained-decomposition}
\end{align}
\end{proposition}

\begin{proof}[Proof of Proposition ~\ref{prop:coarse-grained-lower-bound-decomposition}]
Applying \eqref{eq:restriction-to-good-event} with $G^\omega(\phi)=\exp{S_N^\omega(\phi)}$ gives
\begin{align}
\E\log Z_{N,n}^\omega(u)
\ge
-CN^2e^{-cn^4}
+
\E\log
\widetilde{\E}_{N,\beta}^n
\left[
e^{S_N^\omega(\phi)}
\right].
\label{eq:good-event-restricted-partition-bound}
\end{align}

We next decompose according to the coarse-graining datum $U$:
\begin{align*}
\widetilde{\E}_{N,\beta}^n
\left[
e^{S_N^\omega(\phi)}
\right]
&=
\sum_{\Gamma\in\mathcal U_{N,n,M}}
\widetilde{\P}_{N,\beta}^n(U=\Gamma)
\widetilde{\E}_{N,\beta}^n
\left[
e^{S_N^\omega(\phi)}
\,\middle|\,
U=\Gamma
\right].
\end{align*}
Since the probabilities $\widetilde{\P}_{N,\beta}^n(U=\Gamma)$ are independent of the disorder, concavity of the logarithm yields
\begin{align}
\E\log
\widetilde{\E}_{N,\beta}^n
\left[
e^{S_N^\omega(\phi)}
\right]
&\ge
\min_{\Gamma\in\mathcal U_{N,n,M}}
\E\log
\widetilde{\E}_{N,\beta}^n
\left[
e^{S_N^\omega(\phi)}
\,\middle|\,
U=\Gamma
\right].
\label{eq:coarse-graining-minimum}
\end{align}

Fix $\Gamma\in\mathcal U_{N,n,M}$.
By Lemma~\ref{lem:sharp-local-coarse-graining}, we may further restrict each cell to $D_z^\kappa$.
Using the conditional factorization \eqref{eq:conditional-factorization} and the definitions of the local partition functions and tilted measures, we obtain
\begin{align}
&
\widetilde{\E}_{N,\beta}^n
\left[
e^{S_N^\omega(\phi)}
\,\middle|\,
U=\Gamma
\right]
\nonumber\\
&\qquad\ge
\left(
\prod_z
Z_{B_z}^{\omega,\sharp}(\Gamma)
\right)
\left(
\prod_{\gamma\in\Gamma}
Z_{\overline\gamma}^{\omega,\mathrm{rough}}(\Gamma)
\right)
\mu_{\Gamma}^{\omega,\mathrm{mix}}
\left[
e^{-C_\star(\beta)X_{\mathrm{cr}}}
\right].
\label{eq:sharp-weight-factorization}
\end{align}

Taking logarithms and applying Jensen's inequality to the last factor,
\begin{align*}
\log
\mu_{\Gamma}^{\omega,\mathrm{mix}}
\left[
e^{-C_\star(\beta)X_{\mathrm{cr}}}
\right]
&\ge
-C_\star(\beta)
\mu_{\Gamma}^{\omega,\mathrm{mix}}
\left[
X_{\mathrm{cr}}
\right].
\end{align*}
Therefore,
\begin{align*}
&
\E\log
\widetilde{\E}_{N,\beta}^n
\left[
e^{S_N^\omega(\phi)}
\,\middle|\,
U=\Gamma
\right]
\\
&\qquad\ge
\sum_z
\E\log
Z_{B_z}^{\omega,\sharp}(\Gamma)
+
\sum_{\gamma\in\Gamma}
\E\log
Z_{\overline\gamma}^{\omega,\mathrm{rough}}(\Gamma)
-
C_\star(\beta)
\E
\left[
\mu_{\Gamma}^{\omega,\mathrm{mix}}
\left[
X_{\mathrm{cr}}
\right]
\right].
\end{align*}
Combining this estimate with \eqref{eq:coarse-graining-minimum} and \eqref{eq:good-event-restricted-partition-bound} proves \eqref{eq:sharp-coarse-grained-decomposition}.
\end{proof}

Proposition~\ref{prop:coarse-grained-lower-bound-decomposition} reduces the global lower bound to three separate estimates: a sharp estimate for the cell partition functions, a rough lower bound for the contour-region partition functions, and an upper bound on the cross-edge error.
The first of these requires recovering the sharp one-site and two-site coefficients $a_1(\beta)$ and $a_2(\beta)$ under the conditioned cell law, whereas the latter two can be handled using the uniform contour-restricted peak estimates.
These estimates are carried out in Section~\ref{chap:5}.

\section{Sharp estimates and completion of the lower bound}\label{chap:5}

In this section, we estimate the local terms obtained from the coarse-grained reduction of Section~4 and complete the proof of the sharp lower bound.
The main task is to identify the leading contributions of isolated contacts and nearest-neighbor pairs inside the bulk of each coarse-grained cell.
To recover the sharp coefficients $a_1(\beta)$ and $a_2(\beta)$, we first compare the conditioned SOS law in the interior of a cell with the infinite-volume measure $\mathbb P_\beta$.
This allows us to derive sharp estimates for the numbers of exact singleton and two-site components.
We then use these estimates to obtain a sharp expansion of the cell partition function and show that the centered disorder contributes only at a smaller order.
Finally, we control the contour-region and cross-edge terms by rough estimates and combine all contributions to obtain the desired lower bound.

\subsection{Peak estimates in a conditioned cell}
\label{subsec:5.1}

The sharp cell contribution in Proposition~\ref{prop:coarse-grained-lower-bound-decomposition} requires precise estimates for the probabilities of singleton and nearest-neighbor pair components under the conditioned cell law $\nu_z$.

The rough estimates throughout the cell region $B_z$ follow from the contour-restricted peak estimates collected in Subsection~\ref{subsec:3.3}.
To recover the sharp coefficients $a_1(\beta)$ and $a_2(\beta)$ in the deterministic core $C_z^\circ$, we additionally compare fixed-radius local events under the conditioned cell law with their infinite-volume counterparts.
The required comparison is proved in Appendix~\ref{app:local-comparison}.

More precisely, if $\overline\nu_z$ denotes the law under $\nu_z$ of the reflected field
\begin{align*}
    \psi
    &:=
    n-\phi,
\end{align*}
then Lemma~\ref{lem:local-comparison-conditioned-cell} shows that, for every event $A_x=A_x^+\cap A_x^-$ supported in a fixed neighborhood of $x\in C_z^\circ$, where $A_x^+$ is increasing and $A_x^-$ is decreasing,
\begin{align}
    \left|
        \overline\nu_z(A_x)
        -
        \P_\beta(A_0)
    \right|
    &\le
    Ce^{-cn^2}.
\end{align}
Here $A_0$ denotes the translate of $A_x$ to a neighborhood of the origin.
Combining this comparison with the contour-restricted and infinite-volume peak estimates yields the following result.

For $x\in C_z^\circ$, let
\begin{align*}
S_x
&:=
\left\{
\{x\}
\text{ is a connected component of }
\phi^{-1}(\mathbb Z_-)
\right\},
\end{align*}
and, for $e=\{x,y\}\in E(C_z^\circ)$, let
\begin{align*}
P_e
&:=
\left\{
\{x,y\}
\text{ is a connected component of }
\phi^{-1}(\mathbb Z_-)
\right\}.
\end{align*}

\begin{proposition}[Peak estimates in a conditioned cell]
\label{prop:conditioned-cell-peak-estimates}
Assume that $\beta$ is sufficiently large.
Let $u>0$ be sufficiently small, let $n=n(u)$ satisfy \eqref{eq:admissible-localization-window}, and let $M$ be given by \eqref{eq:coarse-graining-scale}.
Uniformly in $z$, in positive integers $N$ divisible by $M$, and in $\Gamma\in\mathcal U_{N,n,M}$, the following estimates hold.

For every $x\in B_z$ and every $\{x,y\}\in E(B_z)$,
\begin{align}
\nu_z(\phi(x)\le0)
&\le
C_\beta J^{2n},
\qquad
\nu_z(\phi(x)\le0,\phi(y)\le0)
\le
C_\beta J^{3n}.
\label{eq:conditioned-cell-rough-one-two-point}
\end{align}
For every fixed $k\ge1$ and every collection of distinct sites $x_1,\ldots,x_k\in B_z$,
\begin{align}
\nu_z
\left(
\phi(x_i)\le0
\text{ for all }1\le i\le k
\right)
&\le
C_{\beta,k}e^{-3\beta n\sqrt{k}}.
\label{eq:conditioned-cell-multipoint}
\end{align}

Moreover, for $x\in C_z^\circ$ and $e=\{x,y\}\in E(C_z^\circ)$,
\begin{align}
\nu_z(S_x)
&\ge
a_1(\beta)J^{2n}
-
C_\beta
\left(
J^{3n}+e^{-c_\beta n^2}
\right),
\label{eq:conditioned-cell-exact-singleton}
\\
\nu_z(P_e)
&\le
a_2(\beta)J^{3n}
+
C_\beta
\left(
nJ^{4n}+e^{-c_\beta n^2}
\right),
\label{eq:conditioned-cell-exact-pair}
\\
\nu_z
\left(
\{\phi(x)\le0,\phi(y)\le0\}
\setminus P_e
\right)
&\le
C_\beta
\left(
nJ^{4n}+e^{-c_\beta n^2}
\right).
\label{eq:conditioned-cell-pair-in-larger-cluster}
\end{align}
\end{proposition}

\begin{proof}
Fix an admissible realization $\Gamma$, and write
\begin{align*}
\mathcal A_\Gamma
&:= A_{N,n}\cap\{U=\Gamma\},
&
\mu_\Gamma
&:=\P_{N,\beta}^n(\,\cdot\mid\mathcal A_\Gamma).
\end{align*}
The marginal of $\mu_\Gamma$ on $B_z$ is $\nu_z$.
We prove the rough estimates directly under $\mu_\Gamma$, without identifying a new SOS measure on the connected components of $B_z$.

\medskip
\noindent\emph{Step 1: A conditional cylinder-erasure bound.}
Every contour whose interior contains a site of $B_z$ is contained in that cell region and its geometric contour avoids $H_M$.
Indeed, its external ancestor cannot belong to $\Gamma$, since the interiors of the contours in $\Gamma$ have been removed from $B_z$.
That ancestor therefore avoids the grid and is contained in the cell $C_z$.
Externality of the contours in $\Gamma$ also forces its interior to be disjoint from every $\overline\gamma$, $\gamma\in\Gamma$.
The same conclusions hold for its descendants.

Consider a fixed compatible family of negative cylinders $\widehat F=\{(\delta_j,m_j):1\le j\le q\}$ such that $\overline\delta_j\subset B_z$ and $\widetilde\delta_j\cap H_M=\varnothing$ for every $j$.
On the event that these cylinders occur, erase them by setting
\begin{align*}
T_{\widehat F}\phi
&:=\phi+\sum_{j=1}^q m_j\mathbf 1_{\overline\delta_j}.
\end{align*}
By Lemma~\ref{lem:contour-cylinder-representation}, the cylinder collection of $T_{\widehat F}\phi$ is exactly the original collection with $\widehat F$ removed.
Consequently, the map is injective on this event and
\begin{align*}
\mathcal H_{\Lambda_N}^n(\phi)
-\mathcal H_{\Lambda_N}^n(T_{\widehat F}\phi)
&=\sum_{j=1}^q m_j|\widetilde\delta_j|.
\end{align*}
Erasure preserves $\mathcal A_\Gamma$.
It creates no long contour and leaves every contour in $\Gamma$ external.
Any newly exposed external contour is a descendant of an erased contour, and hence still avoids $H_M$.
Comparing Gibbs weights on $\mathcal A_\Gamma$ therefore gives
\begin{align*}
\mu_\Gamma
\left(\widehat F\subset\widehat\Upsilon_n(\phi)\right)
&\le
\exp\left\{-\beta\sum_{j=1}^q m_j|\widetilde\delta_j|\right\}.
\end{align*}
This estimate is uniform in $N$, $z$, and $\Gamma$.

\medskip
\noindent\emph{Step 2: The rough peak estimates.}
For a finite nonempty set $S\subset B_z$, let $\mathcal C_S^-$ denote all negative signed contours in $\mathbb Z^2$ whose interiors intersect $S$, and put $\ell_\delta:=|\widetilde\delta|$.
For finitely supported choices of integers $m_\delta\ge0$, set
\begin{align*}
h_v(\mathbf m)
&:=\sum_{\delta\in\mathcal C_S^-}
 m_\delta\mathbf 1_{\overline\delta}(v),
&
w(\mathbf m)
&:=\exp\left\{-\beta\sum_{\delta\in\mathcal C_S^-}
 m_\delta\ell_\delta\right\}.
\end{align*}
If $\phi(v)\le0$ for every $v\in S$, the total negative-cylinder intensity covering each such site is at least $n$, by \eqref{eq:cylinder-reconstruction}.
All these cylinders are of the type treated in Step~1.
A union bound, followed by dropping compatibility and spatial restrictions, yields
\begin{align*}
\nu_z\bigl(\phi(v)\le0\text{ for all }v\in S\bigr)
&\le T_S(n),
&
T_S(n)
&:=\sum_{\substack{\mathbf m:\ h_v(\mathbf m)\ge n\ \forall v\in S}}
 w(\mathbf m).
\end{align*}
Only one intensity is assigned to each signed contour in this sum.
The elementary contour-counting bound
\begin{align*}
\#\{\delta\in\mathcal C_S^-:\ell_\delta=\ell\}
&\le C|S|\ell^2 3^\ell
\end{align*}
will make the products below finite and uniform in the locations of $S$.
We increase the low-temperature threshold, if necessary, so that $\beta>4\log 3$.

For $S=\{x\}$, let $\delta_x$ be the unique negative contour of length four surrounding $x$.
Sum its intensity exactly, and use $e^{-4\beta(n-a)_+}\le e^{-4\beta n}e^{4\beta a}$ for the total intensity $a$ contributed by the other contours.
This gives
\begin{align*}
T_{\{x\}}(n)
&\le
\frac{e^{-4\beta n}}{1-e^{-4\beta}}
\prod_{\delta\in\mathcal C_{\{x\}}^-\setminus\{\delta_x\}}
\left(1-e^{-\beta(\ell_\delta-4)}\right)^{-1}
\le C_\beta e^{-4\beta n}.
\end{align*}
Here every remaining contour has length at least six, and $\ell_\delta-4\ge\ell_\delta/3$.

For $S=e=\{x,y\}$ with $x\sim y$, let $\delta_e$ be the unique negative contour of length six surrounding both sites and write $r_\delta:=|\overline\delta\cap e|$.
After fixing all other intensities, let $a$ and $b$ be their total contributions at $x$ and $y$.
Summing the intensity of $\delta_e$ requires a value at least $(n-\min\{a,b\})_+$.
Since
\begin{align*}
e^{-6\beta(n-\min\{a,b\})_+}
&\le e^{-6\beta n}e^{3\beta(a+b)},
\end{align*}
we obtain
\begin{align*}
T_e(n)
&\le
\frac{e^{-6\beta n}}{1-e^{-6\beta}}
\prod_{\delta\in\mathcal C_e^-\setminus\{\delta_e\}}
\left(1-e^{-\beta(\ell_\delta-3r_\delta)}\right)^{-1}
\le C_\beta e^{-6\beta n}.
\end{align*}
Indeed, $\ell_\delta-3r_\delta\ge\ell_\delta/4$ in this product: for $r_\delta=1$ the length is at least four, and for $r_\delta=2$ it is at least eight after $\delta_e$ is excluded.
These two bounds prove \eqref{eq:conditioned-cell-rough-one-two-point}, since $J=e^{-2\beta}$.

Finally, let $|S|=k$ and put $r_\delta:=|\overline\delta\cap S|$.
The planar isoperimetric inequality gives $\ell_\delta\ge4\sqrt{r_\delta}$.
Therefore
\begin{align*}
\ell_\delta-\frac{3r_\delta}{\sqrt{k}}
&\ge\frac{\ell_\delta}{4}.
\end{align*}
Since $\sum_{v\in S}h_v(\mathbf m)\ge kn$ on the summation event, an exponential bound with parameter $3\beta/\sqrt{k}$ gives
\begin{align*}
T_S(n)
&\le
 e^{-3\beta n\sqrt{k}}
 \prod_{\delta\in\mathcal C_S^-}
 \left(1-e^{-\beta(\ell_\delta-3r_\delta/\sqrt{k})}\right)^{-1}
\le C_{\beta,k}e^{-3\beta n\sqrt{k}}.
\end{align*}
This proves \eqref{eq:conditioned-cell-multipoint}.
No independence between connected components of $B_z$ is used.

\medskip
\noindent\emph{Step 3: Exact clusters in the deterministic core.}
Set $\psi=n-\phi$ on $B_z$.
For $x\in C_z^\circ$ and $e=\{x,y\}\in E(C_z^\circ)$, write
\begin{align*}
S_x
&=\{\psi(x)\ge n\}
  \cap\{\psi(v)<n\text{ for every }v\sim x\},\\
P_e
&=\{\psi(x)\ge n,\ \psi(y)\ge n\}
  \cap\{\psi(v)<n\text{ for every }v\in\partial\{x,y\}\}.
\end{align*}
By \eqref{eq:core-nearest-neighbor-enlargement}, all sites appearing here belong to $B_z$.
Each event is an intersection of an increasing and a decreasing event, supported within distance two of $x$.
Lemma~\ref{lem:local-comparison-conditioned-cell} therefore compares its probability under the reflected cell law with its probability under $\P_\beta$, with error at most $C_\beta e^{-c_\beta n^2}$.
Combining this with \eqref{eq:infinite-volume-exact-singleton} and \eqref{eq:infinite-volume-exact-pair} from Corollary~\ref{cor:exact-peak-clusters} proves \eqref{eq:conditioned-cell-exact-singleton} and \eqref{eq:conditioned-cell-exact-pair}.

For the last estimate, observe that
\begin{align*}
\{\phi(x)\le0,\phi(y)\le0\}\setminus P_e
&=\{\psi(x)\ge n,\ \psi(y)\ge n\}
\cap
\bigcup_{v\in\partial\{x,y\}}\{\psi(v)\ge n\}.
\end{align*}
This is an increasing event with the same fixed-radius support.
Applying Lemma~\ref{lem:local-comparison-conditioned-cell} and then \eqref{eq:infinite-volume-pair-in-larger-cluster} proves \eqref{eq:conditioned-cell-pair-in-larger-cluster}.
All constants are uniform over the choices specified in the statement.
\end{proof}

\subsection{Sharp cluster estimates in a cell}
\label{subsec:5.2}

The purpose of this subsection is to convert the pointwise estimates of Proposition~\ref{prop:conditioned-cell-peak-estimates} into estimates for the cluster observables appearing in the sharp cell potential.
In particular, we estimate the expected numbers of exact singleton and two-site components and show that all remaining occupied edges give a lower-order contribution.
We also verify that the contact cutoff introduced in \eqref{eq:cell-contact-count-and-cutoff} has negligible probability.

\begin{lemma}[Sharp cluster-count estimates in a cell]
\label{lem:sharp-cluster-counts-cell}
Assume that $\beta$ is sufficiently large.
Recall that the coarse-graining scale is
\begin{align*}
    M
    &=
    \lfloor u^{-\chi}\rfloor,
    \qquad
    0<\chi<\frac1{50},
\end{align*}
as defined in \eqref{eq:coarse-graining-scale}.
Then, for all sufficiently small $u>0$, uniformly over integers $n=n(u)$ satisfying \eqref{eq:admissible-localization-window}, over the cell $z$, over $N$, and over $\Gamma\in\mathcal U_{N,n,M}$, there exist constants $c_\beta,C_\beta>0$ such that
\begin{align}
    \nu_z(I_{1,z})
    &\ge
    a_1(\beta)M^2J^{2n}
    -
    C_\beta
    \left(
        Mn^4J^{2n}
        +
        M^2J^{3n}
        +
        M^2e^{-c_\beta n^2}
    \right),
    \label{eq:sharp-singleton-count-cell}
    \\
    \nu_z(I_{2,z})
    &\le
    2a_2(\beta)M^2J^{3n}
    +
    C_\beta
    \left(
        Mn^4J^{3n}
        +
        M^2nJ^{4n}
        +
        M^2e^{-c_\beta n^2}
    \right),
    \label{eq:sharp-pair-count-cell}
    \\
    \nu_z(R_z)
    &\le
    C_\beta
    \left(
        Mn^4J^{3n}
        +
        M^2nJ^{4n}
        +
        M^2e^{-c_\beta n^2}
    \right).
    \label{eq:sharp-remainder-count-cell}
\end{align}
Moreover, with $\kappa=100$, for all sufficiently small $u$,
\begin{align}
    \nu_z\bigl((D_z^\kappa)^c\bigr)
    =
    \nu_z\left(
        \sum_{x\in B_z}
        \mathbf 1_{\{\phi(x)\le0\}}
        >
        \kappa
    \right)
    &\le
    u^{10}.
    \label{eq:sharp-cell-cutoff}
\end{align}
\end{lemma}

\begin{proof}[Proof of Lemma~\ref{lem:sharp-cluster-counts-cell}]
All constants below are uniform over the choices specified in the statement.

By definition,
\begin{align*}
    I_{1,z}
    &=
    \sum_{x\in C_z^\circ}\mathbf 1_{S_x},
    \qquad
    I_{2,z}
    =
    \sum_{e\in E(C_z^\circ)}\mathbf 1_{P_e}.
\end{align*}
Hence Proposition~\ref{prop:conditioned-cell-peak-estimates} gives
\begin{align*}
    \nu_z(I_{1,z})
    &\ge
    |C_z^\circ|
    \left[
        a_1(\beta)J^{2n}
        -
        C_\beta
        \left(
            J^{3n}
            +
            e^{-c_\beta n^2}
        \right)
    \right],
    \\
    \nu_z(I_{2,z})
    &\le
    |E(C_z^\circ)|
    \left[
        a_2(\beta)J^{3n}
        +
        C_\beta
        \left(
            nJ^{4n}+e^{-c_\beta n^2}
        \right)
    \right].
\end{align*}
Using \eqref{eq:core-cardinality-estimates},
\begin{align*}
    |C_z^\circ|
    &=
    M^2+O(Mn^4),
    \\
    |E(C_z^\circ)|
    &=
    2M^2+O(Mn^4),
\end{align*}
we obtain \eqref{eq:sharp-singleton-count-cell} and \eqref{eq:sharp-pair-count-cell}.
The factor $2$ in the latter comes from the asymptotic nearest-neighbor edge density of a square cell.

We next estimate $R_z$.
Since $E(C_z^\circ)\subset E(B_z)$ by \eqref{eq:core-nearest-neighbor-enlargement},
\begin{align}
    R_z
    &=
    \sum_{e=\{x,y\}\in E(C_z^\circ)}
    \left(
        \mathbf 1_{\{\phi(x)\le0,\phi(y)\le0\}}
        -
        \mathbf 1_{P_e}
    \right)
    \nonumber\\
    &\qquad
    +
    \sum_{\{x,y\}\in E(B_z)\setminus E(C_z^\circ)}
    \mathbf 1_{\{\phi(x)\le0,\phi(y)\le0\}}.
    \label{eq:remainder-edge-decomposition}
\end{align}
For a core edge, Proposition~\ref{prop:conditioned-cell-peak-estimates} gives
\begin{align*}
    \nu_z(R_z)
    &\le
    C_\beta\left(M^2nJ^{4n}+M^2e^{-c_\beta n^2}+Mn^4J^{3n}\right),
\end{align*}
whereas for every edge of $B_z$,
\begin{align*}
    \nu_z(\phi(x)\le0,\phi(y)\le0)
    &\le
    C_\beta J^{3n}.
\end{align*}
Together with \eqref{eq:core-cardinality-estimates} and \eqref{eq:cell-boundary-edge-estimate}, this yields
\begin{align*}
    \nu_z(R_z)
    &\le
    C_\beta
    \left(
        Mn^4J^{3n}
        +
        M^2nJ^{4n}
        +
        M^2e^{-c_\beta n^2}
    \right),
\end{align*}
which proves \eqref{eq:sharp-remainder-count-cell}.

It remains to prove the cutoff estimate.
For every fixed $\kappa\ge1$, the event $(D_z^\kappa)^c$ implies the existence of $\kappa$ distinct sites in $B_z$ at which $\phi\le0$.
Hence, by \eqref{eq:conditioned-cell-multipoint} and a union bound,
\begin{align}
    \nu_z\bigl((D_z^\kappa)^c\bigr)
    &\le
    C_{\beta,\kappa}
    |B_z|^\kappa
    e^{-3\beta n\sqrt{\kappa}}
    \nonumber\\
    &\le
    C_{\beta,\kappa}
    M^{2\kappa}
    J^{\frac32\sqrt{\kappa}\,n}.
    \label{eq:cutoff-general-kappa}
\end{align}
Taking $\kappa=100$ and using $M\le u^{-\chi}$ and $J^n\le C_\beta u$, we obtain
\begin{align*}
    \nu_z\bigl((D_z^{100})^c\bigr)
    &\le
    C_\beta M^{200}J^{15n}
    \le
    C_\beta u^{15-200\chi}
    \le
    u^{10}
\end{align*}
for all sufficiently small $u$, since $\chi<1/50$ implies $15-200\chi>11$.
This proves \eqref{eq:sharp-cell-cutoff}.
\end{proof}

\subsection{Sharp cell contribution}
\label{subsec:5.3}

We now estimate the sharp cell partition functions introduced in Subsection~4.5.
We first extract the leading deterministic contribution from the singleton and two-site components estimated in Lemma~\ref{lem:sharp-cluster-counts-cell}.
We then show that the centered disorder changes the quenched logarithm of the cell partition function only by an error of order $M^2J^{4n}$.

Recall the sharp cell potential $V_z$ from \eqref{eq:sharp-cell-potential}, and define
\begin{align}
    A_z
    &:=
    \nu_z
    \left[
        \mathbf 1_{D_z^\kappa}e^{V_z}
    \right].
    \label{eq:deterministic-cell-normalization}
\end{align}

We first extract the deterministic contribution to the sharp cell partition function.
Using the cluster-count estimates of Lemma~\ref{lem:sharp-cluster-counts-cell}, we identify the leading contributions of exact singleton and two-site components in the deterministic core, while the remaining occupied edges and the contact cutoff contribute only lower-order errors.
This yields the following sharp expansion of the deterministic cell normalization.

\begin{lemma}[Sharp deterministic cell expansion]
\label{lem:sharp-deterministic-cell-expansion}
Assume that $\beta$ is sufficiently large.
Then, for all sufficiently small $u>0$, uniformly over integers $n=n(u)$ satisfying \eqref{eq:admissible-localization-window}, over the cell $z$, over $N$, and over $\Gamma\in\mathcal U_{N,n,M}$,
\begin{align}
    \log A_z
    &\ge
    M^2
    \left[
        a_1(\beta)(e^u-1)J^{2n}
        -
        2a_2(\beta)\theta_2(\beta)J^{3n}
    \right]
    -
    \mathcal E_z^{\mathrm{det}},
    \label{eq:sharp-deterministic-cell-expansion}
\end{align}
where
\begin{align}
    \mathcal E_z^{\mathrm{det}}
    \le
    C_\beta
    \Bigl[
        &Mn^4
        \bigl(
            uJ^{2n}+J^{3n}
        \bigr)
        \nonumber\\
        &+
        M^2
        \bigl(
            uJ^{3n}
            +
            nJ^{4n}
            +
            e^{-c_\beta n^2}
        \bigr)
        +
        u^{10}
    \Bigr].
    \label{eq:deterministic-cell-remainder}
\end{align}
\end{lemma}

\begin{proof}[Proof of Lemma ~\ref{lem:sharp-deterministic-cell-expansion}]
By the definition of $V_z$ and $\theta_2(\beta)=1-e^{-\ell_2(\beta)}$,
\begin{align}
    e^{V_z}
    &=
    e^{uI_{1,z}}
    (1-\theta_2(\beta))^{I_{2,z}}
    e^{-C_\star(\beta)R_z}.
    \label{eq:factorization-deterministic-cell-weight}
\end{align}
Using the Bernoulli inequality and $1-e^{-t}\le t$ for $t\ge0$,
\begin{align}
    1
    -
    (1-\theta_2(\beta))^{I_{2,z}}
    e^{-C_\star(\beta)R_z}
    &\le
    \theta_2(\beta)I_{2,z}
    +
    C_\star(\beta)R_z.
    \label{eq:auxiliary-cell-factor-loss}
\end{align}
Also,
\begin{align*}
    e^{uI_{1,z}}
    &=
    \bigl(1+(e^u-1)\bigr)^{I_{1,z}}
    \ge
    1+(e^u-1)I_{1,z}.
\end{align*}
Combining these estimates gives
\begin{align}
    e^{V_z}
    \ge{}&
    1
    +(e^u-1)I_{1,z}
    -\theta_2(\beta)I_{2,z}
    -C_\star(\beta)R_z
    \nonumber\\
    &-
    (e^u-1)I_{1,z}
    \left(
        \theta_2(\beta)I_{2,z}
        +
        C_\star(\beta)R_z
    \right).
    \label{eq:pointwise-deterministic-cell-expansion}
\end{align}
On $D_z^\kappa$, $I_{1,z}\le\sum_{x\in B_z}\mathbf 1_{\{\phi(x)\le0\}}\le\kappa$.
Since $e^u-1\le Cu$ for $u\in[0,1]$, it follows that
\begin{align}
    e^{V_z}
    \ge{}&
    1
    +(e^u-1)I_{1,z}
    -\theta_2(\beta)I_{2,z}
    -C_\star(\beta)R_z
    \nonumber\\
    &-
    C_{\beta,\kappa}u
    \left(
        I_{2,z}+R_z
    \right)
    \qquad
    \text{on }D_z^\kappa.
    \label{eq:pointwise-deterministic-cell-expansion-cutoff}
\end{align}
This estimate treats configurations containing several singleton or two-site components simultaneously, so no separate multi-cluster event is required.

Multiplying \eqref{eq:pointwise-deterministic-cell-expansion-cutoff} by $\mathbf 1_{D_z^\kappa}$ and taking $\nu_z$-expectations gives
\begin{align}
    A_z
    \ge{}&
    \nu_z(D_z^\kappa)
    +(e^u-1)\nu_z(I_{1,z};D_z^\kappa)
    -\theta_2(\beta)\nu_z(I_{2,z};D_z^\kappa)
    \nonumber\\
    &-
    C_\star(\beta)\nu_z(R_z;D_z^\kappa)
    -
    C_{\beta,\kappa}u\nu_z(I_{2,z}+R_z;D_z^\kappa).
    \label{eq:deterministic-cell-expectation-before-cutoff-removal}
\end{align}
By \eqref{eq:sharp-cell-cutoff}, $\nu_z(D_z^\kappa)\ge1-u^{10}$.
Since $I_{1,z}\le M^2$,
\begin{align*}
    \nu_z(I_{1,z};D_z^\kappa)
    &\ge
    \nu_z(I_{1,z})-M^2u^{10}.
\end{align*}
For the terms with negative coefficients, we use
\begin{align*}
    \nu_z(I_{2,z};D_z^\kappa)
    &\le
    \nu_z(I_{2,z}),
    \qquad
    \nu_z(R_z;D_z^\kappa)
    \le
    \nu_z(R_z).
\end{align*}
Moreover, $(e^u-1)M^2u^{10}\le Cu^{11-2\chi}\le Cu^{10}$ for all sufficiently small $u$.
Therefore,
\begin{align}
    A_z
    \ge{}&
    1
    +(e^u-1)\nu_z(I_{1,z})
    -\theta_2(\beta)\nu_z(I_{2,z})
    -C_\star(\beta)\nu_z(R_z)
    \nonumber\\
    &-
    C_{\beta,\kappa}u
    \left[
        \nu_z(I_{2,z})+\nu_z(R_z)
    \right]
    -C_\beta u^{10}.
    \label{eq:deterministic-cell-expectation-after-cutoff-removal}
\end{align}
We now apply \eqref{eq:sharp-singleton-count-cell}, \eqref{eq:sharp-pair-count-cell}, and \eqref{eq:sharp-remainder-count-cell}.
Since the coefficient of $I_{1,z}$ is positive, we use its lower bound, while for $I_{2,z}$ and $R_z$ we use their upper bounds.
This yields
\begin{align}
    A_z
    &\ge
    1
    +
    M^2
    \left[
        a_1(\beta)(e^u-1)J^{2n}
        -
        2a_2(\beta)\theta_2(\beta)J^{3n}
    \right]
    -
    \widetilde{\mathcal E}_z^{\mathrm{det}},
    \label{eq:deterministic-cell-before-logarithm}
\end{align}
where
\begin{align}
    \widetilde{\mathcal E}_z^{\mathrm{det}}
    \le
    C_\beta
    \Bigl[
        &(e^u-1)
        \left(
            Mn^4J^{2n}
            +M^2J^{3n}
            +M^2e^{-c_\beta n^2}
        \right)
        \nonumber\\
        &+
        Mn^4J^{3n}
        +M^2nJ^{4n}
        +M^2e^{-c_\beta n^2}
        +uM^2J^{3n}
        +u^{10}
    \Bigr].
    \label{eq:preliminary-deterministic-cell-remainder}
\end{align}
Since $e^u-1\le Cu$,
\begin{align}
    \widetilde{\mathcal E}_z^{\mathrm{det}}
    \le
    C_\beta
    \Bigl[
        &Mn^4
        \left(
            uJ^{2n}+J^{3n}
        \right)
        \nonumber\\
        &+
        M^2
        \left(
            uJ^{3n}+nJ^{4n}+e^{-c_\beta n^2}
        \right)
        +u^{10}
    \Bigr].
    \label{eq:simplified-deterministic-cell-remainder}
\end{align}
By \eqref{eq:localization-height-scale} and \eqref{eq:coarse-graining-scale}, $M^2J^{3n}=O(u^{3-2\chi})=o(1)$, and every term on the right-hand side of \eqref{eq:simplified-deterministic-cell-remainder} is also $o(1)$.
Thus the quantity following the leading $1$ on the right-hand side of \eqref{eq:deterministic-cell-before-logarithm} has absolute value at most $CM^2J^{3n}$ for all sufficiently small $u$.
Using $\log(1+t)\ge t-Ct^2$ for $|t|\le1/2$, we obtain
\begin{align*}
    \log A_z
    \ge{}&
    M^2
    \left[
        a_1(\beta)(e^u-1)J^{2n}
        -2a_2(\beta)\theta_2(\beta)J^{3n}
    \right]
    -\widetilde{\mathcal E}_z^{\mathrm{det}}
    -CM^4J^{6n}.
\end{align*}
Finally, $M^4J^{6n}\le CM^2uJ^{3n}$ because $M^2J^{3n}/u=O(u^{2-2\chi})\rightarrow0$.
The logarithmic square error is therefore absorbed by the $M^2uJ^{3n}$ term in \eqref{eq:deterministic-cell-remainder}, which proves \eqref{eq:sharp-deterministic-cell-expansion}.
\end{proof}

We next quantify the effect of the centered disorder on the sharp cell partition function.
To separate the deterministic contribution from the random one, we tilt the conditioned cell law by the deterministic sharp weight obtained above.
Under this tilted measure, the random cell partition function factors into the deterministic normalization $A_z$ and a disorder-dependent factor, whose quenched logarithmic correction will be shown to be of order $M^2J^{4n}$.

Define the deterministic tilted cell measure by
\begin{align}
    \mu_z^{\mathrm{det}}(d\phi)
    &:=
    \frac{
        \mathbf 1_{D_z^\kappa}(\phi)e^{V_z(\phi)}
    }{
        A_z
    }
    \nu_z(d\phi).
    \label{eq:deterministic-sharp-cell-tilt}
\end{align}

\begin{lemma}[Disorder correction in a sharp cell]
\label{lem:disorder-correction-sharp-cell}
Assume that $\beta$ is sufficiently large.
Then, for all sufficiently small $u>0$, uniformly over integers $n=n(u)$ satisfying \eqref{eq:admissible-localization-window}, over $z$, $N$, and $\Gamma\in\mathcal U_{N,n,M}$,
\begin{align}
    \E
    \left[
        \log
        Z_{B_z}^{\omega,\sharp}(\Gamma)
    \right]
    &\ge
    \log A_z
    -
    C_{\alpha,\beta}M^2J^{4n}.
    \label{eq:sharp-cell-disorder-correction}
\end{align}
\end{lemma}

\begin{proof}[Proof of Lemma ~\ref{lem:disorder-correction-sharp-cell}]
On the event $\{\phi(x)>0\text{ for all }x\in B_z\}$, we have $I_{1,z}=I_{2,z}=R_z=0$, and hence $V_z=0$.
Therefore,
\begin{align}
    A_z
    &\ge
    \nu_z\!\left(\phi(x)>0\ \text{for all }x\in B_z\right).
    \label{eq:deterministic-cell-normalization-lower-bound}
\end{align}
By \eqref{eq:conditioned-cell-rough-one-two-point},
\begin{align*}
    \nu_z\!\left(\exists x\in B_z:\phi(x)\le0\right)
    &\le
    \sum_{x\in B_z}
    \nu_z(\phi(x)\le0)
    \le
    C_\beta M^2J^{2n}
    =
    o(1).
\end{align*}
Consequently,
\begin{align}
    A_z
    &\ge
    \frac12
    \label{eq:deterministic-cell-normalization-positive}
\end{align}
for all sufficiently small $u$.

On $D_z^\kappa$, the negative terms in $V_z$ may be discarded, so
\begin{align*}
    V_z
    &\le
    uI_{1,z}
    \le
    u\kappa.
\end{align*}
Together with \eqref{eq:deterministic-cell-normalization-positive} and \eqref{eq:conditioned-cell-rough-one-two-point}, this gives, for every $x\in B_z$,
\begin{align}
    \mu_z^{\mathrm{det}}(\phi(x)\le0)
    &\le
    2e^{u\kappa}\nu_z(\phi(x)\le0)
    \le
    C_\beta J^{2n}.
    \label{eq:one-point-bound-deterministic-sharp-cell-tilt}
\end{align}

Set
\begin{align}
    X_z^\omega
    &:=
    \mu_z^{\mathrm{det}}
    \left[
        \exp
        \left\{
            \sum_{x\in B_z}
            \eta_x^{(\alpha)}
            \mathbf 1_{\{\phi(x)\le0\}}
        \right\}
    \right].
    \label{eq:sharp-cell-disorder-factor}
\end{align}
Then
\begin{align}
    Z_{B_z}^{\omega,\sharp}(\Gamma)
    &=
    A_zX_z^\omega.
    \label{eq:sharp-cell-partition-factorization}
\end{align}
Since $\E[e^{\eta_x^{(\alpha)}}]=1$ and the disorder variables are independent,
\begin{align*}
    \E[X_z^\omega]
    &=
    1.
\end{align*}
Furthermore,
\begin{align*}
    X_z^\omega
    &\ge
    \mu_z^{\mathrm{det}}\!\left(\phi(x)>0\ \text{for all }x\in B_z\right)
    \\
    &\ge
    1-
    \sum_{x\in B_z}
    \mu_z^{\mathrm{det}}(\phi(x)\le0)
    \\
    &\ge
    1-C_\beta M^2J^{2n}
    \ge
    \frac12
\end{align*}
for all sufficiently small $u$.
Hence, using
\begin{align*}
    \log t
    &\ge
    (t-1)-(t-1)^2,
    \qquad
    t\ge\frac12,
\end{align*}
we obtain
\begin{align}
    \E[\log X_z^\omega]
    &\ge
    -
    \E
    \left[
        (X_z^\omega-1)^2
    \right].
    \label{eq:log-disorder-factor-second-moment}
\end{align}

Let $\phi^{(1)}$ and $\phi^{(2)}$ be two independent samples from $\mu_z^{\mathrm{det}}$, and define their contact overlap by
\begin{align}
    \mathcal R_{12}
    &:=
    \sum_{x\in B_z}
    \mathbf 1_{\{\phi^{(1)}(x)\le0\}}
    \mathbf 1_{\{\phi^{(2)}(x)\le0\}}.
    \label{eq:sharp-cell-contact-overlap}
\end{align}
A direct computation using the independence of the disorder gives
\begin{align}
    \E
    \left[
        (X_z^\omega-1)^2
    \right]
    &=
    (\mu_z^{\mathrm{det}})^{\otimes2}
    \left[
        e^{\Delta_\alpha\mathcal R_{12}}
    \right]
    -
    1,
    \qquad
    \Delta_\alpha
    :=
    \lambda(2\alpha)-2\lambda(\alpha).
    \label{eq:sharp-cell-disorder-second-moment-identity}
\end{align}
Since $\mu_z^{\mathrm{det}}$ is supported on $D_z^\kappa$, we have $\mathcal R_{12}\le\kappa$.
Hence
\begin{align*}
    e^{\Delta_\alpha\mathcal R_{12}}-1
    &\le
    C_{\alpha,\kappa}\mathcal R_{12}.
\end{align*}
Using \eqref{eq:one-point-bound-deterministic-sharp-cell-tilt}, we obtain
\begin{align}
    \E
    \left[
        (X_z^\omega-1)^2
    \right]
    &\le
    C_{\alpha,\kappa}
    \sum_{x\in B_z}
    \mu_z^{\mathrm{det}}(\phi(x)\le0)^2
    \nonumber\\
    &\le
    C_{\alpha,\beta}M^2J^{4n}.
    \label{eq:sharp-cell-disorder-second-moment-bound}
\end{align}
Combining \eqref{eq:sharp-cell-partition-factorization}, \eqref{eq:log-disorder-factor-second-moment}, and \eqref{eq:sharp-cell-disorder-second-moment-bound} proves \eqref{eq:sharp-cell-disorder-correction}.
\end{proof}

Combining the sharp deterministic expansion of Lemma~\ref{lem:sharp-deterministic-cell-expansion} with the disorder correction of Lemma~\ref{lem:disorder-correction-sharp-cell} gives the sharp cell contribution in the form required by the coarse-grained lower-bound decomposition.
Summing this estimate over all coarse-grained cells yields the following bulk-cell lower bound.

\begin{corollary}[Sharp bulk-cell lower bound]
\label{cor:sharp-bulk-cell-lower-bound}
Assume that $\beta$ is sufficiently large.
Then, for all sufficiently small $u>0$, uniformly over integers $n=n(u)$ satisfying \eqref{eq:admissible-localization-window} and over $\Gamma\in\mathcal U_{N,n,M}$,
\begin{align}
    \sum_z
    \E
    \left[
        \log
        Z_{B_z}^{\omega,\sharp}(\Gamma)
    \right]
    \ge{}&
    N^2
    \left[
        a_1(\beta)uJ^{2n}
        -
        2a_2(\beta)\theta_2(\beta)J^{3n}
    \right]
    \nonumber\\
    &-
    N^2
    r_{\alpha,\beta}^{\mathrm{cell}}(u,n,M),
    \label{eq:sharp-bulk-cell-lower-bound}
\end{align}
where
\begin{align}
    r_{\alpha,\beta}^{\mathrm{cell}}(u,n,M)
    \le
    C_{\alpha,\beta}
    \Biggl[
        &
        \frac{n^4}{M}
        \left(
            uJ^{2n}+J^{3n}
        \right)
        +
        uJ^{3n}
        +
        nJ^{4n}
        +
        J^{4n}
        \nonumber\\
        &+
        u^2J^{2n}
        +
        e^{-c_\beta n^2}
        +
        \frac{u^{10}}{M^2}
    \Biggr].
    \label{eq:sharp-bulk-cell-remainder}
\end{align}
In particular,
\begin{align}
    r_{\alpha,\beta}^{\mathrm{cell}}(u,n,M)
    &=
    o(u^3)
    \qquad
    \text{as }u\rightarrow 0+,
    \label{eq:sharp-bulk-cell-remainder-small}
\end{align}
uniformly over the localization window \eqref{eq:admissible-localization-window}.
\end{corollary}

\begin{proof}[Proof of Corollary ~\ref{cor:sharp-bulk-cell-lower-bound}]
Combining Lemma~\ref{lem:sharp-deterministic-cell-expansion} and Lemma~\ref{lem:disorder-correction-sharp-cell}, and summing over the $N^2/M^2$ cells, gives
\begin{align*}
    \sum_z
    \E
    \left[
        \log Z_{B_z}^{\omega,\sharp}(\Gamma)
    \right]
    \ge{}&
    N^2
    \left[
        a_1(\beta)(e^u-1)J^{2n}
        -
        2a_2(\beta)\theta_2(\beta)J^{3n}
    \right]
    \\
    &-
    C_{\alpha,\beta}N^2
    \Biggl[
        \frac{n^4}{M}
        \left(
            uJ^{2n}+J^{3n}
        \right)
        +uJ^{3n}
        +nJ^{4n}
        +J^{4n}
        +e^{-c_\beta n^2}
        +\frac{u^{10}}{M^2}
    \Biggr].
\end{align*}
Since $e^u-1=u+O(u^2)$, replacing $e^u-1$ by $u$ produces an additional error bounded by $Cu^2J^{2n}$.
This proves \eqref{eq:sharp-bulk-cell-lower-bound} and \eqref{eq:sharp-bulk-cell-remainder}.

Finally, by \eqref{eq:localization-height-scale} and \eqref{eq:coarse-graining-scale},
\begin{align*}
    \frac{n^4}{M}
    \left(
        uJ^{2n}+J^{3n}
    \right)
    &=
    O
    \left(
        |\log u|^4u^{3+\chi}
    \right)
    =
    o(u^3).
\end{align*}
Moreover,
\begin{align*}
    uJ^{3n},
    \quad
    J^{4n},
    \quad
    u^2J^{2n}
    &=
    O(u^4),
    \\
    nJ^{4n}
    &=
    O(|\log u|u^4),
\end{align*}
while $e^{-c_\beta n^2}$ is smaller than any power of $u$ and
\begin{align*}
    \frac{u^{10}}{M^2}
    &=
    O(u^{10+2\chi}).
\end{align*}
This proves \eqref{eq:sharp-bulk-cell-remainder-small}.
\end{proof}

\subsection{Contour-region and cross-edge errors}
\label{subsec:5.4}

It remains to estimate the two terms in Proposition~\ref{prop:coarse-grained-lower-bound-decomposition} which were not included in the sharp bulk-cell contribution: the rough partition functions of the contour regions and the cross-edge error.
Since neither term contributes to the leading-order coefficient, the uniform contour-restricted peak estimates are sufficient.

\begin{lemma}[Contour-region and cross-edge errors]
\label{lem:contour-region-cross-edge-errors}
Assume that $\beta$ is sufficiently large.
Then, for all sufficiently small $u>0$, uniformly over integers $n=n(u)$ satisfying \eqref{eq:admissible-localization-window}, over $N$, and over $\Gamma\in\mathcal U_{N,n,M}$,
\begin{align}
    \sum_{\gamma\in\Gamma}
    \E
    \left[
        \log
        Z_{\overline\gamma}^{\omega,\mathrm{rough}}
        (\Gamma)
    \right]
    &\ge
    -C_{\alpha,\beta}
    \frac{N^2n^4}{M}J^{3n},
    \label{eq:total-contour-region-lower-bound}
\end{align}
and
\begin{align}
    C_\star(\beta) \E \left[ \mu_{\Gamma}^{\omega,\mathrm{mix}} \left[ X_{\mathrm{cr}} \right] \right]
    &\le
    C_{\alpha,\beta}
    \frac{N^2n^4}{M}J^{4n}.
    \label{eq:sharp-cross-edge-error-bound}
\end{align}
Consequently,
\begin{align}
    \sum_{\gamma\in\Gamma}
    \E
    \left[
        \log
        Z_{\overline\gamma}^{\omega,\mathrm{rough}}
        (\Gamma)
    \right]
    -
    C_\star(\beta) \E \left[ \mu_{\Gamma}^{\omega,\mathrm{mix}} \left[ X_{\mathrm{cr}} \right] \right]
    &\ge
    -o(N^2u^3),
    \label{eq:contour-cross-edge-negligible}
\end{align}
uniformly over the localization window and over $\Gamma\in\mathcal U_{N,n,M}$.
\end{lemma}

\begin{proof}[Proof of Lemma ~\ref{lem:contour-region-cross-edge-errors}]
Fix an admissible realization $\Gamma$.
We first estimate a single contour region $R=\overline\gamma$, $\gamma\in\Gamma$.
Put
\begin{align*}
b_\gamma:=n+\varepsilon(\gamma)\in\{n-1,n+1\}.
\end{align*}
Here $\Gamma$ fixes signed external contours, not their cylinder intensities.
Removing one compulsory signed unit of the external cylinder gives the conditional-law identity of \cite[(4.17)]{lacoin2021solid}.
In the present notation,
\begin{align}
\nu_R
&=
\P_{R,\beta}^{b_\gamma}
\left(
\,\cdot\,
\middle|
\mathcal A^{\mathrm{short},b_\gamma}_{R,n},
\ \varepsilon(\gamma)(\phi-b_\gamma)\ge0
\text{ on }\Delta_\gamma^-
\right).
\label{eq:external-contour-region-law}
\end{align}
Here $\mathcal A^{\mathrm{short},b}_{R,n}$ requires every contour relative to boundary height $b$ to have geometric length at most $n^4$.
After setting $\psi=b_\gamma-\phi$, the law is a zero-boundary contour-restricted SOS measure on $R$.
For all sufficiently small $u$, we have $n\ge2$, and hence $b_\gamma\ge1$.
Applying Proposition~\ref{prop:restricted-peak-estimates} at level $b_\gamma$, and using
\begin{align*}
J^{2b_\gamma}
&\le J^{-2}J^{2n},
\\
J^{3b_\gamma}
&\le J^{-3}J^{3n},
\\
e^{-3\beta b_\gamma\sqrt r}
&\le e^{3\beta\sqrt r}e^{-3\beta n\sqrt r},
\end{align*}
we obtain the following estimates, with the extra factors absorbed into $C_\beta$ and $C_{\beta,r}$.

Therefore Proposition~\ref{prop:restricted-peak-estimates}, applied at level $n-1$, gives
\begin{align}
    \nu_R(\phi(x)\le0)
    &\le
    C_\beta J^{2(n-1)}
    \le
    C_\beta J^{2n},
    \qquad x\in R,
    \nonumber\\
    \nu_R(\phi(x)\le0,\phi(y)\le0)
    &\le
    C_\beta J^{3(n-1)}
    \le
    C_\beta J^{3n},
    \qquad \{x,y\}\in E(R),
    \nonumber\\
    \nu_R
    \left(
        \phi(x_i)\le0
        \text{ for all }1\le i\le r
    \right)
    &\le
    C_{\beta,r}
    e^{-3\beta(n-1)\sqrt r}
    \le
    C_{\beta,r}
    e^{-3\beta n\sqrt r}.
    \label{eq:contour-region-rough-peak-estimates}
\end{align}
for every fixed $r\ge1$.

Let
\begin{align*}
    D_R^\kappa
    &:=
    \left\{
        \sum_{x\in R}
        \mathbf 1_{\{\phi(x)\le0\}}
        \le\kappa
    \right\},
    \qquad
    \kappa=100.
\end{align*}
Since $|R|\le Cn^8$, the multipoint estimate and a union bound yield
\begin{align}
    \nu_R\bigl((D_R^\kappa)^c\bigr)
    &\le
    C_\beta |R|J^{3n}.
    \label{eq:contour-region-cutoff}
\end{align}

Define
\begin{align}
    A_R^{\mathrm{det}}
    &:=
    \nu_R
    \left[
        \mathbf 1_{D_R^\kappa}
        \exp
        \left\{
            u\sum_{x\in R}\mathbf 1_{\{\phi(x)\le0\}}
            -
            C_\star(\beta)
            \sum_{\{x,y\}\in E(R)}
            \mathbf 1_{\{\phi(x)\le0,\phi(y)\le0\}}
        \right\}
    \right].
    \label{eq:contour-region-deterministic-normalization}
\end{align}
Using \eqref{eq:contour-region-cutoff}, Jensen's inequality, and the two-point estimate above, we obtain
\begin{align}
    \log A_R^{\mathrm{det}}
    &\ge
    -C_\beta |R|J^{3n}.
    \label{eq:contour-region-deterministic-lower-bound}
\end{align}

Let $\mu_R^{\mathrm{det}}$ be the corresponding normalized tilted measure.
Since the tilt is uniformly bounded on $D_R^\kappa$ and $A_R^{\mathrm{det}}$ is bounded away from zero,
\begin{align}
    \mu_R^{\mathrm{det}}(\phi(x)\le0)
    &\le
    C_\beta J^{2n},
    \qquad x\in R.
    \label{eq:contour-region-deterministic-tilt-one-point}
\end{align}
The disorder can now be treated exactly as in Lemma~\ref{lem:disorder-correction-sharp-cell}.
Namely, writing
\begin{align*}
    X_R^\omega
    &:=
    \mu_R^{\mathrm{det}}
    \left[
        \exp
        \left\{
            \sum_{x\in R}
            \eta_x^{(\alpha)}
            \mathbf 1_{\{\phi(x)\le0\}}
        \right\}
    \right],
\end{align*}
we have $\E[X_R^\omega]=1$, while the replica overlap estimate gives
\begin{align*}
    \E\left[(X_R^\omega-1)^2\right]
    &\le
    C_{\alpha,\beta}|R|J^{4n}.
\end{align*}
Since $X_R^\omega\ge1/2$ for all sufficiently small $u$,
\begin{align*}
    \E[\log X_R^\omega]
    &\ge
    -C_{\alpha,\beta}|R|J^{4n}.
\end{align*}
Therefore
\begin{align}
    \E
    \left[
        \log Z_R^{\omega,\mathrm{rough}}
    \right]
    &\ge
    -C_{\alpha,\beta}|R|J^{3n}.
    \label{eq:single-contour-region-lower-bound}
\end{align}
Summing over $R=\overline\gamma$ and using \eqref{eq:total-contour-volume} proves \eqref{eq:total-contour-region-lower-bound}.

We next estimate the cross-edge term.
We claim that, for either type of local tilted measure,
\begin{align}
    \E
    \left[
        \mu_z^{\omega,\sharp}(\phi(x)\le0)
    \right]
    &\le
    C_{\alpha,\beta}J^{2n},
    \qquad x\in B_z,
    \nonumber\\
    \E
    \left[
        \mu_{\overline\gamma}^{\omega,\mathrm{rough}}
        (\phi(x)\le0)
    \right]
    &\le
    C_{\alpha,\beta}J^{2n},
    \qquad x\in\overline\gamma.
    \label{eq:mixed-local-quenched-one-point}
\end{align}
Indeed, in both cases the no-contact event has weight one and has probability $1-o(1)$ under the underlying SOS measure.
Hence the corresponding local partition function is bounded below by $1/2$.
Averaging over the centered disorder removes the random exponential, while the deterministic tilt is bounded above by $e^{u\kappa}$ in the sharp cells and by $e^{u|R|}=1+o(1)$ in the contour regions.
The one-point estimates above and \eqref{eq:conditioned-cell-rough-one-two-point} therefore imply \eqref{eq:mixed-local-quenched-one-point}.
For every cross-edge $\{x,y\}\in E_{\mathrm{cr}}(\Gamma)$, the endpoints lie in distinct regions, say $R_x$ and $R_y$.
Conditionally on the disorder, the product structure of $\mu_\Gamma^{\omega,\mathrm{mix}}$ gives
\[
\mu_\Gamma^{\omega,\mathrm{mix}}
\bigl(
\phi(x)\le0,\phi(y)\le0
\bigr)
=
\mu_{R_x}^{\omega}(\phi(x)\le0)
\mu_{R_y}^{\omega}(\phi(y)\le0).
\]
Moreover, the two factors depend on disjoint families of disorder variables, since $R_x\cap R_y=\varnothing$.
Hence they are independent under $\mathbb P$, and by \eqref{eq:mixed-local-quenched-one-point},
\begin{align*}
\mathbb E\left[
\mu_\Gamma^{\omega,\mathrm{mix}}
\bigl(
\phi(x)\le0,\phi(y)\le0
\bigr)
\right]
\le C_{\alpha,\beta}J^{4n}.
\end{align*}

Finally,
\begin{align*}
    \frac{n^4}{M}J^{3n}
    &=
    O\left(|\log u|^4u^{3+\chi}\right)
    =
    o(u^3),
    \\
    \frac{n^4}{M}J^{4n}
    &=
    O\left(|\log u|^4u^{4+\chi}\right)
    =
    o(u^3),
\end{align*}
by \eqref{eq:localization-height-scale} and \eqref{eq:coarse-graining-scale}.
This proves \eqref{eq:contour-cross-edge-negligible}.
\end{proof}

\subsection{Proof of the sharp lower bound}
\label{subsec:5.5}

We now combine the sharp bulk-cell estimate with the bounds on the contour-region and cross-edge terms.
No further local estimates are needed.

\begin{proof}[Proof of Proposition~\ref{prop:sharp-lower-bound}]
Let $u>0$ be sufficiently small and fix an integer $n=n(u)$ satisfying \eqref{eq:admissible-localization-window}.
Recall that $M$ is given by \eqref{eq:coarse-graining-scale}.
For fixed $u$, both $n$ and $M$ are fixed, and it is enough to let $N\rightarrow\infty$ along multiples of $M$.

By Proposition~\ref{prop:coarse-grained-lower-bound-decomposition}, Corollary~\ref{cor:sharp-bulk-cell-lower-bound}, and Lemma~\ref{lem:contour-region-cross-edge-errors}, uniformly over $\Gamma\in\mathcal U_{N,n,M}$,
\begin{align}
    \frac{1}{N^2}
    \E
    \left[
        \log Z_{N,n}^{\omega}(u)
    \right]
    &\ge
    a_1(\beta)uJ^{2n}
    -
    2a_2(\beta)\theta_2(\beta)J^{3n}
    \nonumber\\
    &-
    r_{\alpha,\beta}^{\mathrm{tot}}(u,n,M),
    \label{eq:finite-volume-sharp-lower-bound-final}
\end{align}
where
\begin{align}
    r_{\alpha,\beta}^{\mathrm{tot}}(u,n,M)
    &:=
    r_{\alpha,\beta}^{\mathrm{cell}}(u,n,M)
    +
    C_\beta e^{-c_\beta n^4}
    +
    C_{\alpha,\beta}
    \frac{n^4}{M}
    \left(
        J^{3n}+J^{4n}
    \right).
    \label{eq:total-error-density-sharp-lower-bound}
\end{align}
Here we used \eqref{eq:restriction-to-good-event} for the contribution of the good-contour restriction.

We next estimate the total error.
By \eqref{eq:sharp-bulk-cell-remainder}, together with \eqref{eq:localization-height-scale} and \eqref{eq:coarse-graining-scale},
\begin{align*}
    \frac{n^4}{M}
    \left(
        uJ^{2n}+J^{3n}
    \right)
    &=
    O
    \left(
        |\log u|^4u^{3+\chi}
    \right).
\end{align*}
All the remaining polynomial terms in \eqref{eq:sharp-bulk-cell-remainder} are of order $O(|\log u|u^4)$ or smaller.
Moreover, $e^{-c_\beta n^2}$ and $e^{-c_\beta n^4}$ are smaller than every power of $u$, while
\begin{align*}
    \frac{n^4}{M}J^{4n}
    &=
    O
    \left(
        |\log u|^4u^{4+\chi}
    \right).
\end{align*}
Consequently,
\begin{align}
    r_{\alpha,\beta}^{\mathrm{tot}}(u,n,M)
    &\le
    C_{\alpha,\beta}|\log u|^4 u^{3+\chi}
    =
    o(u^3),
    \label{eq:quantitative-total-error-sharp-lower-bound}
\end{align}
uniformly over all $n$ satisfying \eqref{eq:admissible-localization-window}.

Letting $N\rightarrow\infty$ along multiples of $M$ in \eqref{eq:finite-volume-sharp-lower-bound-final} and using \eqref{eq:free-energy-finite-volume-representation}, we obtain
\begin{align}
    \overline{F}_\beta(\alpha,u)
    &\ge
    a_1(\beta)uJ^{2n}
    -
    2a_2(\beta)\theta_2(\beta)J^{3n}
    -
    C_{\alpha,\beta}
    |\log u|^4u^{3+\chi},
    \label{eq:sharp-lower-bound-uniform-window}
\end{align}
uniformly over the localization window.

Finally, let $n_u$ be a maximizer in \eqref{eq:homogeneous-variational-theta2}.
By Lemma~\ref{lem:localization-window}, $n_u$ satisfies \eqref{eq:admissible-localization-window}.
Taking $n=n_u$ in \eqref{eq:sharp-lower-bound-uniform-window} therefore gives
\begin{align}
    \overline{F}_\beta(\alpha,u)
    &\ge
    \overline{F}_{\mathrm{homo}}(\beta,u)
    -
    C_{\alpha,\beta}
    |\log u|^4u^{3+\chi}
    \\
    &=
    \overline{F}_{\mathrm{homo}}(\beta,u)
    -
    o(u^3).
    \label{eq:conclusion-sharp-lower-bound}
\end{align}
This proves Proposition~\ref{prop:sharp-lower-bound}.
In particular, the quantitative lower bound holds with exponent $\chi$ and logarithmic power $4$.
\end{proof}

\appendix

\section{Geometry of the coarse-grained regions}
\label{app:coarse-grained-geometry}

This appendix contains the geometric estimates used in the coarse-grained reduction of Section~\ref{chap:4}.
The main point is that every grid-crossing contour is short and intersects the coarse-graining grid, so the regions affected by such contours are confined to a thin neighborhood of the grid.
This yields the volume and edge bounds needed to separate the deterministic cell cores from the lower-order exceptional regions.

\begin{proof}[Proof of Lemma~\ref{lem:coarse-grained-geometry}]
Every contour in an admissible realization $\Gamma$ has length at most $n^4$ and intersects the grid $H_M$.
Hence its contour and interior are contained in an $O(n^4)$-neighborhood of $H_M$.
On the other hand, $C_z^\circ$ is at distance at least $4n^4$ from $\partial C_z$, which is contained in the grid.
Therefore every site at distance at most one from $C_z^\circ$ is separated from $H_M$ by at least $4n^4-1$, and no contour in $\Gamma$ can enclose such a site.
This proves \eqref{eq:core-nearest-neighbor-enlargement}, and in particular $C_z^\circ\subset B_z$.

The set $C_z\setminus C_z^\circ$ is a boundary strip of width $4n^4$ and contains at most $CMn^4$ sites and edges.
Consequently,
\begin{align*}
|C_z^\circ|
&=
M^2+O(Mn^4),
\\
|E(C_z^\circ)|
&=
2M^2+O(Mn^4),
\end{align*}
which proves \eqref{eq:core-cardinality-estimates}.
Since $C_z^\circ\subset B_z\subset C_z$, every edge in $E(B_z)\setminus E(C_z^\circ)$ lies in the same boundary strip.
Hence
\begin{align*}
\left|
E(B_z)\setminus E(C_z^\circ)
\right|
&\le
CMn^4,
\end{align*}
proving \eqref{eq:cell-boundary-edge-estimate}.

For a contour $\gamma\in\Gamma$, the planar isoperimetric inequality gives
\begin{align*}
|\overline\gamma|
&\le
C|\widetilde\gamma|^2
\le
Cn^8,
\end{align*}
which proves \eqref{eq:single-contour-region-size}.

Every contour in $\Gamma$, together with its interior, is contained in the $n^4$-neighborhood of $H_M$.
The total number of lattice sites in this neighborhood is at most
\begin{align*}
C\frac{N^2n^4}{M}.
\end{align*}
Since the interiors of the external contours are pairwise disjoint, we obtain
\begin{align*}
\sum_{\gamma\in\Gamma}
|\overline\gamma|
&\le
C\frac{N^2n^4}{M},
\end{align*}
which proves \eqref{eq:total-contour-volume}.

Finally, every cross-edge either joins neighboring cell regions across the coarse-graining grid or belongs to the edge boundary of one of the contour interiors.
For a finite set $A\subset\mathbb Z^2$, let
\[
\partial_E A
:=
\bigl\{\{x,y\}:x\in A,\ y\notin A,\ x\sim y\bigr\}.
\]
Since every vertex has degree four,
\[
|\partial_E A|\le 4|A|.
\]
Consequently, using \eqref{eq:total-contour-volume},
\begin{align}
|E_{\mathrm{cr}}(\Gamma)|
&\le C\frac{N^2}{M}
   +\sum_{\gamma\in\Gamma}|\partial_E\overline\gamma|
   \notag\\
&\le C\frac{N^2}{M}
   +4\sum_{\gamma\in\Gamma}|\overline\gamma|
   \notag\\
&\le C\frac{N^2n^4}{M}.
\end{align}
This proves \eqref{eq:cross-edge-count} and completes the proof.
\end{proof}

\section{Local comparison for conditioned cell measures}
\label{app:local-comparison}

This appendix provides the local comparison used in Subsection~\ref{subsec:5.1}.
The coarse-graining procedure conditions on the external contours intersecting the grid and therefore modifies the SOS law inside each cell.
We show that, sufficiently deep inside the deterministic core, this conditioning changes the probabilities of fixed-radius local events only by an exponentially small amount.

The proof proceeds by identifying the reflected conditional law as a contour-restricted zero-boundary SOS measure.
We then remove the short-contour restriction and the boundary compatibility constraints.
This allows the sharp infinite-volume peak and cluster estimates to be transferred to the conditioned cell law.

Recall that $\nu_z$ denotes the conditional SOS law in the cell region $B_z$, and let $\overline\nu_z$ denote the law under $\nu_z$ of the reflected field
\begin{align*}
    \psi
    &:=
    n-\phi.
\end{align*}

\begin{lemma}[Local comparison in a coarse-grained cell]
\label{lem:local-comparison-conditioned-cell}
Assume that $\beta$ is sufficiently large.
For every fixed $R<\infty$, there exist constants $c=c(\beta,R)>0$ and $C=C(\beta,R)<\infty$ such that the following holds for all sufficiently small $u>0$.

Let $n=n(u)$ satisfy \eqref{eq:admissible-localization-window}, let $\Gamma\in\mathcal U_{N,n,M}$, and let $x\in C_z^\circ$.
Suppose that $A_x$ is supported in $B(x,R)$ and can be written as
\begin{align*}
    A_x
    &=
    A_x^+\cap A_x^-,
\end{align*}
where $A_x^+$ is increasing and $A_x^-$ is decreasing.
Either $A_x^+$ or $A_x^-$ may be the whole configuration space, so that this includes purely increasing and purely decreasing events.
The event $A_x$ may depend on $n$.

If $A_0$ denotes the translate of $A_x$ to a neighborhood of the origin, then
\begin{align}
    \left|
        \overline\nu_z(A_x)
        -
        \P_\beta(A_0)
    \right|
    &\le
    Ce^{-cn^2}.
    \label{eq:local-comparison-cell}
\end{align}
The estimate is uniform in $z$, $N$, and $\Gamma\in\mathcal U_{N,n,M}$.
\end{lemma}

\begin{proof}[Proof of Lemma ~\ref{lem:local-comparison-conditioned-cell}]
Set
\begin{align*}
K
&:=
\operatorname{supp}(A_x)
\subset B(x,R).
\end{align*}
If $K=\varnothing$, the estimate is immediate.
We therefore assume that $K\neq\varnothing$.
We divide the proof into three steps.

\medskip
\noindent
\emph{Step 1: Identification of the reflected conditional law and its allowed contours.}
We use the cylinder representation introduced in Subsection~\ref{subsec:3.2}.
We introduce the intermediate cell interior
\begin{align}
C_z^{\mathrm{in}}
:=
\left\{
x\in C_z:
\operatorname{dist}(x,\partial C_z)\ge n^4
\right\}.
\label{eq:intermediate-cell-interior}
\end{align}
In particular,
\[
C_z^\circ\subset C_z^{\mathrm{in}},
\]
and the distance between $C_z^\circ$ and the boundary of $C_z^{\mathrm{in}}$ is of order $3n^4$.
Thus a cylinder $\widehat\gamma$ records a signed contour $\gamma$, with geometric contour $\widetilde\gamma$, sign $\varepsilon(\gamma)\in\{\pm1\}$, and intensity $k(\widehat\gamma)\in\mathbb N$.
Let $\widehat\Upsilon_n(\phi)$ denote the full cylinder collection of $\phi$ with boundary height $n$, and define
\begin{align}
\widehat\Upsilon_z^{\mathrm{out}}(\phi)
&:=
\left\{
\widehat\gamma\in\widehat\Upsilon_n(\phi):
\widetilde\gamma\not\subset C_z^{\mathrm{in}}
\right\}.
\label{eq:outer-cylinder-data}
\end{align}
Every contour intersecting the coarse-graining grid belongs to $\widehat\Upsilon_z^{\mathrm{out}}$.
Conversely, a contour contained in $C_z^{\mathrm{in}}$ cannot intersect the grid.
Hence $U$ is measurable with respect to $\widehat\Upsilon_z^{\mathrm{out}}$.

Condition further on an admissible realization
\begin{align*}
\widehat\Upsilon_z^{\mathrm{out}}
&=
\widehat G
\end{align*}
in addition to $A_{N,n}\cap\{U=\Gamma\}$.
Only cylinders in $\widehat G$ whose interiors intersect $C_z^{\mathrm{in}}$ can affect the deletion inside $C_z^{\mathrm{in}}$.
Fix such a cylinder $\widehat\gamma\in\widehat G$.
Since $\widetilde\gamma\not\subset C_z^{\mathrm{in}}$, either $\widetilde\gamma$ intersects $\partial C_z^{\mathrm{in}}$ or $\overline\gamma$ contains $C_z^{\mathrm{in}}$.
The latter is impossible for all sufficiently small $u$: indeed, $M\ge16n^4$, whereas a contour enclosing $C_z^{\mathrm{in}}$ must have length of order $M$, contradicting $|\widetilde\gamma|\le n^4$ on $A_{N,n}$.

Hence every cylinder in $\widehat G$ whose interior meets $C_z^{\mathrm{in}}$ has geometric contour intersecting $\partial C_z^{\mathrm{in}}$.
Since $|\widetilde\gamma|\le n^4$, its geometric contour, its interior inside $C_z^{\mathrm{in}}$, and the adjacent sites are contained in a $(2n^4+O(1))$-neighborhood of $\partial C_z$.

On the other hand, since $x\in C_z^\circ$ and $K\subset B(x,R)$,
\begin{align*}
\operatorname{dist}(K,\partial C_z)
&\ge 4n^4-R,\\
\operatorname{dist}(K,\partial C_z^{\mathrm{in}})
&\ge 3n^4-R.
\end{align*}
Therefore, for every cylinder in $\widehat G$ whose interior meets $C_z^{\mathrm{in}}$,
\begin{align*}
\operatorname{dist}(K,\overline\gamma)
&\ge
2n^4-R-O(1)
\ge
n^4
\end{align*}
for all sufficiently large $n$.
Cylinders whose interiors do not meet $C_z^{\mathrm{in}}$ clearly do not affect the deletion.

It follows that $K$ is not removed.
We may therefore define
\begin{align}
\Lambda
&:=
\text{the connected component containing $K$ of }
C_z^{\mathrm{in}}
\setminus
\bigcup_{\widehat\gamma\in\widehat G}
\overline\gamma.
\label{eq:local-comparison-domain}
\end{align}
Moreover,
\begin{align}
K
&\subset\Lambda,
\qquad
\operatorname{dist}(K,\partial\Lambda)\ge n^4.
\label{eq:local-domain-buffer}
\end{align}

The background height induced on $\Lambda$ by the fixed cylinders is exactly $n$.
Indeed, the cylinder reconstruction formula gives
\begin{align}
\phi(v)
&=
n
+
\sum_{\widehat\gamma\in\widehat\Upsilon_n(\phi)}
\varepsilon(\gamma)k(\widehat\gamma)
\mathbf 1_{\overline\gamma}(v).
\label{eq:cylinder-reconstruction-local}
\end{align}
By the construction of $\Lambda$ and the preceding separation argument, no cylinder in $\widehat G$ encloses a point of $\Lambda$.
Hence the contribution of all fixed cylinders vanishes identically on $\Lambda$.

The compatibility conditions with $\widehat G$ determine two, not necessarily disjoint, subsets $U^+,U^-$ of the internal boundary
\begin{align*}
\partial^-\Lambda
&:=
\left\{
v\in\Lambda:
v\sim w
\text{ for some }w\notin\Lambda
\right\}.
\end{align*}
With the labels chosen after reflection, the remaining field satisfies
\begin{align}
\psi|_{U^+}
&\ge0,
\qquad
\psi|_{U^-}
\le0.
\label{eq:reflected-boundary-sign-constraints}
\end{align}
Let
\begin{align}
\mathcal A_{\Lambda,n}^{\mathrm{short}}
&:=
\left\{
|\widetilde\gamma|\le n^4
\text{ for every }
\gamma\in\Upsilon(\psi)
\text{ contained in }\Lambda
\right\}.
\label{eq:local-short-contour-event}
\end{align}
The contour reconstruction formula and the additivity of the contour energy imply the exact conditional-law identity
\begin{align}
&
\operatorname{Law}_{\P_{N,\beta}^n}
\left(
\psi|_\Lambda
\,\middle|\,
U=\Gamma,
A_{N,n},
\widehat\Upsilon_z^{\mathrm{out}}=\widehat G
\right)
\nonumber\\
&\qquad=
\P_{\Lambda,\beta}
\left(
\,\cdot\,
\middle|\,
\psi|_{U^+}\ge0,
\psi|_{U^-}\le0,
\mathcal A_{\Lambda,n}^{\mathrm{short}}
\right).
\label{eq:exact-reflected-conditional-law}
\end{align}
There is no additional restriction coming from $\{U=\Gamma\}$: every contour contained in $\Lambda$ lies in $C_z^{\mathrm{in}}$ and therefore cannot intersect the coarse-graining grid.

Equivalently, the law in \eqref{eq:exact-reflected-conditional-law} is a zero-boundary contour-restricted SOS measure of the type considered in Subsection~\ref{subsec:3.3}.
More precisely, let $\mathcal L_{z,\widehat G}$ be the family of signed contours in $\Lambda$ which have geometric length at most $n^4$ and are compatible with the fixed outer cylinder collection $\widehat G$.
Then
\begin{align*}
&
\operatorname{Law}_{\P_{N,\beta}^n}
\left(
\psi|_\Lambda
\,\middle|\,
U=\Gamma,
A_{N,n},
\widehat\Upsilon_z^{\mathrm{out}}=\widehat G
\right)
=
\P_{\mathcal L_{z,\widehat G},\Lambda,\beta}.
\end{align*}
Furthermore, for every $v\in K$,
\begin{align}
d\left(
v,
\mathcal L_{z,\widehat G}^{\,c}
\right)
&\ge
cn^2.
\label{eq:distance-to-forbidden-contours-local}
\end{align}
Indeed, a contour forbidden only because its length exceeds $n^4$ has spatial diameter at least $c\sqrt{n^4}=cn^2$.
A contour forbidden by the domain or by compatibility with $\widehat G$ must reach $\partial\Lambda$ or the compatibility neighborhood of a fixed outer cylinder.
By \eqref{eq:local-domain-buffer}, these sets are at distance at least $cn^4$ from $K$.
This proves \eqref{eq:distance-to-forbidden-contours-local}.

The same argument, without restricting attention to $K$, shows that, conditionally on the relevant outer cylinder data, the reflected law on each connected component of $B_z$ is a zero-boundary SOS measure with an allowed-contour family of the type appearing in Subsection~\ref{subsec:3.3}.
We shall use this observation below to obtain rough one-point, two-point, and multipoint estimates throughout $B_z$.

\medskip
\noindent
\emph{Step 2: Removal of the long-contour restriction.}
Set
\begin{align*}
\P_{U^+,U^-}
&:=
\P_{\Lambda,\beta}
\left(
\,\cdot\,
\middle|\,
\psi|_{U^+}\ge0,
\psi|_{U^-}\le0
\right).
\end{align*}
Uniformly over the admissible sets $U^+,U^-$, the contour ensemble under $\P_{U^+,U^-}$ is stochastically dominated by the corresponding independent contour process.
Consequently, for $\beta$ sufficiently large,
\begin{align}
\P_{U^+,U^-}
\left(
(\mathcal A_{\Lambda,n}^{\mathrm{short}})^c
\right)
&\le
C|\Lambda|
\sum_{\ell>n^4}3^\ell e^{-\beta\ell}
\nonumber\\
&\le
CM^2e^{-cn^4}
\le
Ce^{-c'n^4}.
\label{eq:remove-long-contours}
\end{align}
Here $|\Lambda|\le M^2$, while $\log M=O(n)$ by \eqref{eq:localization-height-scale} and \eqref{eq:coarse-graining-scale}.
In particular,
\begin{align*}
\P_{U^+,U^-}
\left(
\mathcal A_{\Lambda,n}^{\mathrm{short}}
\right)
&\ge
\frac12
\end{align*}
for all sufficiently small $u$.
Therefore, for every event $A$,
\begin{align}
&
\left|
\P_{U^+,U^-}
\left(
A
\,\middle|\,
\mathcal A_{\Lambda,n}^{\mathrm{short}}
\right)
-
\P_{U^+,U^-}(A)
\right|
\nonumber\\
&\qquad\le
\frac{
\P_{U^+,U^-}
\left(
(\mathcal A_{\Lambda,n}^{\mathrm{short}})^c
\right)
}{
\P_{U^+,U^-}
\left(
\mathcal A_{\Lambda,n}^{\mathrm{short}}
\right)
}
\le
Ce^{-cn^4}.
\label{eq:long-contour-conditioning-error}
\end{align}

\medskip
\noindent
\emph{Step 3: Removal of the boundary compatibility restrictions.}
We first suppose that $A_x$ is increasing.
Define the two extremal boundary-conditioned measures
\begin{align*}
\P_\Lambda^+
&:=
\P_{\Lambda,\beta}
\left(
\,\cdot\,
\middle|\,
\psi|_{\partial^-\Lambda}\ge0
\right),
\\
\P_\Lambda^-
&:=
\P_{\Lambda,\beta}
\left(
\,\cdot\,
\middle|\,
\psi|_{\partial^-\Lambda}\le0
\right).
\end{align*}
By FKG monotonicity,
\begin{align}
\P_\Lambda^-(A_x)
&\le
\P_{U^+,U^-}(A_x)
\le
\P_\Lambda^+(A_x).
\label{eq:FKG-boundary-sandwich}
\end{align}
We claim that
\begin{align}
\left|
\P_\Lambda^\pm(A_x)
-
\P_\beta(A_0)
\right|
&\le
Ce^{-cn^4}.
\label{eq:extremal-boundary-comparison}
\end{align}

We prove the claim for $\P_\Lambda^-$.
Since the event $\{\psi|_{\partial^-\Lambda}\le0\}$ is decreasing,
\begin{align*}
\P_\Lambda^-
&\preceq
\P_{\Lambda,\beta}.
\end{align*}
Hence, using the exponential convergence to the infinite-volume SOS measure and \eqref{eq:local-domain-buffer},
\begin{align}
\P_\Lambda^-(A_x)
&\le
\P_{\Lambda,\beta}(A_x)
\le
\P_\beta(A_0)+Ce^{-cn^4}.
\label{eq:minus-boundary-upper}
\end{align}

For the reverse inequality, let $\Upsilon_\partial^-(\psi)$ be the collection of negative contours adjacent to $\partial\Lambda$, and let $\mathcal E$ be the event that the interior of every boundary-adjacent negative contour stays outside the $n^4/2$-neighborhood of $K$, that is,
\begin{align*}
\overline\gamma
\cap
\left\{
v:
\operatorname{dist}(v,K)\le\frac{n^4}{2}
\right\}
&=
\varnothing
\qquad
\text{for every }
\gamma\in\Upsilon_\partial^-(\psi).
\end{align*} If $\mathcal E^c$ occurs, a boundary-adjacent negative contour has an interior reaching the $n^4/2$-neighborhood of $K$. Since $\operatorname{dist}(K,\partial\Lambda)\ge n^4$, such a contour has diameter at least $n^4/2-O(1)$ and hence length at least $cn^4$. Contour domination and a Peierls estimate therefore yield
\begin{align}
\P_\Lambda^-(\mathcal E^c)
&\le
C|\Lambda|e^{-cn^4}
\le
Ce^{-c'n^4}.
\label{eq:boundary-contour-penetration}
\end{align}
On $\mathcal E$, condition on $\Upsilon_\partial^-(\psi)$ and let $\Lambda'$ be the connected component containing $K$ after deleting the interiors of these boundary contours and the corresponding internal boundary layer.
Then
\begin{align*}
\operatorname{dist}(K,\partial\Lambda')
&\ge
\frac{n^4}{3}.
\end{align*}
By the sign convention for the revealed negative boundary contours, the newly exposed internal boundary lies on the nonnegative side of each deleted contour.
Hence the induced boundary condition on $\Lambda'$ is nonnegative.
The induced boundary condition on $\Lambda'$ is nonnegative.
Thus, by monotonicity, the conditional law in $\Lambda'$ stochastically dominates $\P_{\Lambda',\beta}$.
Exponential convergence in $\Lambda'$ therefore gives
\begin{align}
\P_\Lambda^-(A_x)
&\ge
\E_\Lambda^-
\left[
\mathbf 1_{\mathcal E}
\P_{\Lambda',\beta}(A_x)
\right]
\nonumber\\
&\ge
\left(
1-Ce^{-cn^4}
\right)
\left(
\P_\beta(A_0)-Ce^{-cn^4}
\right)
\nonumber\\
&\ge
\P_\beta(A_0)-Ce^{-cn^4}.
\label{eq:minus-boundary-lower}
\end{align}
This proves \eqref{eq:extremal-boundary-comparison} for $\P_\Lambda^-$.
The proof for $\P_\Lambda^+$ is analogous, using positive boundary contours and reversing the stochastic inequalities.
Combining \eqref{eq:FKG-boundary-sandwich} and \eqref{eq:extremal-boundary-comparison}, we obtain
\begin{align}
\left|
\P_{U^+,U^-}(A_x)
-
\P_\beta(A_0)
\right|
&\le
Ce^{-cn^4}.
\label{eq:comparison-with-sign-constraints}
\end{align}
for every increasing event $A_x$.
The decreasing case follows by taking complements.

Finally, suppose that $A_x=A_x^+\cap A_x^-$, where $A_x^+$ is increasing and $A_x^-$ is decreasing.
Since
\begin{align*}
A_x
&=
A_x^+
\setminus
\left(
A_x^+\cap(A_x^-)^c
\right)
\end{align*}
and both $A_x^+$ and $A_x^+\cap(A_x^-)^c$ are increasing, \eqref{eq:comparison-with-sign-constraints} applied to these two events yields the same comparison for $A_x$.

Combining this comparison with \eqref{eq:exact-reflected-conditional-law} and \eqref{eq:long-contour-conditioning-error}, uniformly over every admissible realization $\widehat G$, we obtain
\begin{align*}
&
\left|
\P_{N,\beta}^n
\left(
A_x
\,\middle|\,
U=\Gamma,
A_{N,n},
\widehat\Upsilon_z^{\mathrm{out}}=\widehat G
\right)
-
\P_\beta(A_0)
\right|
\le
Ce^{-cn^4}.
\end{align*}
Averaging over $\widehat\Upsilon_z^{\mathrm{out}}$ proves \eqref{eq:local-comparison-cell}, since $e^{-cn^4}\le e^{-c'n^2}$ for all sufficiently large $n$.
\end{proof}

\bibliographystyle{plain}
    \bibliography{iso}

\end{document}